\documentclass[11pt, a4paper]{amsart}
\usepackage[utf8]{inputenc}
\usepackage[margin=0.9in]{geometry}

\usepackage{comment}
\usepackage{lmodern}
\usepackage{upgreek}
\usepackage{color}
\usepackage{amsmath}
\usepackage{amssymb}
\usepackage{amsthm}

\usepackage{pdfpages}

\usepackage{fancyhdr}

\usepackage{setspace}

\usepackage{hyperref}
\hypersetup{
	colorlinks=true,
	linkcolor=blue,
	filecolor=blue,      
	urlcolor=blue,
	citecolor=blue,
	pdftitle={Products of prime ideals in  ray class groups},
	pdfpagemode=FullScreen,
}

\numberwithin{equation}{section}
\newtheorem{theorem}{Theorem}[section]
\newtheorem{lemma}[theorem]{Lemma}
\newtheorem{proposition}[theorem]{Proposition}
\newtheorem{corollary}[theorem]{Corollary}

\theoremstyle{remark}
\newtheorem{remark}[theorem]{Remark}

\newtheorem*{hypothesis*}{Hypothesis}

\theoremstyle{definition}

\newtheorem*{prob*}{Problem}

\title[Products of prime ideals in  ray class groups]{Products of prime ideals in  ray class groups}

\author{Likun Xie}
\address{Max-Planck-Institut für Mathematik
	Vivatsgasse 7, 53111, Bonn, Germany} 
\email{xie@mpim-bonn.mpg.de}



\begin{document}
\begin{abstract}

	We prove that every class in the narrow ray class
	group  modulo an integral ideal \(\mathfrak q\) of a fixed number field is represented by a product of
	three prime ideals of norm at most
	$
( N\mathfrak q)^{\max(1,3\alpha,4\alpha_0)+\kappa} 
    $ for any $\kappa>0$, 
		where \(\alpha\) is the  exponent in short character sum bounds for general non-principal
	ray class characters and \(\alpha_0\) comes from a bounded-order
	subconvexity input for Hecke \(L\)-functions. 
	Wu's subconvexity bound gives the
	admissible choice \(\alpha=\alpha_0=103/256\), hence the  explicit  
	bound \((N\mathfrak q)^{103/64+\kappa}\).
	This improves the previous
	\(O_K((N\mathfrak q)^3)\)-scale bound of
	Deshouillers, Gun, Ramaré, and Sivaraman.	
	We also prove that a positive proportion of ray classes are represented by products of two prime ideals.
	The proof extends the multiplicative dense-model and transference framework of
Matom\"aki--Ter\"av\"ainen to narrow ray class groups.

\end{abstract}	
	\maketitle	
\section{Introduction}
A conjecture of Erd\H{o}s states that, for every sufficiently large prime
modulus \(q\), every reduced residue class modulo \(q\) is represented by a
product \(p_1p_2\) of two primes with \(p_1,p_2\le q\); see~\cite[\S 2]{erdos}.
Matom\"aki and Ter\"av\"ainen~\cite{matomaki} proved a ternary version of this
problem. They showed that, for cubefree \(q\), every reduced residue class
modulo \(q\) is represented by a product of three primes, each at most \(q\),
and that for general \(q\) the same is true with the bound \(q^{1+\varepsilon}\).
They also obtained lower bounds for the proportion of classes represented by
products of two primes. Their main tool is a multiplicative dense-model theorem
for primes, which transfers the problem to a product-set problem in a finite
abelian group, up to certain coset obstructions. This approach is rooted in the
dense-model transference method introduced by Green~\cite{Green} and further
developed by Green and Tao~\cite{GreenTao}.

There are also function-field analogues of Erd\H{o}s's conjecture, where
geometric methods become available; see for instance the work of Sawin
\cite{sawin}, and also  \cite{xie}. 
This paper considers the
number-field setting, where the approach is closer in spirit to the integer
case. Our purpose is to extend the dense-model and transference framework of
Matom\"aki and Ter\"av\"ainen \cite{matomaki} to narrow ray class groups over
number fields.
Let \(K\) be a fixed number field with ring of integers \(\mathcal O_K\). Let
\(\mathfrak q\subseteq\mathcal O_K\) be a non-zero integral ideal, put
$
Q:=N\mathfrak q,
$
and let
$
G:=\operatorname{Cl}^{(\infty)}_{\mathfrak q}
$
be the narrow ray class group modulo \(\mathfrak q\), where all real places
are included in the modulus. For \(X\ge2\), define
\[
P_X(\mathfrak q)
:=
\bigl\{[\mathfrak p]\in G:\mathfrak p\nmid\mathfrak q,\ N\mathfrak p\le X\bigr\},
\]
and
\[
E_k(X;\mathfrak q)
:=
\underbrace{
	P_X(\mathfrak q)\cdot \ldots \cdot P_X(\mathfrak q)
}_{k\ \mathrm{times}}
\subseteq G.
\]
Thus \(c\in E_k(X;\mathfrak q)\) precisely when \(c\) is represented by a
product of \(k\) prime ideal classes, with each prime ideal having norm at most
\(X\).

The bounds obtained below depend on two analytic inputs. Let
\[
0<\alpha_0\le \alpha<1.
\]
The first input is a short character-sum estimate for arbitrary non-principal
ray class characters. For every \(\varepsilon>0\) there exists
\(\eta=\eta(K,\varepsilon)>0\) such that  for every finite order non-principal   Hecke
character \(\chi\) with module of definition \(\mathfrak q\),
\begin{equation}
	\sum_{N\mathfrak a\le T}\chi(\mathfrak a)
	\ll_{K,\varepsilon} T^{1-\eta}
	\qquad
	(T\ge Q^{\alpha+\varepsilon}).
	\tag{\(\mathrm{CS}(\alpha)\)}
\end{equation}
The second input is a bounded-order subconvexity estimate. For every fixed
integer \(\ell\ge2\) and every \(\varepsilon>0\), every primitive
non-principal Hecke character \(\chi^*\) of order at most \(\ell\) satisfies
\begin{equation}
	L\Bigl(\frac12+it,\chi^*\Bigr)
	\ll_{K,\ell,\varepsilon}
	C(\chi^*,t)^{\alpha_0/2+\varepsilon}
	\qquad (t\in\mathbb R),
	\tag{\(\mathrm L^{\mathrm b}(\alpha_0)\)}
\end{equation}
where \(C(\chi^*,t)\) denotes the analytic conductor. The first exponent \(\alpha\) controls the general dense-model
estimates, while \(\alpha_0\) controls only the bounded-order obstruction
estimates. 

In particular, \(\mathrm L^{\mathrm b}(\alpha_0)\) implies   the following  short character-sum estimate for  
non-principal bounded order ray class characters: for every fixed
integer \(\ell\ge2\) and every \(\varepsilon>0\),  there exists
\(\eta=\eta( K, \ell,\varepsilon)>0\) such that  every non-principal Hecke
character \(\chi\) of order  at most \(\ell\) with module of definition \(\mathfrak q\) satisfies
\begin{equation} 
	\sum_{N\mathfrak a\le T}\chi(\mathfrak a)
	\ll_{K,\ell,\varepsilon}
	T^{1-\eta}
	\qquad
	(T\ge Q^{\alpha_0+\varepsilon}).
	\tag{\(\mathrm{CS}^{\mathrm b}(\alpha_0)\)}
\end{equation}
The exponent \(\alpha\) and $\alpha_0$ control the general dense-model estimates. 
Our main result is the following ternary representation theorem for number fields. 
 \begin{theorem}
 	\label{intro:three-prime-ray}
 	Assume \(\mathrm{CS}(\alpha)\) and \(\mathrm L^{\mathrm b}(\alpha_0)\). Let \(\kappa>0\) be fixed. For $Q$ sufficiently large, if 
 	$
 	X\ge Q^{\max(1,3\alpha,4\alpha_0)+\kappa},
 	$
 	then
 	$
 	E_3(X;\mathfrak q)=G.
 	$
 \end{theorem}
 
 We also prove the following two-prime density result. 
 
 \begin{theorem} 
 	\label{intro:two-prime-ray}
 	Assume \(\mathrm{CS}(\alpha)\) and \(\mathrm L^{\mathrm b}(\alpha_0)\). Let \(\kappa>0\) be fixed, and let
 	\(\varepsilon>0\) be sufficiently small. For $Q$ sufficiently large in terms of $\varepsilon$, the following hold. 
 	\begin{enumerate}
 		\item[(i)] If
 		$
 		X\ge Q^{\max(1,3\alpha,4\alpha_0)+\kappa},
 		$
 		then
 		$
 		|E_2(X;\mathfrak q)|
 		\ge
 		\bigl(\frac23-\varepsilon\bigr)|G|.
 		$
 		
 		\item[(ii)] If
 		$
 		X\ge Q^{\max(1,4\alpha,4\alpha_0)+\kappa},
 		$
 		then
 		$
 		|E_2(X;\mathfrak q)|
 		\ge
 		\bigl(\frac{11}{16}-\varepsilon\bigr)|G|.
 		$
 	\end{enumerate}
 \end{theorem}

For a general fixed number field \(K\), Wu's subconvexity theorem for Hecke
\(L\)-functions~\cite{wu}, together with the admissible Ramanujan exponent
\(\theta=7/64\) of Blomer--Brumley~\cite{ramanujan}, gives
\[
\alpha=\alpha_0=\frac{103}{256};
\]
see Proposition~\ref{prop_subconvexity}. Hence Theorem~\ref{intro:three-prime-ray}
gives
\[
E_3(X;\mathfrak q)=G
\qquad
\text{whenever}
\qquad
X\ge Q^{103/64+\kappa},\quad 103/64\approx 1.609.
\]
This improves the exponent in the   three-prime result of
Deshouillers--Gun--Ramar\'e--Sivaraman~\cite{ideal_product}, where the scale is
\(O_K(Q^3)\). They  show  that every narrow ray class modulo
\(\mathfrak q\) can be represented by a product of three degree-one unramified prime
ideals of norm \(\ll_K Q^3\). We state our theorem for arbitrary prime ideals.   The restriction
to unramified degree-one prime ideals may also be imposed by removing the
negligible contribution of prime ideals of higher residue degree from the
prime-counting arguments, since the number of prime ideals of residue degree \(>1\) and norm at most \(T\) is \(O_K(T^{1/2})\).

The separation between \(\alpha\) and \(\alpha_0\) is useful when stronger
bounds are available for bounded-order characters. In particular, the Weyl-type subconvexity results of Balkanova, Frolenkov, and Wu for cube-free conductor \cite{cubefree_totally_real} suggest the possibility of a bounded-order input \(\mathrm L^{\mathrm b}(1/3)\). Indeed, if a Hecke character has order
dividing a fixed integer \(\ell\), then its conductor is cube-free away
from the finitely many primes lying above the rational primes dividing
\(\ell\).
Thus one may ask whether  a
variant of their method yields
\(\mathrm L^{\mathrm b}(1/3)\). We do not pursue this refinement here.
 Assuming such a bounded-order input
 \(\mathrm L^{\mathrm b}(1/3)\), while using the general character-sum
exponent \(\alpha=103/256\)  from Wu's subconvexity bound, Theorems~\ref{intro:three-prime-ray} and
\ref{intro:two-prime-ray}{\rm (i)} would hold at the scale
$
X\ge Q^{4/3+\kappa}.
$
While  Theorem~\ref{intro:two-prime-ray}{\rm (ii)}  depends on
the term \(4\alpha\), and therefore improving its scale would require an
improvement in the general character-sum exponent \(\alpha\).

 We now summarize the proof. The first step is to construct a dense model for
 the prime ideal classes in \(G\). This uses the linear sieve over ideals,
 mean-value estimates in ray classes, and the character-sum input
 \(\mathrm{CS}(\alpha)\). The dense model reduces the representation problem to
 product-set estimates for dense subsets of the finite abelian group \(G\).
 
 The second step is finite-group theoretic. Kneser's theorem shows that either
 the relevant product sets are already large enough, or the dense model is
 concentrated in a small number of cosets of a subgroup \(H\le G\). The coset
 obstructions are then treated analytically. The obstructions of indices \(5\)
 and \(8\) are ruled out by weighted prime-sum estimates for bounded-order
 characters, which is where \(\mathrm L^{\mathrm b}(\alpha_0)\) enters. The index
 \(2\) obstruction is the exceptional quadratic case and requires a separate
 ideal-theoretic version of the argument of Matom\"aki--Ter\"av\"ainen. The proof
 of the binary theorem uses the same dense model and Kneser analysis, together
 with a multiplicative energy estimate for products of prime ideals in ray
 classes.
 
 The paper is organized as follows. Section~2 collects the linear sieve,
 ray-class mean-value estimates, and the subconvexity-to-character-sums input.
 Section~3 proves the dense model theorem over ray class groups and derives the
 main transference criteria. Section~4 contains the finite abelian group
 product-set arguments. Section~5 proves the prime-escape estimates for quotient
 indices \(5\) and \(8\). Section~6 proves the ternary Theorem \ref{intro:three-prime-ray}, Section~7 treats
 the exceptional quadratic case, and Section~8 proves the binary density Theorem \ref{intro:two-prime-ray}.
 
\subsection{Notations}
Let \(K\) be a number field of degree
$
n_K := [K:\mathbb Q] ,
$ with ring of integers \(\mathcal O_K\).
Denote by \(\zeta_K(s)\) the Dedekind zeta function of \(K\), and write
$
\rho_K := \operatorname*{Res}_{s=1}\zeta_K(s)
$
for its residue at \(s=1\).
For a nonzero integral ideal \(\mathfrak q\subset \mathcal O_K\), we write
\(\mathrm N\mathfrak q=[\mathcal O_K:\mathfrak q]\) for its absolute norm.

Let \(J(\mathfrak q)\) denote the group of fractional ideals of \(K\)
coprime to \(\mathfrak q\), and let $P ^{+}{(\mathfrak q)} $ denote the subgroup of \(J(\mathfrak q)\) consisting of principal fractional  ideals $ (\alpha)$, \(\alpha\in K^\times, \alpha\equiv 1 \pmod{\mathfrak q},\) and \(\sigma(\alpha)>0\) for every real embedding \(\sigma:K\hookrightarrow
\mathbb R\).  
 
We write
\[
G  :=\operatorname{Cl}^{(\infty)}_{\mathfrak q}
=
J(\mathfrak q)/P ^{+}{(\mathfrak q)}
\]
for the narrow ray class group modulo \(\mathfrak q\). 

Throughout this paper, by a Hecke character modulo \(\mathfrak q\) we mean
a finite-order Hecke character  for which
\(\mathfrak q\) is a module of definition in the terminology of Neukirch \cite[Def.~VII.6.11]{neukirch}.  Equivalently, this is a character of
the narrow ray class group \(G\); or a Dirichlet character modulo \(\mathfrak q\) in the sense of \cite[Def.~VII.6.8]{neukirch}; see also \cite[Sec.~2.3]{notes}.

We denote by \(\widehat G\) the character group of \(G\). If \(\chi\in \widehat G\), we also regard \(\chi\) as a function on integral ideals by setting
\[
\chi(\mathfrak a)
=
\begin{cases}
	\chi([\mathfrak a]), & \text{if }(\mathfrak a,\mathfrak q)=1,\\
	0, & \text{otherwise},
\end{cases}
\]
where \([\mathfrak a]\) denotes the class of \(\mathfrak a\) in
\(G \). We shall use the usual orthogonality relations for
the finite abelian group \(G \). In particular, if
\((\mathfrak a\mathfrak b,\mathfrak q)=1\), then
\[
\frac{1}{|G |}
\sum_{\chi\in \widehat{G }}
\chi(\mathfrak a)\overline{\chi(\mathfrak b)}
=
\mathbf 1_{[\mathfrak a]=[\mathfrak b]  }.
\]

	\section{Preliminary results}
	\subsection{Linear sieve in number fields}
We record a form of the fundamental lemma of the linear sieve for ideals. It follows from Coleman's extension of the Rosser--Iwaniec sieve to number
fields \cite{sieve}, together with the classical construction of the linear
sieve weights \cite[Ch.~3]{opera}.

Fix once and for all a total order \(\le_K\) on the non-zero integral ideals
of \(K\), such that
\[
N\mathfrak a<N\mathfrak b \quad\Longrightarrow\quad
\mathfrak a<_K\mathfrak b.
\]
For \(z\ge2\), put
\[
\mathcal P(z):=\prod_{N\mathfrak p<z}\mathfrak p.
\]

\begin{lemma}\label{sieve_lemma}
	Let $z\ge2$ and $D=z^s$ with $s\ge1$.
	There exist coefficients \(\lambda_{\mathfrak d}^{\pm}=\lambda_{\mathfrak d}^{\pm}(D,z)\), supported on
	squarefree ideals \(\mathfrak d\mid \mathcal P(z)\) with
	\(N\mathfrak d\le D\), such that
	$
	|\lambda_{\mathfrak d}^{\pm}|\le1
	$
	and the following hold.
	
  For every integral ideal \(\mathfrak n\),
	\[
	\sum_{\mathfrak d\mid\mathfrak n}\lambda_{\mathfrak d}^{-}
	\le
	1_{(\mathfrak n,\mathcal P(z))=1}
	\le
	\sum_{\mathfrak d\mid\mathfrak n}\lambda_{\mathfrak d}^{+}.
	\]
	
	Let \(h\) be a multiplicative function on ideals satisfying
	$
	0\le h(\mathfrak p)<1
	$
	for every prime ideal \(\mathfrak p\). Assume that \(h\) has sieve
	dimension \(1\), in the following sense: there exists a constant
	\(C_0\ge1\) such that, for all ideals \(\mathfrak h,\mathfrak j\) with
	\(\mathfrak h<_K\mathfrak j\) and \(N\mathfrak h\ge2\),
	\[
	\prod_{\mathfrak h\le_K\mathfrak p<_K\mathfrak j}
	\bigl(1-h(\mathfrak p)\bigr)^{-1}
	\le
	\frac{\log N\mathfrak j}{\log N\mathfrak h}
	\Bigl(1+\frac{C_0}{\log N\mathfrak h}\Bigr).
	\]

	If \(1\le s\le3\), then
	\[
	\sum_{\mathfrak d\mid\mathcal P(z)}
	\lambda_{\mathfrak d}^{+}h(\mathfrak d)
	\le
	\bigl(F_0(s)+o_K(1)\bigr)
	\prod_{N\mathfrak p<z}\bigl(1-h(\mathfrak p)\bigr),
	\]
	where
	\[
	F_0(s)=\frac{2e^\gamma}{s}.
	\]
	
	If \(2\le s\le4\), then
	\[
	\sum_{\mathfrak d\mid\mathcal P(z)}
	\lambda_{\mathfrak d}^{-}h(\mathfrak d)
	\ge
	\bigl(f_0(s)+o_K(1)\bigr)
	\prod_{N\mathfrak p<z}\bigl(1-h(\mathfrak p)\bigr),
	\]
	where
	\[
	f_0(s)=\frac{2e^\gamma\log(s-1)}{s}.
	\]
Here the \(o_K(1)\)-terms are as \(z\to\infty\).
\end{lemma}

\subsection{Character sums from subconvexity}
\label{sec:HB}

We record a standard consequence of conductor-aspect subconvexity for
Hecke \(L\)-functions. This will provide a concrete admissible value of the
character-sum exponent \(\alpha\) in \(\mathrm{CS}(\alpha)\).

We use the following subconvexity theorem of Wu.
\begin{theorem}[\cite{wu}, Thm.~1.1]
	\label{wu}
	Let \(\chi\) be a Hecke character of \(K\), with analytic conductor
	\(C(\chi)\). Then
	\[
	\biggl|L \Bigl(\frac12,\chi\Bigr)\biggr|
	\ll_{K,\varepsilon}
	C(\chi)^{
		\frac14-\frac{1-2\theta}{16}+\varepsilon},
	\]
	where \(\theta\) is any admissible exponent toward the
	Ramanujan--Petersson conjecture for \(\mathrm{GL}_2\) over \(K\).
\end{theorem}
\begin{remark}
	The best currently known value is \(\theta=7/64\), due to
	Blomer--Brumley \cite[Thm.~1]{ramanujan}.
\end{remark}
 
\begin{proposition}\label{prop_subconvexity}
	Let \(K\) be a fixed number field,  and let \(\mathfrak q\subseteq\mathcal O_K\)
	be a non-zero integral ideal. Let \(\chi\) be a non-principal
	Hecke character modulo \(\mathfrak q\).
	Then, for every \(\varepsilon>0\), there exists
	\(\delta=\delta(K,\varepsilon)>0\) such that
	\[
	\sum_{N\mathfrak a\le X}\chi(\mathfrak a)
	\ll_{K,\varepsilon}
	X^{1-\delta}
	\]
	whenever
	\[
	X  \ge  (N\mathfrak{q})^{ \frac12 - \frac{1-2\theta}{8} + \varepsilon}.
	\]
\end{proposition}
\begin{remark}
	In particular, \(\mathrm{CS}(\alpha)\) holds with
	$
	\alpha=\frac12-\frac{1-2\theta}{8}.
	$
	Using \(\theta=7/64\), this gives
	$
	\alpha=\frac{103}{256}.
	$
	Since \(\mathrm{CS}(\alpha)\) applies to all non-principal Hecke characters,
	we may also take
	$
	\alpha_0=\alpha=\frac{103}{256}
	$
	in \(\mathrm{CS}^{\mathrm b}(\alpha_0)\).
\end{remark}

\begin{proof}
	Put
	\[
	\sigma:=\frac{1-2\theta}{16},
	\qquad
	\rho:=\frac12-2\sigma
	=
	\frac12-\frac{1-2\theta}{8}.
	\]

	Let \(\chi^*\) be the primitive character inducing \(\chi\), and let
	\(\mathfrak f\mid\mathfrak q\) be the conductor of \(\chi^*\). Since
	\(\chi\) is extended by zero on ideals not coprime to \(\mathfrak q\), its
	Dirichlet series is
	\[
	D(s,\chi)
	:=
	\sum_{\mathfrak a}\frac{\chi(\mathfrak a)}{(N\mathfrak a)^s}
	=
	L(s,\chi^*)
	\prod_{{\mathfrak p\mid\mathfrak q\atop
			\mathfrak p\nmid\mathfrak f}}
	\Bigl(1-\frac{\chi^*(\mathfrak p)}{(N\mathfrak p)^s}\Bigr).
	\]
	The finite Euler product satisfies, on \(\Re s=1/2\),
	\[ \biggl|\prod_{{\mathfrak p\mid\mathfrak q\atop
			\mathfrak p\nmid\mathfrak f}}
	\Bigl(1-\frac{\chi^*(\mathfrak p)}{(N\mathfrak p)^s}\Bigr)\biggr|  \le
	\prod_{\mathfrak p\mid\mathfrak q}
	\Bigl(1+\frac1{(N\mathfrak p)^{1/2}}\Bigr)
	\ll_{K,\varepsilon} Q^\varepsilon.
	\]
	Hence, after replacing \(\varepsilon\) by a smaller value if necessary,
	Theorem \ref{wu} applied to the twist
	$
	\chi^*_t:=\chi^*|\cdot|_{\mathbb A}^{it}
	$
	gives
	\begin{equation}\label{subconvex-t}
		D \Bigl(\frac12+it,\chi\Bigr)
		\ll_{K,\varepsilon}
		Q^{1/4-\sigma+\varepsilon}(1+|t|)^{B_K},
	\end{equation}
	for some constant \(B_K>0\) depending only on \(K\). Here we have used the
	standard conductor estimate
	\[
	C(\chi^*_t)\ll_K N\mathfrak f\,(1+|t|)^{B_K}
	\le Q(1+|t|)^{B_K}.
	\]
	
	Let \(0<\eta<1/10\), and choose \(W_\eta\in C_c^\infty([0,\infty))\) such
	that
	\[
	0\le W_\eta\le1,\qquad
	W_\eta(y)=1\quad(0\le y\le1),\qquad
	W_\eta(y)=0\quad(y\ge1+\eta),
	\]
	and
	\[
	\biggl|\Bigl(y\frac{d}{dy}\Bigr)^j W_\eta(y)\biggr|
	\ll_j \eta^{-j}.
	\]
	Let
	\[
	\widehat W_\eta(s):=\int_0^\infty W_\eta(y)y^{s-1}\,dy.
	\]
	Repeated integration by parts gives, for every \(A\ge1\),
	\begin{equation}\label{Mellin-decay}
		\widehat W_\eta(\sigma+iT)
	\ll_{A,c}\eta^{-A}(1+T)^{-A}
	\qquad
	\Bigl(\frac12\le\sigma\le c\Bigr),
	\end{equation}
	
	Define the smoothed sum
	\[
	S_\eta(X):=
	\sum_{\mathfrak a}\chi(\mathfrak a)
	W_\eta \Bigl(\frac{N\mathfrak a}{X}\Bigr).
	\]
	Since \(W_\eta=1\) on \([0,1]\) and is supported in \([0,1+\eta]\), the
	ideal-counting estimate \cite[Thm.~VI.3.3]{lang}
	\begin{equation}\label{ideal_counting}
		\#\{\mathfrak a:N\mathfrak a\le T\}
		=
		\rho_K T+O_K\bigl(T^{1-1/n_K}\bigr),
	\end{equation}
	gives
	\begin{equation}\label{sharp-vs-smooth}
		\sum_{N\mathfrak a\le X}\chi(\mathfrak a)
		=
		S_\eta(X)
		+
		O_K \bigl(\eta X+X^{1-1/n_K} \bigr),
	\end{equation}
	where $\rho_K =\operatorname*{Res}_{s=1}\zeta_K(s)$ and $ n_K=[K:\mathbb Q]$.
	
	By Mellin inversion,
	\[
	S_\eta(X)
	=
	\frac{1}{2\pi i}
	\int_{(c)}
	D(s,\chi)\widehat W_\eta(s)X^s\,ds
	\qquad(c>1).
	\]
	Since \(\chi^*\) is non-principal, \(L(s,\chi^*)\) hence \(D(s,\chi )\) is entire \cite[Sect.~6]{hecke} and of finite order by the 	functional equation.
	
Moreover, by the Phragmen--Lindelöf principle for Hecke
	\(L\)-functions, \(D(s,\chi)\) has polynomial growth in every fixed vertical
	strip; see \eqref{bound1} in the proof of Lemma \ref{lem:weighted-prime-ideal-sums-rescaled}. Thus, for fixed \(c>1\), there is \(C_K>0\) such that 
	\[
	D(\sigma+it,\chi)\ll_{K,c,\varepsilon}
	Q^{O(1)}(1+|t|)^{C_K}
	\qquad
	\Bigl(\frac12\le \sigma\le c\Bigr).
	\]
	
	For \(T\ge1\), integrate over the rectangle with vertical sides
	\(\Re s=c\) and \(\Re s=1/2\), and horizontal sides
	\(\Im s=\pm T\). No pole is crossed. On the horizontal sides, using \eqref{Mellin-decay},
	we get
	\[
	\int_{1/2}^{c}
	D(\sigma+iT,\chi)\widehat W_\eta(\sigma+iT)X^{\sigma+iT}\,d\sigma
	\ll_{K,c,A,\varepsilon}
	Q^{O(1)}X^c\eta^{-A}(1+T)^{C_K-A}.
	\]
	Choosing \(A>C_K+1\), this tends to \(0\) as \(T\to\infty\), and the same
	argument applies to the lower horizontal side. Hence the contour may be
	shifted to \(\Re s=1/2\),  using \eqref{subconvex-t} and
	\eqref{Mellin-decay}, and choosing \(A>\max\{C_K+1,B_K+2\} \), we obtain
	\[
	\begin{aligned}
			S_\eta(X)
		&=
		\frac1{2\pi}
		\int_{-\infty}^{\infty}
		D\Bigl(\frac12+it,\chi\Bigr)
		\widehat W_\eta\Bigl(\frac12+it\Bigr)
		X^{1/2+it}\,dt\\
		&\ll_{K,\varepsilon,A}
		X^{1/2}Q^{1/4-\sigma+\varepsilon}\eta^{-A}
		\int_{-\infty}^{\infty}
		(1+|t|)^{B_K-A}\,dt \\
		&\ll_{K,\varepsilon,A}
		X^{1/2}Q^{1/4-\sigma+\varepsilon}\eta^{-A}.
	\end{aligned}
	\]
	Therefore,
	\begin{equation}\label{basic-bound}
		\sum_{N\mathfrak a\le X}\chi(\mathfrak a)
		\ll_{K,\varepsilon,A}
		X^{1/2}Q^{1/4-\sigma+\varepsilon}\eta^{-A}
		+
		\eta X
		+
		X^{1-1/n_K}.
	\end{equation}
	
	Now suppose
	$
	X\ge Q^{\rho+\varepsilon_0}
	$
	for some fixed \(\varepsilon_0>0\). Choose  
	\(\varepsilon>0\) so small that \(\varepsilon<\varepsilon_0/2\). Since
	$
	\frac14-\sigma=\frac{\rho}{2},
	$
	we have
	\[
	X^{1/2}Q^{1/4-\sigma+\varepsilon}
	\le
	X^{
		1/2+\frac{\rho/2+\varepsilon}{\rho+\varepsilon_0}}
	=
	X^{1-c_0}
	\]
	for some \(c_0=c_0(\varepsilon_0)>0\).
	
	Choose
	$
	\eta=X^{-\lambda}
	$
	with \(0<\lambda<c_0/(2A)\). Then
	\[
	X^{1/2}Q^{1/4-\sigma+\varepsilon}\eta^{-A}
	\ll
	X^{1-c_0/2},
	\qquad
	\eta X=X^{1-\lambda}.
	\]
	Substituting into \eqref{basic-bound}, we get
	\[
	\sum_{N\mathfrak a\le X}\chi(\mathfrak a)
	\ll_{K,\varepsilon_0}
	X^{1-c_0/2}
	+
	X^{1-\lambda}
	+
	X^{1-1/n_K}.
	\]
	Thus
	\[
	\sum_{N\mathfrak a\le X}\chi(\mathfrak a)
	\ll_{K,\varepsilon_0}
	X^{1-\delta},
	\]
	where
	$
	\delta:=
	\min\Bigl\{\frac{c_0}{2},\lambda,\frac1{n_K}\Bigr\}>0.
	$
	This proves the proposition.
\end{proof}

\subsection{Mean values of character sums}
\begin{lemma}
	\label{lem:ray-mean-value}
	Let \(K\) be a number field, let \(\mathfrak q\subseteq \mathcal O_K\)
	be a non-zero integral ideal, and put
\(G=\operatorname{Cl}_{\mathfrak q}^{(\infty)}.\)
	Let \(X\ge 2\), and let \(a_{\mathfrak a}\) be complex coefficients supported
	on integral ideals \(\mathfrak a\) satisfying \(	(\mathfrak a,\mathfrak q)=1 \) and \(	N\mathfrak a\le X. \)
	
	Define
	\[
	M_K(X;\mathfrak q)
	:=
	\max_{c\in G}
	\#\bigl\{
	\mathfrak a\subseteq\mathcal O_K:
	(\mathfrak a,\mathfrak q)=1,\ 
	N\mathfrak a\le X,\ 
	[\mathfrak a]=c
	\bigr\}.
	\]
	Then
	\begin{equation}\label{mean_value}
			\sum_{\chi\in\widehat G}
			\Biggl|
		\sum_{ {N\mathfrak a\le X\atop
				(\mathfrak a,\mathfrak q)=1}}
		a_{\mathfrak a}\chi(\mathfrak a)
			\Biggr|^2
		\le
		|G|\,M_K(X;\mathfrak q)
		\sum_{ {N\mathfrak a\le X\atop
				(\mathfrak a,\mathfrak q)=1}}
		|a_{\mathfrak a}|^2 .
	\end{equation}
\end{lemma}
\begin{proof}
	Squaring out and using the orthogonality of the characters of \(G\), we obtain
	\[
	 \sum_{\chi\in\widehat G}
	\Biggl|
	\sum_{ {N\mathfrak a\le X\atop
			(\mathfrak a,\mathfrak q)=1}}
	a_{\mathfrak a}\chi(\mathfrak a)
	\Biggr|^2
	=
	\sum_{\chi\in\widehat G}
	\sum_{ {N\mathfrak a,N\mathfrak b\le X\atop
			(\mathfrak a\mathfrak b,\mathfrak q)=1}}
	a_{\mathfrak a}\overline{a_{\mathfrak b}}
	\chi(\mathfrak a)\overline{\chi(\mathfrak b)}
	=
	|G|
	\sum_{N\mathfrak a,N\mathfrak b\le X
		\atop
		(\mathfrak a\mathfrak b,\mathfrak q)=1,\ [\mathfrak a]=[\mathfrak b]}
	a_{\mathfrak a}\overline{a_{\mathfrak b}}.
	\]
	Using
	$
	|a_{\mathfrak a}a_{\mathfrak b}|
	\le
	\frac{|a_{\mathfrak a}|^2+|a_{\mathfrak b}|^2}{2},
	$
	we get
	\[
	\begin{aligned}
	 \sum_{\chi\in\widehat G}
			\Biggl|
		\sum_{ {N\mathfrak a\le X\atop
				(\mathfrak a,\mathfrak q)=1}}
		a_{\mathfrak a}\chi(\mathfrak a)
			\Biggr|^2
		&  \le
		|G|
		\sum_{ {N\mathfrak a\le X\atop
				(\mathfrak a,\mathfrak q)=1}}
		|a_{\mathfrak a}|^2
		\#
		\bigl\{
		\mathfrak b:
		(\mathfrak b,\mathfrak q)=1,\
		N\mathfrak b\le X,\
		[\mathfrak b]=[\mathfrak a]
		\bigr\}
		\\
		&  \le
		|G|\,M_K(X;\mathfrak q)
		\sum_{ {N\mathfrak a\le X\atop
				(\mathfrak a,\mathfrak q)=1}}
		|a_{\mathfrak a}|^2.
	\end{aligned}
	\]
	This proves the lemma.
\end{proof}

\begin{lemma}[Mean value theorem for ray class characters]\label{lem:ray-class-mean-value-sharp}
		Let \(K\) be a number field, let \(\mathfrak q\subseteq \mathcal O_K\)
	be a non-zero integral ideal, and put
	\(G=\operatorname{Cl}_{\mathfrak q}^{(\infty)}.\)
	Let \(X\ge 2\), and let \(a_{\mathfrak a}\) be complex coefficients supported
	on integral ideals \(\mathfrak a\) satisfying \(	(\mathfrak a,\mathfrak q)=1 \) and \(	N\mathfrak a\le X. \) Then, we have
	\[
	\sum_{\chi\in\widehat G}
	\Biggl|
	\sum_{ {N\mathfrak a\le X\atop
			(\mathfrak a,\mathfrak q)=1}}
	a_{\mathfrak a}\chi(\mathfrak a)
	\Biggr|^2
	\ll_K
	(X+Q)(\log(3Q))^{n_K}
	\sum_{ {N\mathfrak a\le X\atop
			(\mathfrak a,\mathfrak q)=1}}
	|a_{\mathfrak a}|^2.
	\]
	In particular, for every \(\varepsilon>0\),
	\[
	\sum_{\chi\in\widehat G}
	\Biggl|
	\sum_{ {N\mathfrak a\le X\atop
			(\mathfrak a,\mathfrak q)=1}}
	a_{\mathfrak a}\chi(\mathfrak a)
	\Biggr|^2
	\ll_{K,\varepsilon}
	Q^\varepsilon(X+Q)
	\sum_{ {N\mathfrak a\le X\atop (\mathfrak a,\mathfrak q)=1}} |a_{\mathfrak a}|^2 .
	\]
\end{lemma}
\begin{proof}
	By Lemma~\ref{lem:ray-mean-value}, it is enough to bound
	\(|G|M_K(X;\mathfrak q)\).
	
	By \cite[Thm.~1]{counting_ideal}, we have the following uniform bound for
	\(M_K(X;\mathfrak q)\):
	\[
	M_K(X;\mathfrak q)
	\ll_K
	F(\mathfrak q)(\log(3F(\mathfrak q)))^{n_K}
	\Bigl(1+\frac{X}{Q}\Bigr),
	\]
	where
	\[
	F(\mathfrak q)
	=
	\frac{2^{r_1}\varphi_K(\mathfrak q)h_K}
	{h_{K,\mathfrak q}},
	\qquad
	h_{K,\mathfrak q}
	=
	|G|,
	\]
	\(n_K=r_1+2r_2=[K:\mathbb Q]\), and \(h_K\) is the class number of \(K\).
	
	Since \(h_{K,\mathfrak q}=|G|\), we have
	\(|G|F(\mathfrak q)
	=
	2^{r_1}h_K\varphi_K(\mathfrak q)
	\ll_K Q.\)
	In particular,  \(F(\mathfrak q)\ll_K Q\),   hence
	\(\log(3F(\mathfrak q))\ll_K \log(3Q).\)
	It follows that
	\[
	|G|M_K(X;\mathfrak q)
	\ll_K
	(X+Q)(\log(3Q))^{n_K}
	\ll_{K,\varepsilon}
	Q^\varepsilon(X+Q).
	\]
	Substituting this into Lemma~\ref{lem:ray-mean-value} gives the
	stated   bounds.
\end{proof}

\begin{lemma}
	\label{lem:ideal-counting-prime-to-q}
	Let \(Q=N\mathfrak q\). For \(T\ge1\),
	\[
	\#\{\mathfrak a:N\mathfrak a\le T,\;(\mathfrak a,\mathfrak q)=1\}
	=
	\rho_K\frac{\varphi_K(\mathfrak q)}{N\mathfrak q}T
	+
	O_{K,\lambda}\Bigl(Q^\lambda T^{1-1/n_K}\Bigr)
	\]
	for every fixed \(\lambda>0\). In particular, uniformly for
	\(T\ge Q^\sigma\), with \(\sigma>0\) fixed,
	\[
	\#\{\mathfrak a:N\mathfrak a\le T,\;(\mathfrak a,\mathfrak q)=1\}
	=
	\rho_K\frac{\varphi_K(\mathfrak q)}{N\mathfrak q}T
	\bigl(1+o_{K,\sigma}(1)\bigr).
	\]
\end{lemma}
\begin{proof}
By Möbius inversion and the ideal-counting asymptotic
\eqref{ideal_counting} from \cite[Thm.~VI.3.3]{lang}, we have
\begin{align*}
	\#\{\mathfrak m:N\mathfrak m\le T,\;(\mathfrak m,\mathfrak q)=1\}
	&=
	\rho_K\frac{\varphi_K(\mathfrak q)}{N\mathfrak q}T  
	+
	O_K \biggl(
	T^{1-\frac1{[K:\mathbb Q]}}
	\sum_{\mathfrak e\mid\mathfrak q}
	\frac{1}{(N\mathfrak e)^{1-\frac1{[K:\mathbb Q]}}}
	\biggr) \nonumber\\
	&=
	\rho_K\frac{\varphi_K(\mathfrak q)}{N\mathfrak q}T
	+
	O_K \Bigl(\tau_K(\mathfrak q)T^{1-\frac1{[K:\mathbb Q]}}\Bigr),
	\label{ideal_counting_prime_to_q_original}
\end{align*}
where \(\rho_K=\operatorname*{Res}_{s=1}\zeta_K(s)\), and
\(\tau_K(\mathfrak q)\) denotes the number of ideal divisors of
\(\mathfrak q\).

Grouping prime ideals above each rational prime \(p\), we have
\(\tau_K(\mathfrak q)\le \tau(Q)^{[K:\mathbb Q]},\)
where \(Q=N\mathfrak q\) and \(\tau\) is the usual divisor function for integers.
Hence, for every fixed \(\lambda>0\),
\[
\tau_K(\mathfrak q)\ll_{K,\lambda}Q^\lambda.
\]
Moreover,
\[
\Bigl(\frac{\varphi_K(\mathfrak q)}{N\mathfrak q}\Bigr)^{-1}
=
\prod_{\mathfrak p\mid\mathfrak q}
\Bigl(1-\frac1{N\mathfrak p}\Bigr)^{-1}
\ll_{K,\lambda}Q^\lambda.
\]
For every fixed \(\sigma>0\), choosing \(\lambda>0\) sufficiently small in terms of \(K\) and \(\sigma\),
the error term is \(o_K(1)\) times the main term. Therefore, 
\[\#\{\mathfrak m:N\mathfrak m\le T,\;(\mathfrak m,\mathfrak q)=1\}
=
\rho_K\frac{\varphi_K(\mathfrak q)}{N\mathfrak q}
T +O_{K,\lambda}\Bigl(Q^\lambda T^{1-\frac{1}{[K:\mathbb Q ]}}\Bigr)   = 
\rho_K\frac{\varphi_K(\mathfrak q)}{N\mathfrak q}
T\Bigl(1+o_{K,\sigma}(1)\Bigr),\]
uniformly for \(T\ge Q^\sigma\).
\end{proof}

Let \(0<\alpha<1\). We use the following fixed-parameter version of
\(\mathrm{CS}(\alpha)\). This formulation records the dependence on
\(\varepsilon\) and \(\eta\), since the estimates below will depend
explicitly on these parameters.

	\label{hyp:CS-alpha-eps-eta}
	Let \(\varepsilon>0\) and \(\eta>0\) be fixed. We say that
	\(\mathrm{CS}(\alpha,\varepsilon,\eta)\) holds if, for every non-principal
	character \(\psi\) modulo $\mathfrak{q}$,  whenever 
	$
	T\ge (N\mathfrak q)^{\alpha+\varepsilon},
	$
	we have
\begin{equation}\label{CS_fixed}
		\sum_{N\mathfrak a\le T}\psi(\mathfrak a)
	\ll_{K,\varepsilon}
	T^{1-\eta}.
	\tag{\(\mathrm{CS}(\alpha,\varepsilon,\eta)\)}
\end{equation}

\begin{lemma}[Halász--Montgomery-type estimate]	\label{lem:ray-HM}
	Let \(0<\alpha<1\), \(\varepsilon>0\), \(\eta>0\), and \(C\ge1\) be fixed.
	Assume \eqref{CS_fixed}.
Let \(\chi_1,\dots,\chi_R\) be distinct characters of
\(G=\operatorname{Cl}_{\mathfrak q}^{(\infty)}\), and put
\(Q=N\mathfrak q\).

	\begin{enumerate}
		\item[(i)]
		Let
	$
	  X\ge Q^{\alpha+2\varepsilon}.
	$
		Then, for any complex coefficients \(a_{\mathfrak a}\),
		\[\sum_{j=1}^{R}
			\biggl|
		\sum_{ {N\mathfrak a\le X\atop
				(\mathfrak a,\mathcal P(Q^\varepsilon))=1}}
		a_{\mathfrak a}\chi_j(\mathfrak a)
			\biggr|^2 \ll_{K,\varepsilon,\eta}
		\Bigl(
		\frac{X}{\log Q}
		+
		R\,X^{1-\eta}Q^{\varepsilon\eta}
		\Bigr)
		\sum_{ {N\mathfrak a\le X\atop
				(\mathfrak a,\mathcal P(Q^\varepsilon))=1}}
		|a_{\mathfrak a}|^2.
		\]
		
		\item[(ii)]
		Let
		$
		Q^C\ge X_2\ge X_1\ge Q^{\alpha+2\varepsilon}.
		$
		Then, for any complex coefficients \(a_{\mathfrak a}\),
		\[
		\sum_{j=1}^{R}
			\biggl|
		\sum_{ {X_1<N\mathfrak a\le X_2\atop
				(\mathfrak a,\mathcal P(Q^\varepsilon))=1}}
		\frac{a_{\mathfrak a}}{N\mathfrak a}
		\chi_j(\mathfrak a)
			\biggr|^2 \ll_{K,C,\varepsilon,\eta}
		\biggl(
		1+
		R\Bigl(\frac{Q^\varepsilon}{X_1}\Bigr)^\eta
		\biggr)
		\sum_{ {X_1<N\mathfrak a\le X_2\atop
				(\mathfrak a,\mathcal P(Q^\varepsilon))=1}}
		\frac{|a_{\mathfrak a}|^2}{N\mathfrak a}.\]
	\end{enumerate}
\end{lemma}
\begin{proof} 
We first prove~(ii), since the proof of~(i) is similar.

		\medskip
	\noindent
	\textit{Proof of (ii).}
	By the duality principle \cite[Sect.~7.1, p.~170]{book_analytic}, it suffices to show that for any complex numbers
	\(c_1,\dots,c_R\),
	\begin{equation}\label{3.1}
		\sum_{ {X_1<N\mathfrak a\le X_2\atop
				(\mathfrak a,\mathcal P(Q^{\varepsilon}))=1}}
		\frac1{N\mathfrak a}
			\biggl|
		\sum_{j=1}^{R} c_j\chi_j(\mathfrak a)
			\biggr|^2 \ll
		\biggl(
		1+
		R\Bigl(\frac{Q^\varepsilon}{X_1}\Bigr)^\eta
		\biggr)
		\sum_{j=1}^{R}|c_j|^2.
	\end{equation}

	Let \(\lambda_{\mathfrak d}^{+}\) be the upper-bound linear sieve  weights as in Lemma \ref{sieve_lemma}
	with sifting range \(z=Q^\varepsilon\) and level \(D=z\), so that
	$
	1_{(\mathfrak a,\mathcal P(Q^\varepsilon))=1}
	\le
	\sum_{\mathfrak d\mid\mathfrak a}
	\lambda_{\mathfrak d}^{+}.
	$
	Hence the left-hand side of \eqref{3.1} is at most
	\begin{equation}
			\sum_{X_1<N\mathfrak a\le X_2}
		\frac{1}{N\mathfrak a}
		\sum_{\mathfrak d\mid\mathfrak a}
		\lambda_{\mathfrak d}^{+}
		\biggl|
		\sum_{j=1}^{R} c_j\chi_j(\mathfrak a)
		\biggr|^2 
		=  
		\sum_{j,k=1}^{R}c_j\overline{c_k}
		\sum_{N\mathfrak d\le D}
		\frac{\lambda_{\mathfrak d}^{+}}{N\mathfrak d}
		\chi_j(\mathfrak d)\overline{\chi_k(\mathfrak d)}\!\!\!
		\sum_{X_1/N\mathfrak d<N\mathfrak m\le X_2/N\mathfrak d}\!\!\!
		\frac{
			\chi_j(\mathfrak m)\overline{\chi_k(\mathfrak m)}
		}{N\mathfrak m}.
		\label{3.2}
	\end{equation}
 
	We separate the diagonal and off-diagonal terms.

	\medskip
	\noindent
	\textit{Diagonal terms.}
	If \(j=k\), then
	\(	\chi_j(\mathfrak m)\overline{\chi_j(\mathfrak m)}
	=
	1\)
	for ideals $\mathfrak{m}$ coprime to \(\mathfrak q\). Hence the diagonal contribution is
	\[
	\begin{aligned}
		&\sum_{j=1}^{R}|c_j|^2
		\sum_{ {N\mathfrak d\le D\atop
				(\mathfrak d,\mathfrak q)=1}}
		\frac{\lambda_{\mathfrak d}^{+}}{N\mathfrak d}
		\sum_{ {
				X_1/N\mathfrak d<N\mathfrak m\le X_2/N\mathfrak d\atop
				(\mathfrak m,\mathfrak q)=1}}
		\frac1{N\mathfrak m}.
	\end{aligned}
	\]
	
 For \(N\mathfrak d\le D=Q^\varepsilon\), we have \( X_1/N\mathfrak d\ge Q^{\alpha+\varepsilon}\). By Lemma~\ref{lem:ideal-counting-prime-to-q} and partial summation,   
 \[
 \sum_{ {
 		X_1/N\mathfrak d<N\mathfrak m\le X_2/N\mathfrak d\atop
 		(\mathfrak m,\mathfrak q)=1}}
 \frac1{N\mathfrak m}
 =
 \frac{\varphi_K(\mathfrak q)}{N\mathfrak q}
 \rho_K\log\frac{X_2}{X_1}
 +
 O_{K,C,\varepsilon}(Q^{-\delta_0})
 \]
 for some \(\delta_0>0\).  
 Therefore the diagonal contribution is
 \[
 \sum_{j=1}^{R}|c_j|^2
 \Bigl(
 \frac{\varphi_K(\mathfrak q)}{N\mathfrak q}\log Q
 \sum_{ {N\mathfrak d\le D\atop(\mathfrak d,\mathfrak q)=1}}
 \frac{\lambda_{\mathfrak d}^{+}}{N\mathfrak d}
 +
 o_{K,C,\varepsilon}(1)
 \Bigr).
 \]
 Applying the upper-bound linear sieve with
 \[
 h(\mathfrak p)=
 \begin{cases}
 	1/N\mathfrak p, & \mathfrak p\nmid\mathfrak q,\\
 	0, & \mathfrak p\mid\mathfrak q,
 \end{cases}
 \]
 we get
 \[
 \sum_{ {N\mathfrak d\le D\atop(\mathfrak d,\mathfrak q)=1}}
 \frac{\lambda_{\mathfrak d}^{+}}{N\mathfrak d}
 \ll_{K,\varepsilon}
 \frac{N\mathfrak q}{\varphi_K(\mathfrak q)}\frac1{\log Q}.
 \]
 Thus the diagonal contribution is
\[\ll_{K,C,\varepsilon}
\sum_{j=1}^{R}|c_j|^2.\]

	\medskip
	\noindent
	\textit{Off-diagonal terms.}
	Now suppose \(j\ne k\). Then
	\(	\psi_{j,k}:=\chi_j\overline{\chi_k}\)
	is a non-principal character of \(G\). Since
	\(N\mathfrak d\le D=Q^\varepsilon\)
	and
	\(X_1\ge Q^{\alpha+2\varepsilon},\)
	we have
	\(	\frac{X_1}{N\mathfrak d}
	\ge
	Q^{\alpha+\varepsilon}. \)
	Thus \eqref{CS_fixed} applies to \(\psi_{j,k}\).
	
	By partial summation and \(\mathrm{CS}(\alpha,\varepsilon,\eta)\),
	\[
	\sum_{X_1/N\mathfrak d<N\mathfrak m\le X_2/N\mathfrak d}
	\frac{\psi_{j,k}(\mathfrak m)}{N\mathfrak m}
	\ll_{K,\varepsilon,\eta}
	\Bigl(\frac{X_1}{N\mathfrak d}\Bigr)^{-\eta}.
	\]

	Therefore, 	using
	$
	\sum_{N\mathfrak d\le D}(N\mathfrak d)^{-1+\eta}
	\ll_{K,\eta} D^\eta,
	$ the off-diagonal contribution in \eqref{3.2} is
	\[
	\ll_{K,\varepsilon,\eta}
	X_1^{-\eta}
	\sum_{ {1\le j,k\le R\atop j\ne k}}
	|c_jc_k|
	\sum_{N\mathfrak d\le D}
	(N\mathfrak d)^{-1+\eta}	\ll_{K, \eta}
	\Bigl(\frac{D}{X_1}\Bigr)^\eta
	\sum_{ {1\le j,k\le R\atop j\ne k}}
	|c_jc_k|.
	\]
	Since
$
	\sum_{j\ne k}|c_jc_k|
	\le
(R-1)\sum_{j=1}^{R}|c_j|^2,
	$
	the off-diagonal contribution is
	\[
	\ll_{K,\varepsilon,\eta}
	R
	\Bigl(\frac{Q^\varepsilon}{X_1}\Bigr)^\eta
	\sum_{j=1}^{R}|c_j|^2.
	\]
	
	Combining the diagonal and off-diagonal bounds proves \eqref{3.1}, and hence
	part~(ii).
 
 \medskip
 \noindent
 \textit{Proof of (i).}
	By duality, it suffices to prove that for any complex coefficients
	\(c_1,\dots,c_R\),
	\begin{equation}\label{3.3}	 
			 \sum_{ {N\mathfrak a\le X\atop
					(\mathfrak a,\mathcal P(Q^\varepsilon))=1}}
				\biggl|
			\sum_{j=1}^{R}c_j\chi_j(\mathfrak a)
				\biggr|^2
		 \ll
			\Bigl(
			\frac{X}{\log Q}
			+
			R\,X^{1-\eta}Q^{\varepsilon\eta}
			\Bigr)
			\sum_{j=1}^{R}|c_j|^2.
	\end{equation}
	Using the upper-bound sieve weights with sifting range \(Q^\varepsilon\)
	and level \(D=Q^\varepsilon\), the left-hand side of \eqref{3.3} is at most
	\[\sum_{j,k=1}^{R}c_j\overline{c_k}
	\sum_{N\mathfrak d\le D}
	\lambda_{\mathfrak d}^{+}
	\chi_j(\mathfrak d)\overline{\chi_k(\mathfrak d)} 	\sum_{N\mathfrak m\le X/N\mathfrak d}
	\chi_j(\mathfrak m)\overline{\chi_k(\mathfrak m)}.\]	
As in the diagonal estimate in part~(ii), but without the factor
\(1/N\mathfrak a\), the upper-bound linear sieve together with the
coprime ideal-counting asymptotic gives the diagonal contribution
\[
\ll_{K,\varepsilon}
\frac{X}{\log Q}
\sum_{j=1}^{R}|c_j|^2.
\]

	For the off-diagonal terms \(j\ne k\), the character
	\(\chi_j\overline{\chi_k}\) is non-principal. Since
	$
	\frac{X}{N\mathfrak d}
	\ge
	\frac{X}{Q^\varepsilon}
	\ge
	Q^{\alpha+\varepsilon},
	$
	the hypothesis \eqref{CS_fixed} yields
	\[
	\sum_{N\mathfrak m\le X/N\mathfrak d}
	\chi_j(\mathfrak m)\overline{\chi_k(\mathfrak m)}
	\ll_{K,\varepsilon,\eta}
	\Bigl(\frac{X}{N\mathfrak d}\Bigr)^{1-\eta}.
	\]
	Thus the off-diagonal contribution is
	\[
	\ll_{K,\varepsilon,\eta}
	\sum_{j\ne k}|c_jc_k|
	\sum_{N\mathfrak d\le D}
	\Bigl(\frac{X}{N\mathfrak d}\Bigr)^{1-\eta}
	\ll_{K,\eta}
	X^{1-\eta}
	D^\eta
	\sum_{j\ne k}|c_jc_k|
	\ll
	R\,X^{1-\eta}Q^{\varepsilon\eta}
	\sum_{j=1}^{R}|c_j|^2.
	\]
	Combining the diagonal and off-diagonal contributions proves \eqref{3.3},
	and hence part~(i).
\end{proof}

	\section{Dense model  in the ray class group}\label{sec:dense-kneser}
	In this section we adapt the multiplicative dense model argument of
	Matomäki--Teräväinen to the narrow ray class group
	$
	G=\operatorname{Cl}_{\mathfrak q}^{(\infty)}.
	$
Throughout this section we keep the notation \(Q=N\mathfrak q\) and
\(G=\operatorname{Cl}_{\mathfrak q}^{(\infty)}\).  
 We begin by proving a sieve estimate for ideals whose classes lie in a
 given coset. 
\begin{lemma}\label{rough_ideal_lemma}
 Let \(H\le G\) be a
	subgroup of fixed index \(Y\), and let \(b\in G\). Let
	\(X\ge Q^{3\alpha_0+\varepsilon}\), and define
	$
	\vartheta_0
	:=
	1-\varepsilon-\alpha_0\frac{\log Q}{\log X}.
	$
	Let \(0<\gamma<1\) and \( s:=\frac{\vartheta_0}{\gamma}\), define
	\[
	\mathcal N_\gamma(X;b,H)
	:=
	\#\biggl\{
	\mathfrak n:
	\begin{array}{l}
		N\mathfrak n\le X,\ [\mathfrak n]\in bH,\\[2pt]
		N\mathfrak p>X^\gamma\text{ for every prime ideal }\mathfrak p\mid\mathfrak n
	\end{array}
	\biggr\}.
	\]
	Then, uniformly in \(b\) and in subgroups \(H\le G\) of fixed index \(Y\),
	the following hold.
	
	If \(1\le s\le3\), then
	\[
	\mathcal N_\gamma(X;b,H)
	\le
	(1+o_K(1))
	\frac{2}{Y\vartheta_0}
	\frac{X}{\log X}.
	\]
	
	If \(2<s\le4\), then
	\[
	\mathcal N_\gamma(X;b,H)
	\ge
	(1+o_K(1))
	\frac{2\log(s-1)}{Y\vartheta_0}
	\frac{X}{\log X}.
	\]
\end{lemma}
\begin{proof}
	Let
		\[
	z:=X^\gamma,\qquad D:=X^{\vartheta_0},
	\qquad s:=\frac{\vartheta_0}{\gamma},\qquad
	\mathcal P(z):=\prod_{N\mathfrak p<z}\mathfrak p .
	\]

	Let \(\lambda_{\mathfrak d}^{\pm}\) be the upper and lower linear sieve
	weights as in Lemma \ref{sieve_lemma}, with sifting range $z$ and   level \(D=X^{\vartheta_0}=z^s\). Then
	\[
	\sum_{\mathfrak d\mid(\mathfrak n,\mathcal P(z))}
	\lambda_{\mathfrak d}^{-}
	\le
	1_{(\mathfrak n,\mathcal P(z))=1}
	\le
	\sum_{\mathfrak d\mid(\mathfrak n,\mathcal P(z))}
	\lambda_{\mathfrak d}^{+}.
	\]
	Hence
	\begin{equation}\label{lower}
			\mathcal N_\gamma(X;b,H)
		\ge
		\sum_{\mathfrak d\mid\mathcal P(z)}
		\lambda_{\mathfrak d}^{-}
		\sum_{{N\mathfrak m\le X/N\mathfrak d\atop
				[\mathfrak d\mathfrak m]\in bH}}
		1,
	\end{equation}
	and the corresponding upper bound holds with \(\lambda_{\mathfrak d}^{+}\).
	
	We now estimate the inner sum. Let \(\widehat{G/H}\) denote the characters
	of \(G\) that are trivial on \(H\).   By orthogonality,
	\[
	1_{[\mathfrak d\mathfrak m]\in bH}
	=
	\frac1Y
	\sum_{\chi\in \widehat{G/H}}
	\chi(\mathfrak d)\chi(\mathfrak m)\overline{\chi(b)} .
	\]
 If \((\mathfrak d,\mathfrak q)=1\),
	\[
	\sum_{{N\mathfrak m\le X/N\mathfrak d\atop
			[\mathfrak d\mathfrak m]\in bH}}
	1
	=
	\frac1Y
	\sum_{\chi\in \widehat{G/H}}
	\chi(\mathfrak d)\overline{\chi(b)}
	\sum_{N\mathfrak m\le X/N\mathfrak d}\chi(\mathfrak m).
	\]
	If \((\mathfrak d,\mathfrak q)\ne 1\), the  sum is zero.
	
Since \(N\mathfrak d\le D=X^{\vartheta_0}\),
we have
$
\frac{X}{N\mathfrak d}\ge \frac{X}{D}
=
X^{1-\vartheta_0}
=
X^{\varepsilon}Q^{\alpha_0}\ge Q^{ \alpha_0 +3\alpha _0\varepsilon +\varepsilon^2 }.
$ 
	By Lemma \ref{lem:ideal-counting-prime-to-q}, the principal character contributes
	\[
	\frac1Y
	\#\{\mathfrak m:\, N\mathfrak m\le X/N\mathfrak d,\,
	(\mathfrak m,\mathfrak q)=1\}=
	\frac{\rho_K}{Y}
	\prod_{\mathfrak p\mid\mathfrak q}
	\Bigl(1-\frac1{N\mathfrak p}\Bigr)
	\frac{X}{N\mathfrak d}
	\bigl(1+o_{K,\varepsilon}(1)\bigr).
	\]
	
	For every non-principal \(\chi\in\widehat{G/H}\), we use
	\eqref{CS^b}. 	
  Then, there exists \(\delta>0\) such that
	\[
	\sum_{N\mathfrak m\le X/N\mathfrak d}\chi(\mathfrak m)
	\ll_{K,\varepsilon}
\Bigl(\frac{X}{N\mathfrak d}\Bigr)^{1-\delta}.
	\]
	Since \(Y\) is fixed, the total non-principal contribution is
	$
	O_{K,\varepsilon } \Bigl(
\bigl(\frac{X}{N\mathfrak d}\bigr)^{1-\delta}\Bigr).
	$
	
	Combining the principal and non-principal estimates gives, uniformly for
	\(N\mathfrak d\le D\),
	\[
	\sum_{{N\mathfrak m\le X/N\mathfrak d\atop
			[\mathfrak d\mathfrak m]\in bH}}
	1
	=
	\frac{\rho_K}{Y}
	\prod_{\mathfrak p\mid\mathfrak q}
	\Bigl(1-\frac1{N\mathfrak p}\Bigr)
	\frac{X}{N\mathfrak d}
\bigl(1+o_{K,\varepsilon}(1)\bigr).
	\]

	Substituting this into the lower sieve inequality \eqref{lower}, we obtain
	\[
	\mathcal N_\gamma(X;b,H)
	\ge
	\frac{\rho_K X}{Y}
	\prod_{\mathfrak p\mid\mathfrak q}
	\Bigl(1-\frac1{N\mathfrak p}\Bigr)
	\sum_{{\mathfrak d\mid\mathcal P(z)\atop
			(\mathfrak d,\mathfrak q)=1}}
	\frac{\lambda_{\mathfrak d}^{-}}{N\mathfrak d}
	\bigl(1+o_{K,\varepsilon}(1)\bigr).
	\]

	It remains to evaluate the sieve sum. Define a multiplicative function \(h\)
	on ideals by
	\[
	h(\mathfrak d)
	=
	\begin{cases}
		1/N\mathfrak d, & (\mathfrak d,\mathfrak q)=1,\\
		0, & (\mathfrak d,\mathfrak q)\ne 1.
	\end{cases}
	\]
	The   linear sieve for number fields Lemma \ref{sieve_lemma} gives
	\[
	\sum_{\mathfrak d\mid\mathcal P(z)}
	\lambda_{\mathfrak d}^{-}h(\mathfrak d)
	\ge
	(f_0(s)+o_K(1))
	\prod_{{N\mathfrak p<z\atop \mathfrak p\nmid\mathfrak q}}
	\Bigl(1-\frac1{N\mathfrak p}\Bigr).
	\]

	Multiplying by
	\(\prod_{\mathfrak p\mid\mathfrak q}(1-1/N\mathfrak p)\), we get
	\[
	\prod_{\mathfrak p\mid\mathfrak q}\Bigl(1-\frac1{N\mathfrak p}\Bigr)
	\prod_{{N\mathfrak p<z\atop \mathfrak p\nmid\mathfrak q}}
\Bigl(1-\frac1{N\mathfrak p}\Bigr)
	=
	(1+o_K(1))
	\prod_{N\mathfrak p<z}
	\Bigl(1-\frac1{N\mathfrak p}\Bigr).
	\]
	Here the primes dividing \(\mathfrak q\) with norm \(\ge z\) contribute only
	\(1+O(\frac{\log Q}{z\log z})= 1+o(1)\). By Mertens' theorem \cite[Thm.~2]{merten} for \(K\),
\begin{equation}\label{merten}
	\prod_{N\mathfrak p<z}
\Bigl(1-\frac1{N\mathfrak p}\Bigr)
=
\frac{e^{-\gamma }}{\rho_K\log z} +O_K\Bigl(\frac{1}{(\log z)^2}\Bigr).
\end{equation}
  When  \(2\le s\le 4 \),   $
	f_0(s)=\frac{2e^\gamma\log(s-1)}{s} 
	$. Applying the lower bound sieve in Lemma \ref{sieve_lemma}  gives
	\[
	\mathcal N_\gamma(X;b,H)
	\ge
	(1+o_{K,\varepsilon}(1))
	\frac{  X}{Y}
	f_0(s)\frac{e^{-\gamma }}{\log z} 	=
	(1+o_{K,\varepsilon}(1))
	\frac{2\log (s-1)}{Y\vartheta_0}
	\frac{X}{\log X}.
	\]
	This is the stated lower bound.
	
	The upper bound is identical, using the upper-bound sieve weights
	\(\lambda_{\mathfrak d}^{+}\). The linear sieve gives
	\[
	\sum_{\mathfrak d\mid\mathcal P(z)}
	\lambda_{\mathfrak d}^{+}h(\mathfrak d)
	\le
	(F_0(s)+o_K(1))
	\prod_{{N\mathfrak p<z\atop \mathfrak p\nmid\mathfrak q}}
	\Bigl(1-\frac1{N\mathfrak p}\Bigr),
	\]
	and for \(1\le s\le 3\),
$
	F_0(s)=\frac{2e^\gamma}{s}.
$
	Hence
	\[
	\mathcal N_\gamma(X;b,H)
	\le
	(1+o_{K,\varepsilon}(1))
	\frac{2}{Y\vartheta_0}
	\frac{X}{\log X}.
	\]
	This proves the lemma.
\end{proof}

	For \(X\ge 2\), define
\[
\mathcal A(X;\mathfrak q)
:=
\{\mathfrak a\subset \mathcal O_K :
(\mathfrak a,\mathfrak q)=1,\ N\mathfrak a\le X\}.
\]
For functions
$f:\mathcal A(X;\mathfrak q)\to \mathbb R_{\geq 0}$ and
$g:G\to \mathbb R_{\geq 0}$, write
\[
\mathbb E_{\mathfrak a\in\mathcal A(X;\mathfrak q)} f(\mathfrak a)
:=
\frac{1}{|\mathcal A(X;\mathfrak q)|}
\sum_{\mathfrak a\in\mathcal A(X;\mathfrak q)} f(\mathfrak a),
\qquad
\mathbb E_{c\in G} g(c)
:=
\frac{1}{|G|}\sum_{c\in G}g(c).
\] 
More generally, for a finite non-empty set \(S\) and a function \(F:S^k\to\mathbb C\), write
\[
\mathbb E_{x_1,\dots,x_k\in S}F(x_1,\dots,x_k)
:=
\frac1{|S|^k}\sum_{x_1,\dots,x_k\in S}F(x_1,\dots,x_k).
\]

For $\chi\in\widehat G$,  define 
\[
\widehat f(\chi)
:=
\mathbb E_{\mathfrak a\in\mathcal A(X;\mathfrak q)}
f(\mathfrak a) {\chi([\mathfrak a])},
\qquad
\widehat g(\chi)
:=
\mathbb E_{c\in G} g(c) {\chi(c)}.
\] 

\begin{proposition}[Multiplicative dense model]\label{prop:dense-model-ray}
Let \(r>1\), \(C\ge1\), and \(\eta,\theta\in(0,1)\) be fixed. Let
$
\delta\in
\Bigl(
\bigl(
\frac{10Cr\log\log Q}{\theta\log Q}
\bigr)^{1/r},
\frac{1}{10}
\Bigr),
$ where this interval is assumed to be nonempty.
	Suppose
\(f:\mathcal A(X;\mathfrak q)\to \mathbb R_{\ge0}\)
	satisfies the following two conditions.
	
	\begin{itemize}
		\item[(A1)]
		There exists a function
		\(\nu:\mathcal A(X;\mathfrak q)\to \mathbb R_{\ge0}\) such that  
		\[
		f(\mathfrak a)\le \nu(\mathfrak a)
		\quad(\mathfrak a\in\mathcal A(X;\mathfrak q)),  	
		\]
		and for some fixed constant \(C_1>0\),
		\[
		\mathbb E_{\mathfrak a\in\mathcal A(X;\mathfrak q)}
		\nu(\mathfrak a)\le 1+\eta,\qquad
		\max_{\chi\neq\chi_0}
		|\widehat\nu(\chi)|
		\le C_1Q^{-\theta }.
		\]

		\item[(A2)]
	There are at most \(C\delta^{-r}\) characters
	\(\chi\in\widehat G\) satisfying
		 \[\bigl|
		 \widehat f(\chi)
		 \bigr|
		 \ge \delta.\]
	\end{itemize}
	
	Then there exists a function
	\(	g:G\longrightarrow \mathbb R_{\ge0}\)
	satisfying the following properties.
	
	\begin{itemize}
		\item[(i)]
		For every \(c\in G\),
		\[
		0\le g(c)\le 1+\eta+C_1 Q^{-\theta/2} .
		\]
		
		\item[(ii)]
		For every character \(\chi\in\widehat G\),
		\[
		\bigl|
		\widehat f(\chi)
		-
	\widehat g(\chi)
		\bigr|
		\le \delta.
		\]
		
		\item[(iii)]
		For every character \(\chi\in\widehat G\),
		\[
			\bigl|
	\widehat g(\chi)
			\bigr|
		\le
			\bigl|
		\widehat f(\chi)
		\bigr|,
		 \qquad
		\bigl|
	\widehat f(\chi)
		-
		\widehat g(\chi)
		\bigr|
		\le
			\bigl|
	\widehat f(\chi)
			\bigr|.
		\]
		
		\item[(iv)]
		\[
		\mathbb E_{c\in G} g(c)
		=
		\mathbb E_{\mathfrak a\in\mathcal A(X;\mathfrak q)}
		f(\mathfrak a).
		\]
		
		\item[(v)]
		Let \(H\le G\) be a subgroup of index
		\(Y<1/(2\delta)\). Then, for every coset \(bH\) of \(H\),
		\[
		\bigl|\mathbb E_{\mathfrak a\in\mathcal A(X;\mathfrak q)}
		f(\mathfrak a)\,1_{[\mathfrak a]\in bH}
		-
		\mathbb E_{c\in G}
		g(c)\,1_{c\in bH} \bigr|
		<   \delta.
		\]
	\end{itemize}
\end{proposition}	
\begin{proof}
	Define the large spectrum of \(f\) by
	\[
	\mathcal T
	:=
	\bigl\{\chi\in\widehat G:\ |\widehat f(\chi)|\ge \delta\bigr\}.
	\]
	By assumption (A2),
	$
	|\mathcal T|\le C\delta^{-r}.
	$
	
	We define the multiplicative Bohr set
	\[
	B
	:=
	\Bigl\{
	b\in G:
	|\chi(b)-1|\le \frac{\delta}{5}
	\text{ for every }\chi\in\mathcal T
	\Bigr\}.
	\]
	
	We first give a lower bound for \(|B|\). Write
	\[
	\mathcal T=\{\chi_1,\dots,\chi_k\},
\qquad k\le C\delta^{-r}.
	\]
	Choose
	\(	L=\bigl\lceil \frac{10\pi}{\delta}\bigr\rceil\)
	equally spaced points \(\zeta_1,\dots,\zeta_L\in S^1\). Then every point
	of \(S^1\) has 
	distance at most \(\pi/L\le \delta/10\) from some \(\zeta_\ell\).

 For each \(a\in G\), and for each \(j=1,\dots,k\), choose an index
 \(\ell_j(a)\in\{1,\dots,L\}\) such that
 $
 |\chi_j(a)-\zeta_{\ell_j(a)}|\le \frac{\delta}{10}.
 $
 Thus each \(a\in G\) determines a \(k\)-tuple
$
 (\ell_1(a),\dots,\ell_k(a))\in \{1,\dots,L\}^k.
$
 Since there are \(L^k\) possible such tuples, the pigeonhole principle gives
 a tuple \((\ell_1,\dots,\ell_k)\) which occurs for at least \(L^{-k}|G|\)
 elements of \(G\). Setting
 $
 \xi_j:=\zeta_{\ell_j}$,  $j=1,\dots,k,
 $
 we obtain a set
 \[
 A
 :=
 \Bigl\{
 a\in G:
 |\chi_j(a)-\xi_j|\le \frac{\delta}{10}
 \text{ for every }j=1,\dots,k
 \Bigr\}
 \]
 satisfying
 \[
 |A|\ge L^{-k}|G|
 =
 \Bigl\lceil \frac{10\pi}{\delta}\Bigr\rceil^{-k}|G|.
 \]
 
 If \(x,y\in A\), then for every \(j=1,\dots,k\),
 $
 |\chi_j(x)-\chi_j(y)|
 \le
 |\chi_j(x)-\xi_j|+|\chi_j(y)-\xi_j|
 \le \frac{\delta}{5}.
$
 Since \(|\chi_j(y)|=1\), it follows that
 \[
 \bigl|\chi_j(xy^{-1})-1\bigr|
 =
 \biggl|\frac{\chi_j(x)}{\chi_j(y)}-1\biggr|
 =
 |\chi_j(x)-\chi_j(y)|
 \le \frac{\delta}{5}.
 \]
 Thus \(xy^{-1}\in B\).
 Fix \(y_0\in A\), then
$
 Ay_0^{-1}\subseteq B.
 $
 Therefore
 \[
 |B|\ge |Ay_0^{-1}|=|A|
 \ge
 |G|
 \Bigl\lceil \frac{10\pi}{\delta}\Bigr\rceil^{-C\delta^{-r}}.
 \]

 Since \(\delta<1/10\), we have
$
 \Bigl\lceil \frac{10\pi}{\delta}\Bigr\rceil
 \le \frac{20\pi}{\delta}.
$
 By the choice of \(\delta\),
 $
 C\delta^{-r}
 \le
 \frac{\theta\log Q}{10r\log\log Q}.
$
Moreover, since $Q$ is an integer and the interval defining $\delta$ is nonempty, we have $\log\log Q\ge1$. It follows that 
 $
 \log \frac{20\pi}{\delta}\le 5r\log\log Q.
$ 
 Consequently
 \[
 \log \frac{|G|}{|B|} \le
 C\delta^{-r}\log \frac{20\pi}{\delta}
 \le
 \frac{\theta\log Q}{10r\log\log Q}
 \cdot
 5r\log\log Q
 =
 \frac{\theta}{2}\log Q.
 \]
 Hence
 \[
 \frac{|G|}{|B|}\le Q^{\theta/2}.
 \]

	Now define
	\[
	g(c)
	:=
	|G|\,
	\mathbb E_{b_1,b_2\in B}
	\mathbb E_{\mathfrak a\in\mathcal A(X;\mathfrak q)}
	f(\mathfrak a)\,
	1_{[\mathfrak a]=cb_1b_2^{-1}}.
	\]	
	Clearly \(g(c)\ge0\).
	We now verify the five claims.	

	Using  that \(f\le \nu\) and the orthogonality of   characters of \(G\),
	\[
	1_{[\mathfrak a]=cb_1b_2^{-1}}
	=
	\frac1{|G|}
	\sum_{\chi\in\widehat G}
	\chi([\mathfrak a])\,
	\overline{\chi(c)}\,
	\overline{\chi(b_1)}\,
	\chi(b_2),
	\]
 we have 
	\[
	g(c)
	\le
	|G|\,
	\mathbb E_{b_1,b_2\in B}
	\mathbb E_{\mathfrak a\in\mathcal A(X;\mathfrak q)}
	\nu(\mathfrak a)\,
	1_{[\mathfrak a]=cb_1b_2^{-1}}
	=
	\sum_{\chi\in\widehat G}
	\widehat\nu(\chi)\,
	\overline{\chi(c)}
	\bigl|
	\mathbb E_{b\in B}\chi(b)
	\bigr|^2.
	\]
	Separating the principal character,
	\[
	g(c)
	\le
	\widehat\nu(\chi_0)
	+
	\max_{\chi\ne\chi_0}|\widehat\nu(\chi)|
	\sum_{\chi\in\widehat G}
	\bigl|
	\mathbb E_{b\in B}\chi(b)
	\bigr|^2.
	\]
By (A1),
\[
\widehat \nu(\chi_0)
=
\mathbb E_{\mathfrak a}\nu(\mathfrak a)
\leq 1+\eta,
\qquad\text{and}\qquad
\max_{\chi\neq\chi_0}|\widehat \nu(\chi)|
\leq C_1 Q^{-\theta}.
\]
	Moreover, by orthogonality of characters
	\[
	\sum_{\chi\in\widehat G}
	\bigl|
	\mathbb E_{b\in B}\chi(b)
	\bigr|^2
	=
	\frac{|G|}{|B|}
	\leq Q^{\theta/2}.
	\]
	Hence
	\[
	g(c)
	\le
	1+\eta+C_1 Q^{-\theta/2} ,
	\]
	which proves (i).

For any \(\chi\in\widehat G\), we have
\[
\begin{aligned}
	\mathbb E_{c\in G} g(c)\chi(c)
	&=
	\mathbb E_{b_1,b_2\in B}
	\mathbb E_{\mathfrak a\in\mathcal A(X;\mathfrak q)}
	f(\mathfrak a)\,
	\chi\bigl([\mathfrak a]b_1^{-1}b_2\bigr)  \\
	&=
	\bigl|\mathbb E_{b\in B}\chi(b)\bigr|^2
	\mathbb E_{\mathfrak a\in\mathcal A(X;\mathfrak q)}
	f(\mathfrak a)\chi([\mathfrak a]).
\end{aligned}
\]
Thus
\begin{equation}\label{gchi_fchi}
	\widehat g(\chi)
	=
	\bigl|\mathbb E_{b\in B}\chi(b)\bigr|^2
	\widehat f(\chi).
\end{equation}
It follows from \eqref{gchi_fchi} that
\begin{equation}\label{f-g}
	\widehat f(\chi)-\widehat g(\chi)
	=
	\widehat f(\chi)
	\Bigl(
	1-\bigl|\mathbb E_{b\in B}\chi(b)\bigr|^2
	\Bigr).
\end{equation}
If \(\chi\notin\mathcal T\), then
\(
|\widehat f(\chi)|<\delta.
\)
Since
$
0\le
\bigl|\mathbb E_{b\in B}\chi(b)\bigr|^2
\le 1,
$
we obtain
\[
|\widehat f(\chi)-\widehat g(\chi)|
<
\delta.
\]

Now suppose that \(\chi\in\mathcal T\). For every \(b\in B\), we have
$
|\chi(b)-1|
\le \frac{\delta}{5},
$
and hence
$
\left|
\mathbb E_{b\in B}\chi(b)-1
\right|
\le \frac{\delta}{5}.
$
Therefore
\[
\Bigl|
1-
\bigl|
\mathbb E_{b\in B}\chi(b)
\bigr|^2
\Bigr|
\le
\Bigl(1+\frac{\delta}{5}\Bigr)^2-1.
\]
Since \(f\le \nu\),
\[
\bigl|\widehat f(\chi)\bigr|
\le
\mathbb E_{\mathfrak a}\nu(\mathfrak a)
\le 1+\eta.
\]
Consequently, by \eqref{f-g}, 
\[ \bigl|\widehat f(\chi)-\widehat g(\chi)\bigr| \le (1+\eta) \Bigl( \bigl(1+\frac{\delta}{5}\bigr)^2-1 \Bigr) \le \delta, \]
since \(\eta<1\) and \(\delta<1/10\).
This proves (ii).

  {(iii)} follows  immediately from \eqref{gchi_fchi} and \eqref{f-g}, since
$
0\le
\bigl|\mathbb E_{b\in B}\chi(b)\bigr|^2
\le 1.
$

	(iv) follows  by taking
	\(\chi=\chi_0\) in \eqref{gchi_fchi} , since
	\(
	\mathbb E_{b\in B}\chi_0(b)=1.
	\)
	
	Let \(H\le G\) with index \(Y<1/(2\delta)\), and define
	\[
	\mathcal D
	:=
	\{\chi\in\widehat G:\ \chi(h)=1\text{ for every }h\in H\}.
	\]
	Then \(|\mathcal D|=Y\), and the orthogonality of characters gives
	\[
	1_{d\in bH}
	=
	\frac1Y
	\sum_{\chi\in\mathcal D}
	\chi(d)\overline{\chi(b)}.
	\]
	Therefore
	\[
\mathbb E_{\mathfrak a}
f(\mathfrak a)\,1_{[\mathfrak a]\in bH}
-
\mathbb E_{c\in G}
g(c)\,1_{c\in bH}   =
\frac1Y
\sum_{\chi\in\mathcal D}
\overline{\chi(b)}
\bigl(
\widehat f( {\chi})
-
\widehat g( {\chi})
\bigr).
	\]
	
	We now show that each summand is \(O(\delta)\).
	
	If \(\chi\notin\mathcal T\), then by (iii) and the definition of
	\(\mathcal T\),
	\[
	\bigl|\widehat f(\chi)-\widehat g(\chi)\bigr|
	\le \bigl|\widehat f(\chi)\bigr|
	< \delta.
	\]
	
	Now suppose that \(\chi\in\mathcal T\cap\mathcal D\). Since
	\(\chi\) is trivial on \(H\), its order divides \(Y\). Hence every value
	\(\chi(b)\) is a \(Y\)-th root of unity. Since
	$
	Y<\frac1{2\delta},
	$
the only \(Y\)-th root of unity lying within
distance \(\delta/5\) of \(1\) is \(1\) itself. Since
\[
	B
:=
\Bigl\{
b\in G:
|\chi(b)-1|\le \frac{\delta}{5}
\text{ for every }\chi\in\mathcal T
\Bigr\},
\]
it follows that \(\chi(b)=1\) for every \(b\in B\).
	Hence
	$
	\mathbb E_{b\in B}\chi(b)=1.
	$
	By \eqref{gchi_fchi},
	$
	\widehat g(\chi)=\widehat f(\chi).
	$
Thus the contribution of every \(\chi\in\mathcal T\cap\mathcal D\) vanishes.
	
	Combining the two cases, every summand is \(< \delta\),  therefore
	\[
	\bigl|\mathbb E_{\mathfrak a}
	f(\mathfrak a)\,1_{[\mathfrak a]\in bH}
	-
	\mathbb E_{c\in G}
	g(c)\,1_{c\in bH}\bigr|< \Bigl(1-\frac{|\mathcal T\cap\mathcal D|}{Y}\Bigr) 
	 \delta \leq \delta .
	\]
	This proves (v), and hence the proposition.
\end{proof}
	
\subsection{Applying the transference principle}

For \(k\ge1\), \(X\ge 2\), recall the notations
\[
\mathcal A(X;\mathfrak q)
:=
\bigl\{\mathfrak a\subset \mathcal O_K :
(\mathfrak a,\mathfrak q)=1,\ N\mathfrak a\le X\bigr\},
\]
\[
E_k(X;\mathfrak q)
:=
\bigl\{
[\mathfrak p_1]\cdots[\mathfrak p_k]\in G :
\mathfrak p_i\nmid\mathfrak q,\
N\mathfrak p_i\le X
\bigr\}.
\]
For a subgroup \(H\le G\) and a coset \(bH\subseteq G\), write
\begin{equation}\label{def_pi}
	\pi (X;bH)
	:= \#
	 \bigl\{
	\mathfrak p\subset\mathcal O_K:
	\mathfrak p\nmid\mathfrak q,\ 
	N\mathfrak p\le X,\ 
	[\mathfrak p]\in bH
\bigr\}.
\end{equation}
We shall use the following two short character-sum bound assumptions. Let
$
0<\alpha_0\le \alpha<1.
$
We assume the following two character sum assumptions.  For every \(\varepsilon>0\), there exists \(\eta_1=\eta_1(\varepsilon)>0\) such that for every non-principal Hecke character \(\chi\) modulo \(\mathfrak q\), whenever
\(T\ge Q^{\alpha+\varepsilon},\)
\begin{equation}\label{CS}
	\sum_{N\mathfrak a\le T}\chi(\mathfrak a)
	\ll_{K,\varepsilon}
	T^{1-\eta_1};
	\tag{\(\mathrm{CS}   {(\alpha )}\)}
\end{equation}

For every fixed
integer \(\ell\ge2\)  and every \(\varepsilon>0\), there exists \(\eta_2=\eta_2(\ell,\varepsilon)>0\) such that for every non-principal Hecke character \(\chi\) modulo \(\mathfrak q\) of order at most $\ell$,  whenever
\(T\ge Q^{\alpha_0+\varepsilon},\)
\begin{equation}\label{CS^b}
	\sum_{N\mathfrak a\le T}\chi(\mathfrak a)
	\ll_{K,\ell,\varepsilon}
	T^{1-\eta_2}
	\tag{\(\mathrm{CS}^{\mathrm{b }} {(\alpha_0)}\)}.
\end{equation}
We assume throughout   that \(\mathrm{CS} (\alpha )\) and \(\mathrm{CS}^{\mathrm b}(\alpha_0)\) 
hold  for some fixed  \(0<\alpha_0\leq \alpha<1\).  

For \(X\ge2\), define
\begin{equation}\label{5.1}
	\vartheta
	:=
	1-\varepsilon-\alpha\frac{\log Q}{\log X},
	\qquad
	\vartheta_0
	:=
	1-\varepsilon-\alpha_0\frac{\log Q}{\log X}.
\end{equation}

\begin{proposition}[Conclusion of transference over ray class groups]
	\label{prop:transference-conclusion}
 Let \(\kappa,\varepsilon>0\) and \(C\ge1\) be fixed, with
\(\varepsilon>0\) sufficiently small.  Let \(\vartheta\) and \(\vartheta_0\) be as in \eqref{5.1}.
	Assume that  
	\[X \in \bigl[Q^{\frac{\alpha+3\varepsilon}{1-\varepsilon}}, Q^C \bigr]  .\] 	
	Then,  when $Q$ is sufficiently large, there exists a set \(A\subseteq G\) such that the following hold.
	
	\begin{enumerate}
		\item[(i)]
		\[
		|A|
		\ge
		\Bigl(\frac{\vartheta}{2}-\varepsilon\Bigr)|G|.
		\]
		
		\item[(ii)]
		Assume that \(X\ge Q^{1+\kappa}.\)
		Then for all but
		$
		O_{K,\varepsilon}\bigl(|G|(\log Q)^{-\varepsilon /2}\bigr),  
		$
		elements \(c\in G\), we have
		\[
		(1_A*1_A)(c)
		\gg
		\frac{|G|}{(\log Q)^{1/2-\varepsilon}}
		\quad\Longrightarrow\quad
		c\in E_2(X;\mathfrak q).
		\]
		
		\item[(iii)]
		Assume  that
		\(X\ge Q^{1+\kappa}.\)
		Then, for every \(c\in G\),
		\[
		(1_A*1_A*1_A)(c)
		\gg
		\frac{|G|^2}{(\log Q)^{1/2-\varepsilon}}
		\quad\Longrightarrow\quad
		c\in E_3(X;\mathfrak q).
		\]
		
		\item[(iv)]
		Assume   that
	\(B:=\frac{\log X}{\log Q}\ge 2\alpha+\kappa.\)
		Let \(	\beta\in(2\varepsilon,B]\) and \(	L=\bigl\lfloor\frac{B}{\beta}\bigr\rfloor.\)
		Then, for every \(c\in G\),
		\[
		\sum_{[\mathfrak p]a_1a_2=c,\ a_1,a_2\in A
			\atop
			Q^{\beta-\varepsilon}<N\mathfrak p\le Q^\beta}
		\frac1{N\mathfrak p}
		\gg
		\frac{|G|}
		{(\log Q)^{(1-2\varepsilon)/(2L)}}
		\quad\Longrightarrow\quad
		c\in E_3(X;\mathfrak q).
		\]
		\item[(v)]
			Assume  that
		\(X\ge Q^{3\alpha_0+\kappa}.\)
	Let \(H\le G\) be a subgroup of index \(Y<\varepsilon^{-1/2}\). Then there
	exist at least
	$
	\Bigl\lceil
	\bigl(
	\frac{\vartheta_0}{2}
	-
	3\varepsilon^{1/2}\frac{\vartheta_0}{\vartheta}
	\bigr)Y
	\Bigr\rceil
	$
	distinct cosets \(bH\) of \(H\) such that
\[	|A\cap bH|>\varepsilon |G|.\]
		
		\item[(vi)]
		Let \(H\le G\) have index \(Y<\varepsilon^{-1/2}\), and let \(bH\) be a coset. 
		Let \(	\pi (X;bH)\) be as defined in \eqref{def_pi}. Then 
		\[
		|A\cap bH|
		\ge
		\Bigl(
		\frac{\vartheta}{2}
		\frac{ 	\pi (X;bH) }
		{X/\log X}
		-
		\frac{\varepsilon}{5Y}
		\Bigr)|G|.
		\]
	\end{enumerate}
\end{proposition}

\begin{proof}
	Let
	$
	V_{\mathfrak q}
:=
\rho_K
\prod_{\mathfrak p\mid\mathfrak q}
\Bigl(1-\frac1{N\mathfrak p}\Bigr),
	$
	so that by Lemma \ref{lem:ideal-counting-prime-to-q}, 
	\begin{equation}\label{5.2}
			|\mathcal A(X;\mathfrak q)|
		=
		V_{\mathfrak q}X\bigl(1+o_K(1)\bigr).
	\end{equation}
	
	Put
	\[
	D=X^\vartheta,
	\qquad
	z=X^{\vartheta/3}.
	\]
	Let \(\lambda_{\mathfrak d}^+\) be the upper-bound linear sieve weights
	with level \(D\) and sifting range \(z\) as in Lemma \ref{sieve_lemma}.
	
	We shall apply Proposition~\ref{prop:dense-model-ray}
	with
	\[
	r=2,
	\qquad
	\delta
	:=
	\frac{1}{(\log Q)^{1/2-\varepsilon/2}}.
	\]
	
	Define
\(f, \nu:\mathcal A(X;\mathfrak q)\to\mathbb R_{\ge0}  \)
	by
	\[
	f(\mathfrak a)
	:=
	\frac{\vartheta}{2}
	V_{\mathfrak q}\log X\,
	1_{\mathfrak a=\mathfrak p}
	1_{N\mathfrak p\ge z},
	\]
	and  
	\begin{equation}\label{5.6}
			\nu(\mathfrak a)
		:=
		\frac{\vartheta}{2}
		V_{\mathfrak q}\log X
		\sum_{{\mathfrak d\mid\mathfrak a\atop
				N\mathfrak d\le D}}
		\lambda_{\mathfrak d}^+.
	\end{equation}
	By the upper-bound sieve inequality as in  Lemma \ref{sieve_lemma},
	\[
	f(\mathfrak a)\le \nu(\mathfrak a)
	\qquad
	(\mathfrak a\in\mathcal A(X;\mathfrak q)).
	\]
	
	We now verify the assumptions of
	Proposition~\ref{prop:dense-model-ray}.
	
	\medskip
	\noindent
	\textit{Verification of (A2).}
	Let
	\[
	\mathcal T
	:=
	\bigl\{
	\chi\in\widehat G:
	|\widehat f(\chi)|\ge \delta
	\bigr\}.
	\]
	We claim that
	\begin{equation}\label{5.7}
			|\mathcal T|\ll_{K, \varepsilon}\delta^{-2},\qquad \sum_{\chi\in\mathcal T}|\widehat f(\chi)|^2
			\ll_{K, \varepsilon  } 1 
	\end{equation}

	First note that the hypothesis
$
	B:=\frac{\log X}{\log Q}
	\ge
	\frac{\alpha+3\varepsilon}{1-\varepsilon}
	$
	implies
	$
	B\vartheta
	=
	B(1-\varepsilon)-\alpha
	\ge 3\varepsilon.
	$
	Hence
	\[
	z=X^{\vartheta/3}=Q^{B\vartheta/3}\ge Q^\varepsilon.
	\]
	Thus every ideal in the support of \(f\) is coprime to
	\(\mathcal P(Q^\varepsilon)\). Moreover, the same lower bound for \(B\)
	implies \(X\ge Q^{\alpha+2\varepsilon}\), so Lemma~\ref{lem:ray-HM}(i) is
	applicable.
	
	Applying Lemma~\ref{lem:ray-HM}(i) to the set of characters
	\(\mathcal T\), with coefficients
	$
	a_{\mathfrak a}=f(\mathfrak a),
	$
	gives
	\[
	\sum_{\chi\in\mathcal T}
	\biggl|
	\sum_{\mathfrak a\in\mathcal A(X;\mathfrak q)}
	f(\mathfrak a)\chi(\mathfrak a)
	\biggr|^2
	\ll_{K, \varepsilon}
	\Bigl(
	\frac{X}{\log Q}
	+
	|\mathcal T|X^{1-\eta}Q^{\varepsilon\eta}
	\Bigr)
	\sum_{\mathfrak a\in\mathcal A(X;\mathfrak q)}
	f(\mathfrak a)^2
	\]
	for some \(\eta =\eta(\varepsilon)>0\). Dividing by
	\(|\mathcal A(X;\mathfrak q)|^2\), we obtain
	\[
	\sum_{\chi\in\mathcal T}
	|\widehat f(\chi)|^2
	\ll_{K ,\varepsilon}
	\Bigl(
	\frac{X}{\log Q}
	+
	|\mathcal T|X^{1-\eta}Q^{\varepsilon\eta}
	\Bigr)
	\frac{1}{|\mathcal A(X;\mathfrak q)|^2}
	\sum_{\mathfrak a\in\mathcal A(X;\mathfrak q)}
	f(\mathfrak a)^2.
	\]
	
	Now
	$
	f(\mathfrak a)
	=
	\frac{\vartheta}{2}V_{\mathfrak q}\log X
	$
	on prime ideals \(\mathfrak p\) with \(z\le N\mathfrak p\le X\), and is zero
	otherwise. Hence, by the prime ideal theorem and \eqref{5.2},
	\[
	\frac{1}{|\mathcal A(X;\mathfrak q)|^2}
	\sum_{\mathfrak a\in\mathcal A(X;\mathfrak q)}
	f(\mathfrak a)^2
	\ll_K
	\frac{\log X}{X}.
	\]
	Therefore
	\begin{equation}\label{again}
			|\mathcal T|\delta^2
		\le
		\sum_{\chi\in\mathcal T}|\widehat f(\chi)|^2
		\ll_{K, \varepsilon}
		\frac{\log X}{\log Q}
		+
		|\mathcal T|X^{-\eta}Q^{\varepsilon\eta}\log X.
	\end{equation}
	Since \(X\le Q^C\), the first term is \(O_{K,C}(1)\). Also, since
	\(X\ge Q^{3\varepsilon}\) and \(\delta^2=(\log Q)^{-1+\varepsilon}\),
	\[
	X^{-\eta}Q^{\varepsilon\eta}\log X
	\le
	Q^{-2\varepsilon\eta}\log Q
	=o_\varepsilon(\delta^2).
	\]
	Thus, for \(Q\) sufficiently large,
	$
	|\mathcal T|\delta^2
	\ll_{K, \varepsilon}
	1+\frac12|\mathcal T|\delta^2.
	$
	Absorbing the second term gives
	$
	|\mathcal T|\ll_{K,\varepsilon}\delta^{-2}.
	$
	Applying the bound for $|\mathcal T|$ back to \eqref{again} gives $
	\sum_{\chi\in\mathcal T}|\widehat f(\chi)|^2
	\ll_{K, \varepsilon } 1 
	$
	This proves \eqref{5.7}.

	\medskip
	\noindent
	\textit{Verification of (A1) for non-principal characters.}
	Let \(\chi\ne\chi_0\). Since \(\nu\) is supported on
	\(\mathcal A(X;\mathfrak q)\), after interchanging the order of summation in
	\eqref{5.6} we get
	\[
	\widehat \nu(\chi)
	=
	\mathbb E_{\mathfrak a\in\mathcal A(X;\mathfrak q)}
	\nu(\mathfrak a)\chi(\mathfrak a)
	=
	\frac{\vartheta}{2}
	V_{\mathfrak q}\log X
	\sum_{ 
			(\mathfrak d,\mathfrak q)=1} 
	\lambda_{\mathfrak d}^{+}\chi(\mathfrak d)
	\frac1{|\mathcal A(X;\mathfrak q)|}
	\sum_{N\mathfrak m\le X/N\mathfrak d}\chi(\mathfrak m).
	\]
	  Since \(N\mathfrak d\le D=X^\vartheta\), we have
	$
	\frac{X}{N\mathfrak d}
	\ge
	X^{1-\vartheta}
	=
	X^\varepsilon Q^\alpha \ge Q^{\alpha+\varepsilon_1} 
	$
  for some  \(\varepsilon_1>0\)  depending only on
	\(\alpha\) and \(\varepsilon\).
	Thus, by \eqref{CS}, 
 there exists
	\(\eta_1>0\) such that
	\[
	\sum_{N\mathfrak m\le X/N\mathfrak d}\chi(\mathfrak m)
	\ll_{K,\varepsilon}
	\Bigl(\frac{X}{N\mathfrak d}\Bigr)^{1-\eta_1}.
	\]
	
	Using \(|\lambda_{\mathfrak d}^{+}|\le1\) and \eqref{5.2}, 
	we obtain
	\[
	\begin{aligned}
		|\widehat \nu(\chi)|
		&\ll_{K,\varepsilon}
		\log X
		\sum_{N\mathfrak d\le D}
		\frac{1}{X}
		\Bigl(\frac{X}{N\mathfrak d}\Bigr)^{1-\eta_1} =
		\log X\,
		X^{-\eta_1}
		\sum_{N\mathfrak d\le D}
		\frac1{(N\mathfrak d)^{1-\eta_1}}.
	\end{aligned}
	\]
	The elementary  estimate
	$
	\sum_{N\mathfrak d\le D}
	(N\mathfrak d)^{-1+\eta_1}
	\ll_{K,\eta_1}D^{\eta_1}
	$
	gives
	\[
	|\widehat\nu(\chi)|
	\ll_{K,\varepsilon}
	(\log X)\,X^{-\eta_1}D^{\eta_1}
	\ll_{K,\varepsilon}
	(\log Q)Q^{-\eta_1(\alpha+\varepsilon_1)}
	\ll_{K, \varepsilon}
	Q^{-\eta_2}
	\]
	for some \(\eta_2>0\). Therefore
	\[
	\max_{\chi\ne\chi_0}
	|\widehat\nu(\chi)|
	\ll_{K, \varepsilon}
	Q^{-\eta_2}.
	\]
 \medskip
 \noindent
 \textit{Verification of (A1) for the principal character.}
 We now estimate the average of the majorant \(\nu\). From \eqref{5.6},
 after interchanging the order of summation, we have
 \[
 \begin{aligned}
 	\mathbb E_{\mathfrak a\in\mathcal A(X;\mathfrak q)}
 	\nu(\mathfrak a)
 	&=
 	\frac{\vartheta}{2}
 	V_{\mathfrak q}\log X
 	\frac1{|\mathcal A(X;\mathfrak q)|}
 	\sum_{ 
 			(\mathfrak d,\mathfrak q)=1} 
 	\lambda_{\mathfrak d}^{+}
 	\#\Bigl\{
 	\mathfrak m:
 	N\mathfrak m\le \frac{X}{N\mathfrak d},\
 	(\mathfrak m,\mathfrak q)=1
 	\Bigr\}.
 \end{aligned}
 \]
 Applying
 Lemma~\ref{lem:ideal-counting-prime-to-q},  together with \eqref{5.2},   gives
 \begin{equation}\label{5.10}
 	\mathbb E_{\mathfrak a\in\mathcal A(X;\mathfrak q)}
 	\nu(\mathfrak a)
 	=
 	\frac{\vartheta}{2}
 	V_{\mathfrak q}\log X
 	\sum_{ 
 			(\mathfrak d,\mathfrak q)=1} 
 	\frac{\lambda_{\mathfrak d}^{+}}{N\mathfrak d}
 	+
 	o_K(1).
 \end{equation}
 
The upper-bound linear sieve Lemma \ref{sieve_lemma} with
 \[
 h(\mathfrak p)=
 \begin{cases}
 	1/N\mathfrak p, & \mathfrak p\nmid\mathfrak q,\\
 	0, & \mathfrak p\mid\mathfrak q,
 \end{cases}
 \]
 together with the prime-ideal Mertens theorem \eqref{merten} gives
 \begin{equation}\label{5.11}
 V_{\mathfrak q}
 \sum_{ 
 	(\mathfrak d,\mathfrak q)=1} 
 \frac{\lambda_{\mathfrak d}^{+}}{N\mathfrak d}
 \le
 \Bigl(\frac{2e^\gamma}{3}+o_K(1)\Bigr)
 \rho_K
 \prod_{\mathfrak p\mid\mathfrak q}
 \Bigl(1-\frac1{N\mathfrak p}\Bigr)
 \prod_{{N\mathfrak p<z\atop
 		\mathfrak p\nmid\mathfrak q}}
 \Bigl(1-\frac1{N\mathfrak p}\Bigr) \le
 \bigl(1+o_K(1)\bigr)
 \frac{2}{\vartheta\log X}.
 \end{equation}
 Substituting \eqref{5.11} into \eqref{5.10}, we obtain
 \[
 \mathbb E_{\mathfrak a}
 \nu(\mathfrak a)
 \le
 1+o_K(1).
 \]
 This verifies the principal-character part of (A1).
 
Therefore,  the hypotheses of
 Proposition~\ref{prop:dense-model-ray} are satisfied for \(Q\) sufficiently
 large. Hence we obtain a function
 $
 g:G\to[0,1+o_K(1)]
 $
 satisfying the conclusions of Proposition~\ref{prop:dense-model-ray}.
 
 By Proposition~\ref{prop:dense-model-ray}(iv),
 \[
 \mathbb E_{c\in G}g(c)
 =
 \mathbb E_{\mathfrak a\in\mathcal A(X;\mathfrak q)}
 f(\mathfrak a).
 \]
 Using the prime ideal theorem and \eqref{5.2}, we have
 \[
 	\mathbb E_{\mathfrak a}f(\mathfrak a)
  =
 \frac{\vartheta}{2}
 V_{\mathfrak q}\log X
 \cdot
 \frac{
 	\#\bigl\{\mathfrak p\nmid\mathfrak q:
 	z\le N\mathfrak p\le X\bigr\}
 }{
 	|\mathcal A(X;\mathfrak q)|
 }   = 
 \frac{\vartheta}{2}+o_K(1).
 \]
 Therefore
 \begin{equation}\label{5.13}
 	\mathbb E_{c\in G}g(c)
 	=
 	\frac{\vartheta}{2}+o_K(1).
 \end{equation}

	Moreover,
\begin{equation}\label{5.14}
		\frac1{|\mathcal A(X;\mathfrak q)|^2}
	\sum_{\mathfrak a\in\mathcal A(X;\mathfrak q)}
	f(\mathfrak a)^2 =
	\frac1{|\mathcal A(X;\mathfrak q)|^2}
	\Bigl(	\frac{\vartheta}{2}
	V_{\mathfrak q}\log X\Bigr)^2 	\#\bigl\{\mathfrak p\nmid\mathfrak q:
	z\le N\mathfrak p\le X\bigr\}
	\ll_K
	\frac{\log X}{X}.
\end{equation}
	
	Define
	\begin{equation}\label{def_A}
			A
		:=
		\Bigl\{
		c\in G:
		g(c)\ge \frac{\varepsilon}{10}
		\Bigr\}.
	\end{equation}
	
	We now prove the six conclusions.
	
	\medskip
	\noindent
	\textit{Proof of (i).}
	Since \(g(c)\le1+o_K(1)\),
	\[
	\mathbb E_{c\in G}g(c)
	\le
	\frac{\varepsilon}{10}
	+
	\frac{|A|}{|G|}(1+o_K(1)).
	\]
	Combining this with \eqref{5.13}, for $Q$ sufficiently large, 
	\[
	|A|
	\ge
	\Bigl(
	\frac{\vartheta}{2}-\varepsilon
	\Bigr)|G|.
	\]
	
	\medskip
	\noindent
	\textit{Proof of (iii).}
	For a character \(\chi\) of \(G\), recall that 
	\[
	\widehat f(\chi)
	:=
	\mathbb E_{\mathfrak a 	\in\mathcal A(X;\mathfrak q)}
	f(\mathfrak a )\chi([\mathfrak a]),
	\qquad
	\widehat g(\chi)
	:=
	\mathbb E_{c\in G}
	g(c)\chi(c).
	\]
	
	For \(c\in G\), define
	\[
	T_3(c)
	:=
	\frac1{|\mathcal A(X;\mathfrak q)|^3}
	\sum_{{
			\mathfrak a_1,\mathfrak a_2,\mathfrak a_3
			\in\mathcal A(X;\mathfrak q)\atop
			[\mathfrak a_1\mathfrak a_2\mathfrak a_3]=c}}
	f(\mathfrak a_1)f(\mathfrak a_2)f(\mathfrak a_3).
	\]
	If \(T_3(c)>0\), then \(c\in E_3(X;\mathfrak q)\).
	
	By orthogonality of characters, we have 
	\[
		T_3(c)
	-
	\frac1{|G|^3}(g*g*g)(c) = 
		\frac1{|G|}
	\sum_{\chi}
\bigl(\widehat f(\chi) ^3  - \widehat g(\chi) ^3  \bigr) 
 \overline{\chi(c)}.
	\]
	Using $|\widehat{g}(\chi)| \leq |\widehat{f}(\chi)|$, we have 
\begin{equation}
		\begin{aligned}\label{5.16}
		\biggl|T_3(c)
		-
		\frac1{|G|^3}(g*g*g)(c)\biggr|
		& \le 
		\frac3{|G|}
		\sum_{\chi}
		|\widehat f(\chi)|^2
		|\widehat f(\chi)-\widehat g(\chi)|
		 .
	\end{aligned}
\end{equation}
	
We show that the sum in \eqref{5.16} is \(O_{K,\varepsilon}(\delta)\). Put \[
\tau:=\frac{\delta}{(\log Q)^{n_K+2}}.
\]
Split the characters into
\[
\mathcal X_1
:=
\{\chi:\ |\widehat f(\chi)|\le\tau\},
\qquad
\mathcal X_2
:=
\{\chi:\ \tau<|\widehat f(\chi)|\le\delta\},
\qquad
\mathcal X_3
:=
\{\chi:\ |\widehat f(\chi)|>\delta\}.
\]

By the mean value theorem Lemma~\ref{lem:ray-class-mean-value-sharp},
together with \eqref{5.14}, we have
\begin{align*} 
	\sum_{\chi\in\widehat G}|\widehat f(\chi)|^2
	&=
	\frac1{|\mathcal A(X;\mathfrak q)|^2}
	\sum_{\chi\in\widehat G}
	\biggl|
	\sum_{\mathfrak a\in\mathcal A(X;\mathfrak q)}
	f(\mathfrak a)\chi([\mathfrak a])
	\biggr|^2 \nonumber \\
	&\ll_K
	(X+Q)(\log(3Q))^{n_K}
	\frac1{|\mathcal A(X;\mathfrak q)|^2}
	\sum_{\mathfrak a\in\mathcal A(X;\mathfrak q)}
	f(\mathfrak a)^2 \nonumber\\
	&\ll_K 
    \Bigl(1+\frac QX\Bigr)
    (\log(3Q))^{n_K}\log X.
\end{align*}
Since \(X\le Q^C\) and \(X\ge Q^{1+\kappa}\), this implies 
\begin{equation}\label{f_l2_mean_bound}
	\sum_{\chi\in\widehat G}|\widehat f(\chi)|^2
	\ll_{K }
	(\log Q)^{n_K+1}.
\end{equation}

For \(\mathcal X_1\), using
\(|\widehat f(\chi)-\widehat g(\chi)|
\le
|\widehat f(\chi)|\leq \tau \)  and \eqref{f_l2_mean_bound}, we obtain
\begin{equation}\label{5.17}
		\sum_{\chi\in\mathcal X_1}
	|\widehat f(\chi)|^2
	|\widehat f(\chi)-\widehat g(\chi)|
\le
	\tau\sum_{\chi\in\widehat G}|\widehat f(\chi)|^2 \ll_{K }
	\frac{\delta}{(\log Q)^{n_K+2}}
	(\log Q)^{n_K+1}
	\ll
	\delta.
\end{equation}

  For \(\mathcal X_2\), decompose dyadically according to
$
T<|\widehat f(\chi)|\le 2T$,
 $
\tau\le T\le\delta.
$
Let
\[
R(T)
:=
\#\{\chi:\ T<|\widehat f(\chi)|\le2T\}.
\]
Since \(\tau\) is  a negative
power of \(\log Q\), applying  Lemma~\ref{lem:ray-HM}(i)  as in the verification of (A2), we have 
$
R(T)\ll_{K,\varepsilon } T^{-2}
$
uniformly for \(\tau\le T\le\delta\). Hence
\begin{equation}\label{5.18}
		\sum_{\chi\in\mathcal X_2}
	|\widehat f(\chi)|^2
	|\widehat f(\chi)-\widehat g(\chi)|
	 \leq
	\sum_{\tau\le T\le\delta}
	R(T)T^3  \\
	 \ll_{K,\varepsilon }
	\sum_{\tau\le T\le\delta}T
	\ll
	\delta .
\end{equation}

For \(\mathcal X_3\), by Lemma~\ref{lem:ray-HM}(i), again applied as in the proof of \eqref{5.7} in the verification 
of (A2), we have $
|\mathcal X_3| \ll_{K,\varepsilon }\delta^{-2}  
$
 and \(\sum_{\chi\in  \mathcal X_3}|\widehat f(\chi)|^2
 \ll_{K, \varepsilon  } 1  \).  Moreover, Proposition~\ref{prop:dense-model-ray}(ii) gives
 \(|\widehat f(\chi)-\widehat g(\chi)|\le\delta.\)
Therefore
\begin{equation}\label{5.19}
	\sum_{\chi\in\mathcal X_3}
	|\widehat f(\chi)|^2
	|\widehat f(\chi)-\widehat g(\chi)|
	\ll_{K,\varepsilon }
	\delta .
\end{equation}

Combining \eqref{5.17}, \eqref{5.18}, and \eqref{5.19}, we obtain
\[
\sum_{\chi\in\widehat G}
|\widehat f(\chi)|^2
|\widehat f(\chi)-\widehat g(\chi)|
\ll_{K,\varepsilon }
\delta.
\]
Thus \eqref{5.16} gives
\begin{equation}\label{T_3_error}
	T_3(c)
	=
	\frac1{|G|^3}(g*g*g)(c)
	+
	O_{K,\varepsilon }\biggl(\frac{\delta}{|G|}\biggr).
\end{equation}
If
$
(1_A*1_A*1_A)(c)
\gg
\frac{|G|^2}{(\log Q)^{1/2-\varepsilon}},
$
then, as \(g\ge \varepsilon/10\) on \(A\), we have
\[
\frac1{|G|^3}(g*g*g)(c)
\ge \frac1{|G|^3}
\Bigl(\frac{\varepsilon}{10}\Bigr)^3
(1_A*1_A*1_A)(c)
\gg 
\frac{1}{|G|(\log Q)^{1/2-\varepsilon}} ,
\]
which dominates the error term in \eqref{T_3_error}  for \(Q\) sufficiently large,  since \(\frac{\delta}{|G|} =  \frac{1}{|G|(\log Q)^{1/2-\varepsilon/2}}. \)

Therefore \(T_3(c)>0\), and hence \(c\in E_3(X;\mathfrak q)\).
	
	\medskip
	\noindent
	\textit{Proof of (ii).}
	Define
	\[
	T_2(c)
	:=
	\frac1{|\mathcal A(X;\mathfrak q)|^2}
	\sum_{{
			\mathfrak a_1,\mathfrak a_2
			\in\mathcal A(X;\mathfrak q)\atop
			[\mathfrak a_1\mathfrak a_2]=c}}
	f(\mathfrak a_1)f(\mathfrak a_2).
	\]
	If \(T_2(c)>0\), then \(c\in E_2(X;\mathfrak q)\).
	
	Again by orthogonality of characters,
\[
T_2(c)
-
\frac1{|G|^2}(g*g)(c)
=
\frac1{|G|}
\sum_{\chi\in\widehat G}
\bigl(
\widehat f(\chi)^2-\widehat g(\chi)^2
\bigr)
\overline{\chi(c)}.
\]
Using $|\widehat{g}(\chi)| \leq |\widehat{f}(\chi)|$, we have  
\begin{equation}\label{5.21}
	\sum_{c\in G}
	\Bigl|
	T_2(c)
	-
	\frac1{|G|^2}(g*g)(c)
	\Bigr|^2
\le
	\frac4{|G|}
	\sum_\chi
	|\widehat f(\chi)|^2
	|\widehat f(\chi)-\widehat g(\chi)|^2.
\end{equation}
We estimate the sum in \eqref{5.21} using the same decomposition
\(\mathcal X_1,\mathcal X_2,\mathcal X_3\) as in the proof of (iii).

For \(\mathcal X_1\), using
\(|\widehat f(\chi)-\widehat g(\chi)|
\le
|\widehat f(\chi)|,\)
and \eqref{f_l2_mean_bound}, we obtain
\[	\sum_{\chi\in\mathcal X_1}
|\widehat f(\chi)|^2
|\widehat f(\chi)-\widehat g(\chi)|^2
 \le
\tau^2
\sum_{\chi\in\widehat G}|\widehat f(\chi)|^2  \\
 \ll_K
\frac{\delta^2}{(\log Q)^{2n_K+4}}
(\log Q)^{n_K+1}
\ll
\delta^2 .\]

For \(\mathcal X_2\), dyadically decompose according to
$
T<|\widehat f(\chi)|\le 2T$,
$
\tau\le T\le\delta.
$
As in the proof of (iii), Lemma~\ref{lem:ray-HM}(i) gives
\[
R(T):=
\#\{\chi:\ T<|\widehat f(\chi)|\le2T\}\ll_{K,\varepsilon} T^{-2}.
\]
Therefore
\[	\sum_{\chi\in\mathcal X_2}
|\widehat f(\chi)|^2
|\widehat f(\chi)-\widehat g(\chi)|^2
 \le
\sum_{\tau\le T\le\delta}
R(T)T^4  \\
 \ll_{K,\varepsilon}
\sum_{\tau\le T\le\delta}
T^2
\ll
\delta^2 .\]

For \(\mathcal X_3\),   similar as before,  Lemma~\ref{lem:ray-HM}(i) gives
$
|\mathcal X_3| \ll_{K,\varepsilon}\delta^{-2} 
$
and \(\sum_{\chi\in  \mathcal X_3}|\widehat f(\chi)|^2
\ll_{K, \varepsilon  } 1  \).
Hence
\[
\sum_{\chi\in\mathcal X_3}
|\widehat f(\chi)|^2
|\widehat f(\chi)-\widehat g(\chi)|^2
\ll_{K, \varepsilon  } 
\delta^2.
\]

Combining the three estimates, we get
\[
\sum_{\chi\in\widehat G}
|\widehat f(\chi)|^2
|\widehat f(\chi)-\widehat g(\chi)|^2
\ll_{K, \varepsilon  } 
\delta^2.
\]
Therefore \eqref{5.21} gives
\[\sum_{c\in G}
\Bigl|
T_2(c)
-
\frac1{|G|^2}(g*g)(c)
\Bigr|^2
\ll_{K, \varepsilon  } 
\frac{\delta^2}{|G|}.\]
Let
\[
\lambda:=
\frac{1}{|G|(\log Q)^{1/2-3\varepsilon/4}}.
\]
By Chebyshev's inequality,
\[
\#\Bigl\{c\in G:\ \Bigl|
T_2(c)
-
\frac1{|G|^2}(g*g)(c)
\Bigr|>\lambda\Bigr\}
\le
\lambda^{-2}
\sum_{c\in G}\Bigl|
T_2(c)
-
\frac1{|G|^2}(g*g)(c)
\Bigr|^2 \ll_{K, \varepsilon  }  |G|(\log Q)^{-\varepsilon/2}.
\]
Now suppose that \(c\) is outside this exceptional set of size \(O_{K, \varepsilon  } \bigl(|G|(\log Q)^{-\varepsilon /2}\bigr), \) and that
$
(1_A*1_A)(c)
\gg
\frac{|G|}{(\log Q)^{1/2-\varepsilon}}.
$
Then \[T_2(c)
=
\frac1{|G|^2}(g*g)(c)
+
O\Bigl(
\frac{1}{|G|(\log Q)^{1/2-3\varepsilon/4}}
\Bigr).\]
Since \(g\ge \varepsilon/10\) on \(A\), we have
\[
(g*g)(c)
\ge
\Bigl(\frac{\varepsilon}{10}\Bigr)^2
(1_A*1_A)(c)
\gg
\frac{|G|}{(\log Q)^{1/2-\varepsilon}},
\]
which dominates the error term. 
Therefore, for $Q$ sufficiently large,   \(T_2(c)>0\), and hence  \(c\in E_2(X;\mathfrak q)\).

	\medskip
	\noindent
\textit{Proof of (iv).}
Assume that
\(B:=\frac{\log X}{\log Q}\ge 2\alpha+\kappa.\)
Let \(\beta\in(2\varepsilon,B],\)
\(L=\bigl\lfloor\frac{B}{\beta}\bigr\rfloor.\)
Define
\[
f_0(\mathfrak a)
:=
\frac{
	1_{\mathfrak a=\mathfrak p}
	1_{Q^{\beta-\varepsilon}<N\mathfrak p\le Q^\beta}
}{N\mathfrak p},
\]
\[
(f*f*f_0)(c)
:=
\sum_{{
		\mathfrak a_1,\mathfrak a_2,\mathfrak a_3\in\mathcal A(X;\mathfrak q)\atop
		[\mathfrak a_1\mathfrak a_2\mathfrak a_3]=c}}
f(\mathfrak a_1)f(\mathfrak a_2)f_0(\mathfrak a_3),
\]
\[
(g*g*f_0)(c)
:=
\sum_{{
		a_1,a_2\in G,\ \mathfrak a_3\in\mathcal A(X;\mathfrak q)\atop
		a_1a_2[\mathfrak a_3]=c}}
g(a_1)g(a_2)f_0(\mathfrak a_3).
\]
If \((f*f*f_0)(c)>0,\)
then \(c\in E_3(X;\mathfrak q)\).

By orthogonality of characters  and that \(|\widehat g(\chi)|\leq |\widehat f(\chi)| \), we have
\begin{equation}\label{5.37}
	\biggl|\frac1{|\mathcal A(X;\mathfrak q)|^2}
	(f*f*f_0)(c)
	-
	\frac1{|G|^2}
	(g*g*f_0)(c) \biggr|
	\leq  
	\frac2{|G|}
	\sum_{\chi\in\widehat G}
	|\widehat f(\chi)|
	|\widehat f_0(\chi)|
	|\widehat f(\chi)-\widehat g(\chi)|
	.
\end{equation}
We split the characters into
\[
\mathcal X_1
:=
\Bigl\{
\chi\in\widehat G:
|\widehat f_0(\chi)|
<(\log Q)^{-(n_K+2)}
\Bigr\}, \qquad  \mathcal X_2:=\widehat G\setminus\mathcal X_1.
\]

For \(\chi\in\mathcal X_1\), using \(|\widehat f(\chi) -\widehat g(\chi)|\leq |\widehat f(\chi)| \) and \eqref{f_l2_mean_bound}, we have 
\begin{equation}\label{5.36}
	\frac1{|G|}
	\sum_{\chi\in\mathcal X_1}
	|\widehat f(\chi)|
	|\widehat f_0(\chi)|
	|\widehat f(\chi)-\widehat g(\chi)|\leq  \frac{1}{ |G|( \log Q)^{n_K+2 } } 	\sum_{\chi\in\widehat G}|\widehat f(\chi)|^2 \ll_K \frac{1}{ |G| \log Q }
\end{equation}
 
We now treat \(\mathcal X_2\).
By Proposition~\ref{prop:dense-model-ray}(ii),(iii), \(|\widehat f(\chi) -\widehat g(\chi)|\leq |\widehat f(\chi)| \) and \(|\widehat f(\chi) -\widehat g(\chi)|\leq  \delta \), together with H\"older's inequality, we have 
\begin{equation}\label{5.38}
	\frac1{|G|}
	\sum_{\chi\in\mathcal X_2}
	|\widehat f(\chi)|
	|\widehat f_0(\chi)|
	|\widehat f(\chi)-\widehat g(\chi)|\leq  	\frac{\delta^{\frac1L}}{|G|}
	\biggl(
	\sum_{\chi\in\mathcal X_2}
	|\widehat f_0(\chi)|^{2L}
	\biggr)^{\frac1{2L}}
\biggl(
	\sum_{\chi\in\mathcal X_2}
	|\widehat f(\chi)|^2
	\biggr)^{1-\frac1{2L}}.
\end{equation}

We first estimate the \(2L\)-th moment.
Expanding the \(L\)-th power,
\[
\bigl(
\widehat f_0(\chi)
\bigr)^L = 
\Bigl(
\sum_{\mathfrak a}
f_0(\mathfrak a)\chi(\mathfrak a)
\Bigr)^L
=
\sum_{\mathfrak n}
b_{\mathfrak n}\chi(\mathfrak n),
\]
where \(b_{\mathfrak n}\) is supported on ideals
\(\mathfrak n=\mathfrak p_1\cdots\mathfrak p_L\) with
\(Q^{\beta-\varepsilon}<N\mathfrak p_j\le Q^\beta,\)
and satisfies
\(|b_{\mathfrak n}|
\le
\frac{\tau_L(\mathfrak n)}{N\mathfrak n}.\)
Thus
\[Q^{L(\beta-\varepsilon)}
<
N\mathfrak n
\le
Q^{L\beta}
\le
Q^B
=
X.\]

We next verify that the lower endpoint lies in the range required for
Lemma~\ref{lem:ray-HM}(ii).
We claim that
\begin{equation}\label{claim}
	L(\beta-\varepsilon)>\frac B2-\varepsilon.
\end{equation}
If \(L=1\), then \(\beta>B/2\), and the claim is immediate. Suppose now
that \(L\ge2\). Since \(L=\lfloor B/\beta\rfloor\), we have
\[
\frac{B}{L+1}<\beta\le \frac BL.
\]
If \(B/(L+1)\ge2\varepsilon\), then
\[
L(\beta-\varepsilon)>
\frac{LB}{L+1}-L\varepsilon
\ge
\frac B2-\varepsilon.
\]
If \(B/(L+1)<2\varepsilon\), then, since \(\beta>2\varepsilon\),
\[
L(\beta-\varepsilon)>L\varepsilon>
\frac{LB}{2(L+1)}
=
\frac B2-\frac{B}{2(L+1)}
>
\frac B2-\varepsilon.
\]
This proves \eqref{claim}.

Because \(B\ge2\alpha+\kappa\), we may choose \(\varepsilon>0\) sufficiently
small so that  
$
L(\beta-\varepsilon)\ge \alpha+3\varepsilon.
$
Moreover, since \(\beta>2\varepsilon\), every prime factor of \(\mathfrak n\) has norm \(>Q^\varepsilon\). Hence
\((\mathfrak n,P(Q^\varepsilon))=1.\)
Therefore Lemma~\ref{lem:ray-HM}(ii) applies to the coefficients
\(a_{\mathfrak n}:=N\mathfrak n\,b_{\mathfrak n},\)
and gives that for some \(c_0=c_0(\alpha,\varepsilon,\kappa)>0\), 
\[
\sum_{\chi\in\mathcal X_2}
\bigl|
\widehat f_0(\chi)
\bigr|^{2L}
 \ll_{K,\varepsilon}
\bigl(
1+
|\mathcal X_2|
Q^{-c_0}
\bigr)
\sum_{\mathfrak n}
\frac{|a_{\mathfrak n}|^2}{N\mathfrak n}.
\]
Since
\(|a_{\mathfrak n}|
\le
\tau_L(\mathfrak n),\) by Mertens' theorem for
prime ideals \eqref{merten}, 
we have
\[
\sum_{\mathfrak n}
\frac{|a_{\mathfrak n}|^2}{N\mathfrak n}
\ll
\prod_{Q^{\beta-\varepsilon}<N\mathfrak p\le Q^\beta}
\Bigl(
1+\frac{O_L(1)}{N\mathfrak p}
\Bigr)
\ll_{K,\varepsilon} 1.
\]
Hence
\begin{equation}\label{5.41}
		\sum_{\chi\in\mathcal X_2}
	\bigl|
	\widehat f_0(\chi)
	\bigr|^{2L}
	\ll_{K,\varepsilon}
	1+|\mathcal X_2|Q^{-c_0}.
\end{equation}
On the other hand, by the definition of \(\mathcal X_2\),
\[
(\log Q)^{-2L(n_K+2)}|\mathcal X_2|
\le
\sum_{\chi\in\mathcal X_2}
\bigl|
\widehat f_0(\chi)
\bigr|^{2L}.
\]
Combining this with \eqref{5.41} and taking \(Q\) sufficiently large, we obtain
\begin{equation}\label{bound_X_2}
	|\mathcal X_2|
	\ll_{K,\varepsilon}
	(\log Q)^{2L(n_K+2)}.
\end{equation}
Applying this bound \eqref{bound_X_2} for $|\mathcal{X}_2|$ back to  \eqref{5.41}, we have 
\begin{equation}\label{5.42}
	\sum_{\chi\in\mathcal X_2}
	\bigl|
	\widehat f_0(\chi)
	\bigr|^{2L}
	\ll_{K,\varepsilon} 1.
\end{equation}
On the other hand, Lemma~\ref{lem:ray-HM}(i) gives for some $\eta=\eta(\varepsilon)>0$, 
\[
\sum_{\chi\in\mathcal X_2}
\bigl|\widehat f(\chi)\bigr|^2
\ll
\Bigl(
\frac{X}{\log Q}
+
|\mathcal X_2|X^{1-\eta}Q^{\varepsilon\eta}
\Bigr)
\frac1{|\mathcal A(X;\mathfrak q)|^2}
\sum_{\mathfrak a\in\mathcal A(X;\mathfrak q)}
f(\mathfrak a)^2.
\]
Using \eqref{5.14} and \eqref{bound_X_2} and taking $\varepsilon$ sufficiently small, we have 
\begin{equation}\label{5.43}
	\sum_{\chi\in\mathcal X_2}
\bigl|\widehat f(\chi)\bigr|^2
	\ll_{K,\varepsilon}
	\frac{\log X}{\log Q}
	+
	|\mathcal X_2|X^{-\eta}Q^{\varepsilon\eta}\log X
	\ll 
	1+
	|\mathcal X_2|X^{-\eta}Q^{\varepsilon\eta}\log Q\ll_{\varepsilon}  1 .
\end{equation} 

Substituting \eqref{5.42} and \eqref{5.43} into \eqref{5.38}, we obtain
\[
	\frac1{|G|}
\sum_{\chi\in\mathcal X_2}
\bigl|\widehat f(\chi)\bigr|
\bigl|\widehat f_0(\chi)\bigr|
\bigl|\widehat f(\chi)-\widehat g(\chi)\bigr|
\ll_{K,\varepsilon}
\frac{\delta^{1/L}}{|G|}.
\]

Together with \eqref{5.37} and \eqref{5.36},
\begin{equation}\label{5.44}
	\frac1{|\mathcal A(X;\mathfrak q)|^2}
	(f*f*f_0)(c)
	=
	\frac1{|G|^2}
	(g*g*f_0)(c)
	+
	O_K\biggl(
	\frac{\delta^{1/L}}{|G|}
	\biggr).
\end{equation}
Recall that
\(	\delta
=
(\log Q)^{-1/2+\varepsilon/2},\)
so
\[ \frac{\delta^{1/L}}{|G|}
=\frac{1}
{|G|(\log Q)^{(1- \varepsilon)/(2L)}}.\]
Since \(g\ge\varepsilon/10\) on \(A\),
\[
\frac1{|G|^2}(g*g*f_0)(c)
\gg_\varepsilon \frac1{|G|^2}
\sum_{{
		[\mathfrak p]a_1a_2=c\atop
		Q^{\beta-\varepsilon}<N\mathfrak p\le Q^\beta\\
		a_1,a_2\in A}}
\frac1{N\mathfrak p} \gg
\frac{1}
{|G|(\log Q)^{(1-2\varepsilon)/(2L)}}.
\]
 which  dominates the error term in  \eqref{5.44}.  Hence
\((f*f*f_0)(c)>0,\)
and therefore
\(c\in E_3(X;\mathfrak q).\)

\medskip
\noindent
\textit{Proof of (v).}
Let \(H\le G\) have index \(Y<\varepsilon^{-1/2}\). Let
\(b_1H,\dots,b_KH\)
be the cosets of \(H\) such that \(A\) contains more than
\(\varepsilon |G|\) elements from \(b_jH\). Let the remaining cosets be
\(c_1H,\dots,c_{Y-K}H.\)

For \(j\le K\), we first estimate the \(f\)-mass in the coset \(b_jH\).
By the definition of \(f\),
\begin{equation}\label{E_a_f_coset}
	\mathbb E_{\mathfrak a}
	f(\mathfrak a)1_{[\mathfrak a]\in b_jH}
	=
	\frac{\vartheta}{2}V_{\mathfrak q}\log X
	\frac{
		\#\{\mathfrak p:\ X^{\vartheta/3}\le N\mathfrak p\le X,\ [\mathfrak p]\in b_jH\}
	}{
		|\mathcal A(X;\mathfrak q)|
	}.
\end{equation}
Since \(\vartheta\le \vartheta_0\), the prime ideal theorem gives
$
\#\{\mathfrak p:  X^{\vartheta/3}\le N\mathfrak p<X^{\vartheta_0/3}\}
=o_K(X/\log X).
$
Applying the upper bound in Lemma \ref{rough_ideal_lemma} with $\gamma=\vartheta_0/3$  gives, uniformly in the coset \(b_jH\),
\[
	\#\{\mathfrak p:\ X^{\vartheta_0/3}\le N\mathfrak p\le X,\ [\mathfrak p]\in b_jH\}
\le
\mathcal{N}_{\vartheta_0/3}(X;b_j,H)  \le
(1+o_{K,\varepsilon}(1))
\frac{2}{Y\vartheta_0}
\frac{X}{\log X}.
\]
Therefore, combining these together with \eqref{5.2} in \eqref{E_a_f_coset},
we conclude that
\[
\mathbb E_{\mathfrak a}
f(\mathfrak a)1_{[\mathfrak a]\in b_jH}
\le
(1+o_{K,\varepsilon}(1))\frac{\vartheta}{Y\vartheta_0}.
\]
By Proposition~\ref{prop:dense-model-ray}(v),
\begin{equation}\label{approximation}
	\mathbb E_{c\in G}g(c)1_{c\in b_jH}
	=
	\mathbb E_{\mathfrak a}
	f(\mathfrak a)1_{[\mathfrak a]\in b_jH}
	+
	O(\delta),
\end{equation}
therefore 
\begin{equation}\label{5.8}
	\frac1{|G|}\sum_{a\in b_jH}g(a) =\mathbb E_{c\in G}g(c)1_{c\in b_jH}\le
	(1+o_{K,\varepsilon}(1))\frac{\vartheta}{Y\vartheta_0}.
\end{equation}

For \(j>K\), we use \(g\le 1+o_K(1)\) on \(A\) and \(g<\varepsilon/10\) on \(G\setminus A\), obtaining
\[	
\frac1{|G|}\sum_{a\in c_jH}g(a)
\le
(1+o_K(1))\frac{|A\cap c_jH|}{|G|}
+
\frac{\varepsilon}{10}\frac{|c_jH|}{|G|}  \le
(1+o_K(1))\varepsilon+\frac{\varepsilon}{10Y}.
\]
Since \(Y<\varepsilon^{-1/2}\), we have
\(\varepsilon\le \varepsilon^{1/2}/Y\). Hence, for \(Q\) sufficiently large,
\begin{equation}\label{5.9}
	\frac1{|G|}\sum_{a\in c_jH}g(a)
	\le
	\frac{2\varepsilon^{1/2}}{Y}.
\end{equation}
Summing \eqref{5.8} and \eqref{5.9} over all cosets and using
\eqref{5.13}, we get
\[
\frac{\vartheta}{2} +o_K(1) =\mathbb E_{c\in G}g(c)
\le
(1+o_K(1))
K\frac{\vartheta}{Y\vartheta_0}
+
2\varepsilon^{1/2}.
\]
Rearranging, for $Q$ sufficiently large, 
\[
K
\ge
\Bigl(
\frac{\vartheta_0}{2}
-
3\varepsilon^{1/2}
\frac{\vartheta_0}{\vartheta}
\Bigr)Y.
\]	
	\medskip
	\noindent
	\textit{Proof of (vi).}
	Since the primes with \(N\mathfrak p<X^{\vartheta_0/3}\) contribute
	\(o_K(X/\log X)\), by \eqref{E_a_f_coset} and \eqref{5.2}, we have
\begin{equation}\label{5.29}
		\mathbb E_{\mathfrak a}
	f(\mathfrak a)1_{[\mathfrak a]\in bH}
	\ge
	\frac{\vartheta}{2}
	\frac{\pi(X;bH)}{X/\log X}
	+o_K(1).
\end{equation}

 On the other hand,   using \(g\le 1+o_K(1)\) on \(A\) and \(g<\varepsilon/10\) on \(G\setminus A\), we have 
 \begin{equation}\label{bound_g_sum_coset}
 	\frac1{|G|}\sum_{a\in bH}g(a)
 	\le
 	(1+o_K(1))\frac{|A\cap bH|}{|G|}
 	+
 	\frac{\varepsilon}{10Y} 
 \end{equation}
	Combining \eqref{approximation}, \eqref{bound_g_sum_coset}, and \eqref{5.29}, we obtain that for $Q$ sufficiently large, 
	\[
	|A\cap bH|
	\ge
\Bigl(
\frac{\vartheta}{2}
\frac{\pi(X;bH)}{X/\log X}
-\frac{\varepsilon}{5Y}
\Bigr)|G|.
	\]	
	This proves (vi), and completes the proof.
\end{proof}
\section{Product sets}
This section converts the dense-model frameworks from
Proposition~\ref{prop:transference-conclusion} into product-set information
inside the ray class group. Apart from the transference inputs proved in the
previous section, the argument is purely finite-abelian-group theoretic and
closely follows the product-set argument of Matomäki--Teräväinen
\cite{matomaki}.

Throughout this section, let
$
G=\operatorname{Cl}_{\mathfrak q}^{(\infty)}$,
$
Q=N\mathfrak q.
$
For subsets \(A,B\subseteq G\), we write
\[
A\cdot B=\{ab:a\in A,\ b\in B\},
\]
and for functions \(f,g:G\to\mathbb C\), we write
\[
(f*g)(c)
:=
\sum_{{a,b\in G\atop ab=c}}
f(a)g(b).
\]

We shall use the following elementary lower bound for convolutions of
characteristic functions on finite abelian groups, which appears as
\cite[Lem.~3.4]{matomaki}. It may be viewed as a simple consequence of
inclusion--exclusion and is useful for obtaining lower bounds on product
sets inside cosets.

\begin{lemma}\cite[Lem.~3.4]{matomaki}\label{lem:convolution-product-lower-bound}
	Let $G$ be a finite abelian group.
	
	\begin{enumerate}
		\item[(i)]
		Let $A,B\subseteq G$ be nonempty subsets. Then, for every
		$c\in G$,
		\[
		1_A*1_B(c)
		\geq
		|A|+|B|-|G|.
		\]
		
		\item[(ii)]
		Let $H\leq G$, let $a,b\in G$, and let
		$A\subseteq aH$ and $B\subseteq bH$. Then, for every
		$c\in abH$,
		\[
		1_A*1_B(c)
		\geq
		|A|+|B|-|H|.
		\]
	\end{enumerate}
\end{lemma}
We shall also use Kneser’s theorem; see, for example,
\cite[Thm.~5.5]{taovu}.

\begin{lemma}[Kneser’s theorem]\label{lem:kneser}
	Let $G$ be a finite abelian group, and let $A,B\subseteq G$.
	Let $H$ be the stabilizer of $A\cdot B$. Then
	\[
	|A\cdot B|
	\geq
	|A\cdot H|+|B\cdot H|-|H|
	\geq
	|A|+|B|-|H|.
	\]
\end{lemma}
\subsection{Reduction from popular products to product sets}
 We first use a popular-products lemma to reduce the problem to two cases where either many elements of the group $G$ are represented in $A\cdot B$,
 or after deleting few elements one obtains a product set $A'\cdot B' $ whose
 elements all have many representations. We then use Kneser's theorem to analyze the latter case. 
\begin{lemma}\cite[Lem.~6.1]{matomaki}
	\label{lem:popular-products}
	Let \(t\ge u\ge1\) be integers. Let \(A,B\) be finite subsets of a finite
	abelian group \(G\), with
	\(	|A|,|B|\ge t.\)
	Then at least one of the following holds.
	
	\begin{enumerate}
		\item[(a)]
		
		For at least
		\[
		|A|+|B|-2t-\frac{u|G|}{t}
		\]
		elements \(c\in G\), we have
		\[
		(1_A*1_B)(c)\ge u.
		\]
		
		\item[(b)]
		There exist subsets \(A'\subseteq A\) and \(B'\subseteq B\) such that
		\[
		|A\setminus A'|+|B\setminus B'|\le t-1,
		\]
		and for every \(c\in A'\cdot B'\),
		\[
		(1_A*1_B)(c)\ge t.
		\]
		
	\end{enumerate}
\end{lemma}
\begin{proof}
See \cite[Lem.~6.1]{matomaki}.
\end{proof}
We now combine the dense model set \(A\) with the popular-products lemma.
The following proposition packages the criteria that will be used in the
proofs of the two-prime and three-prime theorems.
\begin{proposition}[Criteria on \(E_2(X;\mathfrak q)\) and \(E_3(X;\mathfrak q)\)]
	\label{prop:criteria-E2-E3}
 Let \(\kappa,\varepsilon>0\) and \(C\ge1\) be fixed, with
	\(\varepsilon>0\) sufficiently small.  
	Assume
	\[
	X\in
	 \bigl[
	Q^{\max( 1, 2\alpha, 3\alpha_0)+\kappa},
	Q^C
 \bigr].
	\]
	Let \(\vartheta\) and \(\vartheta_0\) be as in \eqref{5.1}.
	Then, when $Q$ is sufficiently large, at least one of the following holds.
	
	\begin{enumerate}
		\item[(a)]
		We have
		\[
		|E_2(X;\mathfrak q)|
		\ge
		(\vartheta-3\varepsilon)|G|.
		\]
		If moreover
		\(	\vartheta>\frac23+3\varepsilon,\)
		then
		\[
		E_3(X;\mathfrak q)=G.
		\]
		
		\item[(b)]
		There exist subsets \(A',B'\subseteq G\) such that the following hold.
		
		\begin{enumerate}
			\item[(b.i)]
			$
			|A'|,\ |B'|
			\ge
			\bigl(
			\frac{\vartheta}{2}-\frac{3\varepsilon}{2}
			\bigr)|G|.
			$
			
			\item[(b.ii)]
			$
			|(A'\cdot B')\cap E_2(X;\mathfrak q)|
			\ge
			|A'\cdot B'|-\varepsilon|G|.
			$
			
			\item[(b.iii)]  
			For every \(c\in G\),
			\[
			(1_{A'\cdot B'}*1_{A'})(c)\gg |G|
			\quad\Longrightarrow\quad
			c\in E_3(X;\mathfrak q).
			\]
			
			\item[(b.iv)]
				Let
			\(B:=\frac{\log X}{\log Q}  \) and
		 \(	\beta\in(2\varepsilon,B]\). 
			For every \(c\in G\),
			\[
		\sum_{[\mathfrak p]\,d=c,\ d\in A'\cdot B'
			\atop
			Q^{\beta-\varepsilon}<N\mathfrak p\le Q^\beta}
		\frac1{N\mathfrak p}
			\gg 1
			\quad\Longrightarrow\quad
			c\in E_3(X;\mathfrak q).
			\]
			
			\item[(b.v)]
			For any subgroup \(H\le G\) of index \(Y\le\varepsilon^{-1/2}\), there are
			at least
			$
			\Bigl\lceil
			\bigl(
			\frac{\vartheta_0}{2}
			-
			3\varepsilon^{1/2}
			\frac{\vartheta_0}{\vartheta}
			\bigr)Y
			\Bigr\rceil
			$
			distinct cosets \(bH\) of \(H\) such that
			\[
			|A'\cap B'\cap bH|\ge \frac{\varepsilon}{2}|G|.
			\]
			
			\item[(b.vi)]
			For any subgroup \(H\le G\) of index \(Y\le\varepsilon^{-1/2}\) and any
			coset \(bH\subseteq G\),
			\[
		|A'\cap B'\cap bH|
		\ge
		\Bigl(
		\frac{\vartheta}{2}
		\frac{\pi(X;bH)}{X/\log X}
		-\frac{3\varepsilon}{10}
		\Bigr)|G|.
			\]	where \(	\pi (X;bH)\) is defined in \eqref{def_pi}. 
		\end{enumerate}
	\end{enumerate}
\end{proposition}
\begin{proof}
	We may assume that \(\varepsilon>0\) is sufficiently small. Note that the condition
	\(X\in
	 \bigl[
	Q^{\max(1,2\alpha,3\alpha_0)+\kappa},
	Q^C
	 \bigr]\)
	ensures that the conditions on $X$ in each part of
	Proposition~\ref{prop:transference-conclusion} are satisfied.
	
	Let \(A\subseteq G\) be the set supplied by
	Proposition~\ref{prop:transference-conclusion}. Thus
	\begin{equation}	\label{6.2}
		|A|
		\ge
		\Bigl(
		\frac{\vartheta}{2}-\varepsilon
		\Bigr)|G|.
	\end{equation}

 Let   $M_{\varepsilon}\geq C_0\varepsilon^{-2}$ with $C_0>0$ chosen sufficiently large to absorb all rounding errors when applying Lemma \ref{lem:popular-products} later . For example, we may take $C_0= 2000$.   We first handle  the case \(|G|<M_\varepsilon\). Since \(|G|\) is then bounded in terms
 of \(\varepsilon\), the logarithmic-saving thresholds and exceptional sets
 in Proposition~\ref{prop:transference-conclusion} are, for \(Q\) sufficiently
 large, smaller than \(1\). Thus every element of \(A\cdot A\) belongs to
 \(E_2(X;\mathfrak q)\), and every element of \(A\cdot A\cdot A\) belongs
 to \(E_3(X;\mathfrak q)\). Taking \(A'=B'=A\), conclusions
 \((b.i)\)--\((b.vi)\) follow directly from
 Proposition~\ref{prop:transference-conclusion}. 
 
 Hence we may assume
 \(|G|\ge M_\varepsilon\).
 Let 
 \[
 t=
 \biggl\lceil
 \frac{\varepsilon|G|}{10}
 \biggr\rceil,
 \qquad
 u=
 \biggl\lfloor
 \frac{\varepsilon^2|G|}{1000}
 \biggr\rfloor.
 \]
 Then $ u\ge \frac{\varepsilon^2 |G|}{2000}\ge 1  $  and $t\ge u$.
 As \(\varepsilon\) is sufficiently small, \eqref{6.2} gives \(|A|\ge t\).
 	We now apply Lemma~\ref{lem:popular-products} with $ t, u$ and both sets   equal to \(A\). 
	
	\medskip
	\noindent
	\textit{Case 1: Lemma~\ref{lem:popular-products}(a) holds.}
	
		Then the set
	\[
	\mathcal C
	:=
	\bigl\{
	c\in G:
	(1_A*1_A)(c)\ge u
	\bigr\}
	\]
	has size 
	\[
		|\mathcal C|\ge 	2|A|-2t-\frac{u|G|}{t}
	>
	\Bigl(
	\vartheta-\frac 52\varepsilon
	\Bigr)|G| .
	\]

	By Proposition~\ref{prop:transference-conclusion}(ii), all but
	\(O(|G|(\log Q)^{-\varepsilon/2})\) of these elements in \(\mathcal C\) belong to
	\(E_2(X;\mathfrak q)\). Absorbing this negligible exceptional set into the
	\(\varepsilon|G|\)-term, we obtain
	\[
	|E_2(X;\mathfrak q)|
	\ge
	(\vartheta-3\varepsilon)|G|.
	\]
	This proves the first assertion of~(a).
	
	Assume now that
\(	\vartheta>\frac23+3\varepsilon.\)
	Then 
	\[
	|\mathcal C|
	\ge
	\Bigl(
	\vartheta-\frac{5}{2}\varepsilon
	\Bigr)|G|
	>
	\Bigl(
	\frac23+\frac{\varepsilon}{2}
	\Bigr)|G|,
	\qquad
	|A|
	\ge
	\Bigl(
	\frac{\vartheta}{2}-\varepsilon
	\Bigr)|G|
	>
	\Bigl(
	\frac13+\frac{\varepsilon}{2}
	\Bigr)|G|.
	\]

	Hence
	$
	|\mathcal C|+|A|>(1+\varepsilon )|G|.
	$
	By Lemma~\ref{lem:convolution-product-lower-bound}(i), for every
	\(c\in G\),
	\[
	(1_{\mathcal C}*1_A)(c)
	\ge
	|\mathcal C|+|A|-|G|
	> \varepsilon|G|.
	\]
	But
	\[
(1_A*1_A*1_A)(c)
 =
\sum_{d\in G}
(1_A*1_A)(d)\,1_A(cd^{-1})
 \ge
u\,
(1_{\mathcal C}*1_A)(c)
 \ge
\frac{\varepsilon^3 }{2000}|G|^2.
	\]
	Therefore Proposition~\ref{prop:transference-conclusion}(iii) implies that  
\(	c\in E_3(X;\mathfrak q)\) for every \(c\in G\).
	 Hence
\(E_3(X;\mathfrak q)=G.\)
	This completes the proof of~(a).
	
	\medskip
	\noindent
	\textit{Case 2: Lemma~\ref{lem:popular-products}(b) holds.}
	
	Then there exist \(A',B'\subseteq A\) such that
	\begin{equation}	\label{6.3}
		|A\setminus A'|+|A\setminus B'|
	\le
	t-1
	\le
	\frac{\varepsilon|G|}{10}.
	\end{equation}

	We now verify the six claims in~(b).
	
	\medskip
	\noindent
	\textit{Proof of (b.i).}
	Since \(A',B'\subseteq A\),
	\[
	|A'|
	 \ge
	|A|-\frac{\varepsilon|G|}{10},\qquad  	|B'|
	 \ge
	|A|-\frac{\varepsilon|G|}{10}.
	\]
	Using \eqref{6.2}, we get
	\[
	|A'|,\ |B'|
	\ge
	\Bigl(
	\frac{\vartheta}{2}-\frac{3\varepsilon}{2}
	\Bigr)|G|.
	\]
	
	\medskip
	\noindent
	\textit{Proof of (b.ii).}
	By Lemma~\ref{lem:popular-products}(b), for every
	\(d\in A'\cdot B'\),
\begin{equation}\label{6.4}
	(1_A*1_A)(d)\ge t\gg |G|.
\end{equation}
	Hence Proposition~\ref{prop:transference-conclusion}(ii) implies that,
	apart from at most
\(	O(|G|(\log Q)^{-\varepsilon/2})\)
	exceptions, every element of \(A'\cdot B'\) belongs to
	\(E_2(X;\mathfrak q)\). Since the exceptional set is
	\(o(\varepsilon|G|)\), we get
	\[
	|(A'\cdot B')\cap E_2(X;\mathfrak q)|
	\ge
	|A'\cdot B'|-\varepsilon|G|.
	\]
	
	\medskip
	\noindent
	\textit{Proof of (b.iii).}
	For any \(c\in G\),
	$(1_{A'\cdot B'}*1_{A'})(c)
	 =
	\sum_{d\in A'\cdot B'}
	1_{A'}(cd^{-1}).$
	Using \eqref{6.4},
	\[
		(1_A*1_A*1_A)(c)
		=
		\sum_{d\in G}
		(1_A*1_A)(d)\,1_A(cd^{-1})
		\ge
		\sum_{d\in A'\cdot B'}
		(1_A*1_A)(d)\,1_{A'}(cd^{-1})
		\ge
		t\,
		(1_{A'\cdot B'}*1_{A'})(c).
	\]
	Therefore,
\(	(1_{A'\cdot B'}*1_{A'})(c)\gg |G|\)
	implies
	\(	(1_A*1_A*1_A)(c)\ge t|G| \gg |G|^2.\)
	By Proposition~\ref{prop:transference-conclusion}(iii),
	\(c\in E_3(X;\mathfrak q).\)
	
	\medskip
	\noindent
	\textit{Proof of (b.iv).}
Let \(\beta\in(2\varepsilon,B] \) and \(	L=\bigl\lfloor\frac{B}{\beta}\bigr\rfloor.\)
By \eqref{6.4},  
for every \(c\in G\),
\[
\sum_{[\mathfrak p]a_1a_2=c,\ a_1,a_2\in A
	\atop
	Q^{\beta-\varepsilon}<N\mathfrak p\le Q^\beta}
\frac1{N\mathfrak p}
\ge
\sum_{[\mathfrak p]d=c,\ d\in A'\cdot B'
	\atop
	Q^{\beta-\varepsilon}<N\mathfrak p\le Q^\beta}
\frac{(1_A*1_A)(d)}{N\mathfrak p}
\gg
|G|
\sum_{[\mathfrak p]d=c,\ d\in A'\cdot B'
	\atop
	Q^{\beta-\varepsilon}<N\mathfrak p\le Q^\beta}
\frac1{N\mathfrak p}.
\]
Hence the hypothesis in (b.iv) implies the hypothesis of
Proposition~\ref{prop:transference-conclusion}(iv), and thus
\(c\in E_3(X;\mathfrak q).\)
	
	\medskip
	\noindent
	\textit{Proof of (b.v).}
	Let \(H\le G\) be a subgroup of index
	\(Y\le\varepsilon^{-1/2}\).	
	By Proposition~\ref{prop:transference-conclusion}(v),   Then there
	exist at least
\(	\Bigl\lceil
\bigl(
\frac{\vartheta_0}{2}
-
3\varepsilon^{1/2}\frac{\vartheta_0}{\vartheta}
\bigr)Y
\Bigr\rceil\)
	distinct cosets \(bH\) of \(H\) such that
	$	|A\cap bH|>\varepsilon |G|.$  
	Take one such coset \(bH\). Using \eqref{6.3},
	\[
		|A'\cap B'\cap bH|
	 \ge
	|A\cap bH|
	-
	|A\setminus A'|
	-
	|A\setminus B'|
	 >
	\varepsilon|G|
	-
	\frac{\varepsilon|G|}{10}
	 >
	\frac{\varepsilon}{2}|G|.
	\]	
	\medskip
	\noindent
	\textit{Proof of (b.vi).}
	By Proposition~\ref{prop:transference-conclusion}(vi),
	\(	|A\cap bH|
	\ge
	\bigl(
	\frac{\vartheta}{2}
	\frac{|	\pi (X;bH)|}{X/\log X}
	-
	\frac{\varepsilon}{5Y}
	\bigr)|G|.\)
	Using \eqref{6.3}, 
	\[
	\begin{aligned}
		|A'\cap B'\cap bH|
		&\ge
		|A\cap bH|
		-
		|A\setminus A'|
		-
		|A\setminus B'|
		\\
		&\ge
		\Bigl(
		\frac{\vartheta}{2}
		\frac{|	\pi (X;bH)|}{X/\log X}
		-
		\frac{\varepsilon}{5Y}
		-
		\frac{\varepsilon}{10}
		\Bigr)|G|\\
		&\ge 	 
		\Bigl(
		\frac{\vartheta}{2}
		\frac{|	\pi (X;bH)|}{X/\log X}
		-
		\frac{3\varepsilon}{10}
			\Bigr)|G|.
	\end{aligned}\]	
	This completes the proof. 
\end{proof}
\subsection{Structure of sets with small doubling}
The remaining lemma is a structural statement for sets with small product
set. It is the finite-abelian-group form of the corresponding lemma of
Matomäki--Teräväinen and follows from Kneser's theorem.
\begin{lemma} 
	\label{lem:small-doubling-structure}
	Let $G$ be a finite abelian group, and let
	$\alpha,\alpha',\beta\in(0,1]$ satisfy
\(	\beta<2\alpha\le2\alpha'.\)
	Let \(A,B\subseteq G\) with
	\(	|A|,  |B|\ge \alpha|G|.\)
	Assume that
	each of $A$ and $B$ meets at least proportion $\alpha'$ of the cosets of
	every subgroup $H_0\leq G$ with
	\(	[G:H_0]<\frac{1}{2\alpha-\beta}.\)
	Then at least one of the following  holds.
	
	\begin{enumerate}
		\item[(a)]
		\(	|A\cdot B|\ge \beta|G|.\)
		\item[(b)]
		Let \(H\le G\) be the stabilizer of \(A\cdot B\), and write
		\(	Y=[G:H].\)
		Then:
		
		\begin{enumerate}
			\item[(b.i)]
		\(	1<Y<\frac{1}{2\alpha'-\beta}.\)
			
			\item[(b.ii)]
			If
			\(	\alpha'>\frac13,
			\beta\le\frac23,\)
			then
			\(	Y=3k+2\)
			for some integer \(k\ge0\) satisfying
	\(\frac{k+1}{3k+2}\geq \alpha',\)
	each of $A$ and $B$ meets exactly $k+1$ cosets of $H$, and
	$A\cdot B$ is the union of exactly $2k+1$ cosets of $H$.
			
			\item[(b.iii)]
			If, for some sufficiently small \(\varepsilon>0\),
			$
			\alpha'\ge \frac38-\varepsilon,
			$
			$\beta<\frac{11}{16}-2\varepsilon,
			$
			then
			$
			Y=3k+2
			$
			for some
            $
			k\in\{0,1,2\},
			$
			each of $A$ and $B$
			meets exactly $k+1$ cosets of $H$, and $A\cdot B$ is the union
			of exactly $2k+1$ cosets of $H$.
		\end{enumerate}
	\end{enumerate}
\end{lemma}
\begin{proof}
	This is the finite-abelian-group form of   \cite[Lemma 6.3]{matomaki}. Their proof uses only
	Kneser's theorem for finite abelian groups and elementary counting of
	cosets of the stabilizer of \(A\cdot B\). It applies verbatim to
	any finite abelian group \(G \).
\end{proof}
\section{Prime ideals in cosets of indices \(5\) and \(8\)}

Let
$
G=\operatorname{Cl}_{\mathfrak q}^{(\infty)}$,
$ 
Q=N\mathfrak q.
$
In order to deal with the   case arising from
Lemma~\ref{lem:small-doubling-structure}, we need information on the
distribution of prime ideals in cosets of subgroups of \(G\) of index
\(3k+2\), for \(k\in\{1,2\}\).

We first record the Hecke-character analogue of the weighted prime-sum input
used by Matomäki--Teräväinen.
 
\begin{lemma} \label{lem:weighted-prime-ideal-sums-rescaled}

	Let \(A_0>0\) be fixed, and define
	\[
	f_{A_0}(u)
	:=
	\begin{cases}
		A_0-u, & 0\le u\le A_0,\\
		0, & u>A_0.
	\end{cases}
	\]
    Let 	\(\chi\) be a non-principal Hecke character
 modulo \(\mathfrak q\) induced by the primitive character $\chi^*$ modulo \(\mathfrak f_\chi\), suppose that  $
	\chi^*$ satisfies 
	\begin{equation}\label{B.1}
			L\Bigl(\frac12+it,\chi^*\Bigr)
		\ll_{K,\varepsilon}
	 C(\chi^*,t) ^{\alpha_0/2+\varepsilon}
	 ,
	\end{equation}
	where $C(\chi^*,t)$ is the analytic conductor of $\chi^*$, then 
	\begin{equation}	\label{non-principal2}
			\Re
		\sum_{\mathfrak p\nmid \mathfrak q}
		\chi(\mathfrak p)
		\frac{\log N\mathfrak p}{N\mathfrak p}
		f_{A_0}\Bigl(
		\frac{\log N\mathfrak p}{\log Q}
		\Bigr)
		\le
		\Bigl(
		\frac{A_0\alpha_0}{2}
		+ o_K(1)
	\Bigr)\log Q.
	\end{equation}
	For the principal character,
	\begin{equation}\label{principal}
		\sum_{\mathfrak p\nmid\mathfrak q}
	\frac{\log N\mathfrak p}{N\mathfrak p}
	f_{A_0}\Bigl(
	\frac{\log N\mathfrak p}{\log Q}
	\Bigr)
	=
	\Bigl(
	\frac{A_0^2}{2}
	+ o_K(1)
\Bigr)\log Q.
	\end{equation}
\end{lemma}

\begin{proof}
We use the number-field version of the weighted explicit-formula inequality due to Zaman \cite[Prop.~6.2]{zaman}, which is analogous to Heath-Brown
\cite[Lem.~5.2]{heathbrown}. 
	Take \(\vartheta=1\) in Zaman's notation \cite[(3.2), (3.3)]{zaman}, so that
	\[
	\mathcal L=\log d_K+\log N\mathfrak q+n_K\nu(n_K)
	=\log Q+O_K(1),
	\qquad
	\mathcal L_\chi
	=\log d_K+\log N\mathfrak f_\chi+n_K\nu(n_K)
	\le \mathcal L,
	\]
	where $\nu: [1,\infty)\to  [4,\infty)$ is any fixed increasing function with $\nu(x) \gg \log(x+4)$.  
	
	We first check that  \cite[Lem.~4.2]{zaman}  is satisfied with
	\(\phi=\alpha_0/2\). Let \(\eta=(\log\mathcal L)/\mathcal L\).   Since \(\chi^*\) is primitive and
	non-principal, \(L(s,\chi^*)\) is entire of finite order.   	
	Let \(s=\sigma+it\), where
	\(	\frac12\le \sigma\le 1+\eta .\)
	By the assumed bound on the
	central line and the standard estimate
	$
	C(\chi^*,t)\ll_K d_KN\mathfrak f_\chi(1+|t|)^{n_K},
	$
	we have for any $\varepsilon_0>0$,
	\[
		\Bigl|L \Bigl(\frac12+it,\chi^*\Bigr)\Bigr|
	 \ll_{K,\varepsilon }
	\bigl(d_K N \mathfrak f_\chi \bigr)^{\alpha_0/2+\varepsilon_0}
	(1+|t|)^{n_K(\alpha_0/2+\varepsilon_0 )} .
	\]
	On the other hand, on the line \(\sigma=1+\eta\), the Euler product gives
	\[
	 \bigl|L(1+\eta+it,\chi^*)\bigr|
	\le
	\zeta_K(1+\eta)
	\ll_K
	\eta^{-n_K}.
	\]

	Writing
\(	\ell(\sigma)
=
\frac{1+\eta-\sigma}{1/2+\eta} \) and  applying
	the Phragmen--Lindelof principle \cite[Thm.~2]{rademacherPL} in  the strip
	$
	\frac12\le \sigma\le 1+\eta
	$
	to   \(L(s,\chi^*)\) gives
\begin{equation}\label{bound1}
		\log |L(s,\chi^*)|
	\le
	\ell(\sigma)
	\bigl(\frac{\alpha_0}{2}+\varepsilon_0\Bigr)
	\bigl(\log d_K+\log N\mathfrak f_\chi\bigr)
	+
	O_{K,\varepsilon_0 }(\log(1+|t|))
	+
	O_{K,\varepsilon_0 }(\log \mathcal L).
\end{equation}
Take $\eta(x) =\log x /2+2$ in the notation of  \cite[(3.2)]{zaman}, we have \[
\mathcal{T} = \bigl(\mathcal{L} -n_K\nu(n_K)\bigr) ^{ \frac{2 \log (n_K+1) }{\log n_K +4}} +\nu(n_K).
\]
	Now assume that \(|t|\le \mathcal T\), then 
	$
	\log(1+|t|)
	\ll_K
	\log \mathcal L
	=
	o_K (\mathcal L).
	$
	 
	Since
	\[
	\ell(\sigma)
	=
	\frac{1+\eta-\sigma}{1/2+\eta}
	=
	2(1-\sigma)+O(\eta) \qquad\text{and}\qquad \log N\mathfrak f_\chi \leq \mathcal L _\chi , 
 	\]
  by \eqref{bound1} and taking $\varepsilon_0$ sufficiently small in terms of $\varepsilon$,   uniformly for
	$
	\frac12\le \sigma\le 1+\frac{\log \mathcal L}{\mathcal L}
	$, we have 
	$|t|\le \mathcal T,
	$
	\[
	\log |L(s,\chi^*)|
	\le
	\alpha_0 \mathcal L_\chi(1-\sigma+\varepsilon )
	+
	o_{K, \varepsilon }(\mathcal L).
	\]
	This is the input of \cite[Lem.~4.2]{zaman} with
	\(	\phi=\frac{\alpha_0}{2}.\)

	The weight \(f_{A_0}\) satisfies Condition~1 of \cite[p.~335]{zaman}
	with \(x_0=A_0\): it is continuous on \([0,\infty)\), vanishes for
	\(u\ge A_0\), and is twice differentiable on \((0,A_0)\), with
	\(f_{A_0}''=0\).
  Applying  \cite[Prop.~6.2]{zaman} to   \(L(s,\chi)\) at \(s=1\)   with
	weight \(f_{A_0}\) and $\phi= \alpha_0/2$ gives that for any $\varepsilon>0$, there exists $\delta =\delta (\varepsilon)>0$  such that 
\begin{equation}\label{zero}
	\Re
	\sum_{\mathfrak a}
	\frac{\Lambda_K(\mathfrak a)\chi(\mathfrak a)}
	{N\mathfrak a}
	f_{A_0} \Bigl(
	\frac{\log N\mathfrak a}{\mathcal L}
	\Bigr)
	\le
	-\mathcal L
	\sum_{|1-\rho|\le \delta}
	\Re F_{A_0}((1-\rho)\mathcal L)
	+
	A_0\frac{\alpha_0}{2}\mathcal L _\chi
	+
	\varepsilon\mathcal L,
\end{equation}
where the sum is over the non-trivial zeros \(\rho\) of the primitive
Hecke \(L\)-function \(L(s,\chi^*)\)  counted with multiplicity,   and
	\[
	F_{A_0}(z)
	:=
	\int_0^{A_0}
	(A_0-u)e^{-zu}\,du
	\]
	is the Laplace transform of \(f_{A_0}\).

	We now check that the zero term in \eqref{zero} is non-positive.  First
	observe that \(F_{A_0}\) extends holomorphically to \(z=0\), with
	\(F_{A_0}(0)=A_0^2/2\). Thus \(\Re F_{A_0}\) is
	harmonic in the half-plane \(\Re z>0\) and continuous on its closure
	\(\Re z\ge0\).  On the boundary \(\Re z=0\), say \(z=iy\), we
	have
	\[
	\Re F_{A_0}(iy)
	=
	\int_0^{A_0}(A_0-u)\cos(yu)\,du
	=
	\frac{1-\cos(A_0y)}{y^2}
	\ge 0,
	\]
	with the value at \(y=0\) understood by continuity.
	
	Moreover,
	$
	F_{A_0}(z)
	=
	\frac{A_0z-1+e^{-A_0z}}{z^2},
	$
	so \(F_{A_0}(z)\to 0\) as \(|z|\to\infty\) uniformly in the half-plane
	\(\Re z\ge0\).  Applying the minimum principle
	\cite[Cor.~1.9]{harmonic} to
	\(\Re F_{A_0}(z)\) in the half-discs
	\(\{z:\Re z\ge0,\ |z|\le R\}\), and then letting \(R\to\infty\), gives
	\[
	\Re F_{A_0}(z)\ge0
	\qquad (\Re z\ge0).
	\]
	Since every non-trivial zero satisfies \(\Re\rho\le1\), we have
	$
	\Re\bigl((1-\rho)\mathcal L\bigr)\ge0.
	$
	Therefore
	\[
	\Re F_{A_0}\bigl((1-\rho)\mathcal L\bigr)\ge0,
	\]
	and hence the zero term in \eqref{zero} is non-positive and may be discarded. Therefore, since $\mathcal{L}_\chi\leq \mathcal{L}$, we have 
	\[	\Re
	\sum_{\mathfrak a}
	\frac{\Lambda_K(\mathfrak a)\chi(\mathfrak a)}
	{N\mathfrak a}
	f_{A_0}\Bigl(
	\frac{\log N\mathfrak a}{\mathcal L }
	\Bigr)
	\le
	\Bigl(
	\frac{A_0\alpha_0}{2}
	+ \varepsilon
	\Bigr)\mathcal L. \]
	As \(\mathcal L=\log Q+O_K(1)\),  and $\varepsilon>0$ is arbitrary,   it follows that	\begin{equation}\label{non-principal1}
			\Re
		\sum_{\mathfrak a}
		\frac{\Lambda_K(\mathfrak a)\chi(\mathfrak a)}
		{N\mathfrak a}
		f_{A_0}\Bigl(
		\frac{\log N\mathfrak a}{\log Q }
		\Bigr)
		\le
		\Bigl(
		\frac{A_0\alpha_0}{2}
		+ o (1)
				\Bigr)\log Q.
	\end{equation}
	Moreover, the  contribution of prime powers \(\mathfrak p^m\), \(m\ge2\), is
	\(O_K(1)\).   Thus  \eqref{non-principal1} yields \eqref{non-principal2}.
	
	We now prove the principal-character asymptotic. By Mertens' theorem for
	prime ideals \cite[Lem.~2.3]{merten},
	\[
	  \sum_{N\mathfrak p\le y}
	\frac{\log N\mathfrak p}{N\mathfrak p}
	=
	\log y + O_K(1).
	\]
	Partial summation gives
	\[ \sum_{\mathfrak p} \frac{\log N\mathfrak p}{N\mathfrak p} f_{A_0}\Bigl( \frac{\log N\mathfrak p}{\log Q} \Bigr) =
	\log Q\int_0^{A_0}(A_0-u)\,du+O_{A_0,K}(1)=      \frac{A_0^2}{2}\log Q +O_{A_0,K}(1)   . \]
 
	Moreover, the prime ideals dividing \(\mathfrak q\) contribute only
	\(o_K(\log Q)\) since \(\sum_{\mathfrak p\mid\mathfrak q}
	\frac{\log N\mathfrak p}{N\mathfrak p}
	\ll_K
	( \log  Q)^{1/2} \); see for example \cite[Lem.~2.4]{zaman}.  This proves \eqref{principal}.
\end{proof}
We now assume    a bounded-order subconvexity input in order to apply Lemma \ref{lem:weighted-prime-ideal-sums-rescaled}. For every fixed
integer \(\ell\ge2\) and every \(\varepsilon>0\), every primitive
non-principal Hecke character \(\chi^*\) of order at most \(\ell\) satisfies
\begin{equation} \label{Lb}
	L\Bigl(\frac12+it,\chi^*\Bigr)
	\ll_{K,\ell,\varepsilon}
	C(\chi^*,t)^{\alpha_0/2+\varepsilon}
	\qquad (t\in\mathbb R),
	\tag{\(\mathrm L^{\mathrm b}(\alpha_0)\)}
\end{equation}
where \(C(\chi^*,t)\) denotes the analytic conductor of \(\chi^*\).  
 
Note that the same proof of Proposition \ref{prop_subconvexity} shows that \eqref{Lb} implies \eqref{CS^b}. 
\begin{lemma}
	\label{lem:prime-escape-rescaled} Assume \eqref{Lb}.
	Let \(k\in\{1,2\}\), and let \(H\le G\) be a subgroup of index
	\(3k+2\). Suppose that \(G/H\) is cyclic and generated by \(gH\).	
	Put
\(	A_0=4\alpha_0.\)
	Then
	\begin{equation}\label{first}
		\sum_{ {
				N\mathfrak p\le Q^{A_0}, \,	\mathfrak p\nmid\mathfrak q\atop
				[\mathfrak p]\notin
				\bigcup_{j=1}^{k+1}g^{k+j}H}}
		\frac{\log N\mathfrak p}{N\mathfrak p}
		\gg_{K,\alpha_0}
		\log Q.
	\end{equation}
\end{lemma}
\begin{proof}
	Put \(A_0=4\alpha_0\), and define
	\[
	w(\mathfrak p)
	:=
	\frac{\log N\mathfrak p}{N\mathfrak p}
	f_{A_0}\Bigl(\frac{\log N\mathfrak p}{\log Q}\Bigr), \qquad W
	:=
	\sum_{\mathfrak p\nmid\mathfrak q} w(\mathfrak p),
	\]
	where $f_{A_0}$ is as defined in Lemma~\ref{lem:weighted-prime-ideal-sums-rescaled}. 
	 
	 Lemma~\ref{lem:weighted-prime-ideal-sums-rescaled}  applied to the
	principal character gives 
	\begin{equation}\label{7.1}
		W
		=
	\Bigl(
		\frac{A_0^2}{2}+o_K(1)
		\Bigr)\log Q.
	\end{equation}
	For every non-principal character \(\chi\) of \(G\) of bounded order, using  \eqref{Lb} in
	Lemma~\ref{lem:weighted-prime-ideal-sums-rescaled}  gives
	\[
	\Re
	\sum_{\mathfrak p\nmid\mathfrak q}
	\chi(\mathfrak p)w(\mathfrak p)
	\le
	\Bigl(
	\frac{A_0\alpha_0}{2}+o_K(1)
	\Bigr)\log Q.
	\]
	Since \(A_0=4\alpha_0\), this may be rewritten  using \eqref{7.1}  as
	\begin{equation}\label{7.2}
		\Re
	\sum_{\mathfrak p\nmid\mathfrak q}
	\chi(\mathfrak p)w(\mathfrak p)
	\le
	\Bigl(
	\frac14+o_K(1)
	\Bigr)W.
	\end{equation}
  Since \(0\le f_{A_0}\le A_0\), we have
	$
	w(\mathfrak p)
	\le
	A_0\frac{\log N\mathfrak p}{N\mathfrak p}.
	$
		It is enough to
	show that the weighted sum outside the exceptional block of cosets is at
	least a fixed positive proportion of \(W\), say \(10^{-3}W\). That is, 
	 it suffices to show that 
	 \begin{equation}\label{suffice}
	 		\sum_{ {
	 				\mathfrak p\nmid\mathfrak q\atop
	 			[\mathfrak p]\notin
	 			\bigcup_{j=1}^{k+1}g^{k+j}H}}
	 	w(\mathfrak p)
	 	\ge
	 	10^{-3}W.
	 \end{equation}
	We now prove  this estimate  by contradiction. 
	
	First suppose that \(G/H\) is cyclic of order \(5\). Let \(gH\) generate
	\(G/H\), and suppose  for contradiction  that
	\[
	\sum_{{\mathfrak p\nmid\mathfrak q\atop
			[\mathfrak p]\notin g^2H\cup g^3H}}
	w(\mathfrak p)
	<
	10^{-3}W.
	\]
	Define \(\beta_0,\beta_1,\beta_2 \ge 0 \) by
	\[
	\beta_0 W
	:=
	\sum_{ {\mathfrak p\nmid\mathfrak q\atop
			[\mathfrak p]\in H}}
	w(\mathfrak p),\qquad 	\beta_1 W
	:=
	\sum_{ {\mathfrak p\nmid\mathfrak q\atop
			[\mathfrak p]\in gH\cup g^4H}}
	w(\mathfrak p), \qquad	\beta_2 W
	:=
	\sum_{ {\mathfrak p\nmid\mathfrak q\atop
			[\mathfrak p]\in g^2H\cup g^3H}}
	w(\mathfrak p).
	\]
	Then
	$
	\beta_0+\beta_1+\beta_2=1$, $
	\beta_0+\beta_1<10^{-3},
	$
	and hence
	$
	\beta_2>1-10^{-3}.
	$
	
	Let \(\chi\) be the character of \(G\) induced from \(G/H\) by
	\[
	\chi(g)=e^{2\pi i\cdot 2/5},
	\qquad
	\chi(h)=1\quad(h\in H).
	\]
	Then
	\[
	\Re
	\sum_{\mathfrak p\nmid\mathfrak q}
	\chi(\mathfrak p)w(\mathfrak p)
	=
	\Bigl(
	\beta_0
	+
	\beta_1\cos\frac{4\pi}{5}
	+
	\beta_2\cos\frac{2\pi}{5}
	\bigr)W   \ge
	\Bigl(
	10^{-3}\cos\frac{4\pi}{5}
	+
	(1-10^{-3})\cos\frac{2\pi}{5}
	\Bigr)W.
	\]
	The constant on the right  is $\approx 0.3079$ which is strictly larger than \(1/4\). This contradicts
\eqref{7.2} for \(Q\) sufficiently large. Therefore
		\begin{equation}\label{7.4}
	\sum_{ {\mathfrak p\nmid\mathfrak q\atop
			[\mathfrak p]\notin g^2H\cup g^3H}}
	w(\mathfrak p)
	\ge
	10^{-3}W.
	\end{equation}
	
	We now suppose that \(G/H\) is cyclic of order \(8\).  Let \(gH\) generate \(G/H\),
	and suppose  for contradiction  that
	\[
	\sum_{ {\mathfrak p\nmid\mathfrak q\atop
			[\mathfrak p]\notin g^3H\cup g^4H\cup g^5H}}
	w(\mathfrak p)
	<
	10^{-3}W.
	\]
	Define \(\beta_0,\dots,\beta_4  \geq 0 \) by
	\[
		\beta_0 W
	:=
	\sum_{{\mathfrak p\nmid\mathfrak q\atop
			[\mathfrak p]\in H}}
	w(\mathfrak p),\qquad 	\beta_4 W
	:=
	\sum_{{\mathfrak p\nmid\mathfrak q \atop
			[\mathfrak p]\in g^4H}}
	w(\mathfrak p),
\qquad
		\beta_i W
	:=
	\sum_{{\mathfrak p\nmid\mathfrak q \atop
			[\mathfrak p]\in g^iH\cup g^{8-i}H}}
	w(\mathfrak p), \, (i=1,2,3).
	\]
	Then
	\begin{equation}\label{7.6}
			\sum_{j=0}^4\beta_j=1,
		\qquad
		\beta_0+\beta_1+\beta_2<10^{-3}.
	\end{equation}

	Let \(\chi\) be the character induced from \(G/H\) by
	\[
	\chi(g)=e^{2\pi i/8},
	\qquad
	\chi(h)=1\quad(h\in H).
	\]
	Applying \eqref{7.2} to the non-principal character \(\chi^2\) and \(\chi^3\), we get
\begin{equation}\label{apply}
		\Re
	\sum_{\mathfrak p\nmid\mathfrak q}
	\chi(\mathfrak p)^2 w(\mathfrak p)
	\le
	\Bigl(\frac14+o_K(1)\Bigr)W,\qquad 	\Re
	\sum_{\mathfrak p\nmid\mathfrak q}
	\chi(\mathfrak p)^3 w(\mathfrak p)
	\le
	\Bigl(\frac14+o_K(1)\Bigr)W
\end{equation}
Since
	\(	\Re \chi(g^j)^2
	=
	\cos\frac{\pi j}{2},\)
 we have 
	\[
	\Re
	\sum_{\mathfrak p\nmid\mathfrak q}
	\chi(\mathfrak p)^2 w(\mathfrak p)
	=
	(\beta_0-\beta_2+\beta_4)W.
	\]
	Therefore, since $\beta_2< 10^{-3}$, 
	\begin{equation}\label{7.7}
			\beta_4
		\le
		\frac14+\beta_2+o_K(1)
		\le
		\frac14+10^{-3}+o_K(1).
	\end{equation}
	Combining \eqref{7.6} and \eqref{7.7}, we obtain
	\begin{equation}\label{7.8}
			\beta_3
		=
		1-\beta_0-\beta_1-\beta_2-\beta_4
		\ge
		\frac34-2\cdot10^{-3}-o_K(1).
	\end{equation}
Since
$
	\Re \chi(g^j)^3
	=
	\cos\frac{3\pi j}{4},
$  
	we have
	\[
		\Re
	\sum_{\mathfrak p\nmid\mathfrak q}
	\chi(\mathfrak p)^3 w(\mathfrak p)
=
	\Bigl(
	\beta_0
	-
	\frac{\beta_1}{\sqrt2}
	+
	\frac{\beta_3}{\sqrt2}
	-
	\beta_4
	\Bigr)W.
	\]
	Using \eqref{7.6}, \eqref{7.7}, and \eqref{7.8}, 
this is at least
	\[
\Bigl(
	-\frac{10^{-3}}{\sqrt2}
	+
	\frac{3/4-2\cdot10^{-3}}{\sqrt2}
	-
	\frac14
	-
	10^{-3}
	-o_K(1)
	\Bigr)W.
	\]
	The constant in the parentheses is $\approx 0.2772$ which is strictly larger than \(1/4\). This contradicts
	\eqref{apply} for \(Q\) sufficiently large. Therefore
	\begin{equation} \label{7.9}
			\sum_{ \mathfrak p\nmid\mathfrak q\atop
				[\mathfrak p]\notin g^3H\cup g^4H\cup g^5H}
		w(\mathfrak p)
		\ge
		10^{-3}W.
	\end{equation}
This finishes the proof of  \eqref{suffice} by \eqref{7.4} and \eqref{7.9}.  
\end{proof}
\begin{corollary} 
	\label{cor:dyadic-prime-escape}
	Under the hypotheses of Lemma~\ref{lem:prime-escape-rescaled}, let $\varepsilon>0$ be fixed and sufficiently small. Then 
	there exists
\(	\beta\in(2\varepsilon,4\alpha_0]\)
	such that
	\[
	\sum_{ 
			Q^{\beta-\varepsilon}<N\mathfrak p\le Q^\beta\atop
			[\mathfrak p]\notin
			\bigcup_{j=1}^{k+1}g^{k+j}H}
	\frac1{N\mathfrak p}
	\gg_{K,\alpha_0,\varepsilon}
	1.
	\]
\end{corollary}
\begin{proof}
		The contribution   in the range 
\(	N\mathfrak p\le Q^{2\varepsilon}\) in \eqref{first} 
	is
\(	O_K(\varepsilon\log Q).\)
	Choosing \(\varepsilon>0\) sufficiently small, Lemma \ref{lem:prime-escape-rescaled} gives
	\[
	\sum_{ 
			Q^{2\varepsilon}<N\mathfrak p\le Q^{4\alpha_0}\atop
			[\mathfrak p]\notin
			 \bigcup_{j=1}^{k+1}g^{k+j}H}
	\frac{\log N\mathfrak p}{N\mathfrak p}
		\gg_{K,\alpha_0,\varepsilon}
	\log Q.
	\]
Cover the interval \((Q^{2\varepsilon},Q^{4\alpha_0}]\) by
\(O_{\alpha_0,\varepsilon}(1)\) logarithmic intervals of the form  
	$
	(Q^{\beta-\varepsilon},Q^\beta].
	$
	Hence there exists at least one
	$
	\beta\in(2\varepsilon,4\alpha_0]
	$
	such that
	\[
	\sum_{ 
			Q^{\beta-\varepsilon}<N\mathfrak p\le Q^\beta\atop
			[\mathfrak p]\notin
		\bigcup_{j=1}^{k+1}g^{k+j}H}
	\frac{\log N\mathfrak p}{N\mathfrak p}
		\gg_{K,\alpha_0,\varepsilon}
	\log Q.
	\]
	Since $\log N\mathfrak p\asymp_{\alpha_0,\varepsilon}\log Q$ throughout
	this interval, the desired estimate follows.
\end{proof}
 
\section{Representation by three prime ideals}\label{three}

In this section, we prove Theorem~\ref{intro:three-prime-ray}, using
  exceptional quadratic estimates that will be established in the
next section. Throughout, we write
$
G=\operatorname{Cl}_{\mathfrak q}^{(\infty)}$,
$
Q=N\mathfrak q, 
$
and assume that \eqref{CS} and \eqref{Lb} hold for fixed parameters
\(0<\alpha_0\le \alpha<1\).

\begin{proof}[Proof of Theorem \ref{intro:three-prime-ray}]
		Let \(\kappa >0\) be fixed. We want to prove that if 
	$X\ge Q^{\max(1,3\alpha,4\alpha_0)+\kappa},$
	then
	$	E_3(X;\mathfrak q)=G.$
	Let $\varepsilon>0$ be fixed that is taken to be sufficiently small. 
	Set  
	\[
	\vartheta
	=
	1-\varepsilon-\alpha\frac{\log Q}{\log X },
	\qquad
	\vartheta_0
	=
	1-\varepsilon- {\alpha_0} \frac{\log Q}{\log X }.
	\]
	It suffices to take \(X= Q^{\max(1,3\alpha,4\alpha_0)+\kappa} \) 	 and  we will apply Proposition~\ref{prop:criteria-E2-E3}. 
	
	After decreasing \(\varepsilon\) if necessary, we may fix a small constant $\eta>0$, say $0<\eta < 10^{-3}$, such that 
	\begin{equation}\label{8.1}
			\frac{\vartheta}{2}-\frac{3\varepsilon}{2}
		>
	\frac13+\eta
	\end{equation}
	and
	\begin{equation}\label{8.2}
			\frac{\vartheta_0}{2}
		-
		3\varepsilon^{1/2}\frac{\vartheta_0}{\vartheta}
		>
		\frac38-\eta.
	\end{equation}

	If Proposition~\ref{prop:criteria-E2-E3}(a) holds, then by \eqref{8.1}
	we have
$\vartheta>\frac23+3\varepsilon,$
so we immediately obtain 
 \(E_3(X;\mathfrak q)=G.\)
	
	We may therefore assume that Proposition~\ref{prop:criteria-E2-E3}(b)
	holds. Let \(A',B'\subseteq G\) be the sets supplied there. By
	Proposition~\ref{prop:criteria-E2-E3}(b.i) and \eqref{8.1},
\begin{equation}\label{8.4}
		|A'|,\ |B'|
	\ge
	\Bigl(\frac13+\eta\Bigr)|G|.
\end{equation}
	
 We apply
	Lemma~\ref{lem:small-doubling-structure} with
\[
\alpha_*=\frac13+\eta,
\qquad
\alpha_*'=\frac38-\eta,
\qquad
\beta_*=\frac23.
\]
These parameters satisfy
$	\beta_*<2\alpha_*\le 2\alpha_*' .$
 Moreover, for $\varepsilon$ sufficiently small, 
 \( \frac1{2\alpha_*-\beta_*} =(2\eta)^{-1}\leq  \varepsilon^{-1/2}\). 
Hence,  Proposition~\ref{prop:criteria-E2-E3}(b.v) together with  \eqref{8.2} shows that
 \(A'\) and \(B'\) each meet at least proportion
\(\alpha_*'=\frac38-\eta\)
 of the cosets of  every subgroup \(H_0\le G\) with
 $
 [G:H_0]<\frac1{2\alpha_*-\beta_*}
 $. Indeed, it gives at least
 \(\alpha_*'\) proportion of cosets \(bH_0\) satisfying
\( |A'\cap B'\cap bH_0|\ge \frac{\varepsilon}{2}|G|,\)
 and each such coset is met by both \(A'\) and \(B'\).
 
	If Lemma~\ref{lem:small-doubling-structure}(a) holds, then
$
	|A'\cdot B'|
	\ge
 \frac23 |G|.
$
Together with \eqref{8.4}, this gives
	\[
	|A'\cdot B'|+|A'|\ge
	(1+ \eta)|G|.
	\]
	Hence, by the convolution lower bound from Lemma \ref{lem:convolution-product-lower-bound},  for any \(c\in G\),
	\[
	(1_{A'\cdot B'}*1_{A'})(c)\ge  \eta  |G|
	 .
	\]
	Proposition~\ref{prop:criteria-E2-E3}(b.iii) then gives
	$
	c\in E_3(X;\mathfrak q)
 $
	and therefore
	$
	E_3(X;\mathfrak q)=G.
	$
	
	We may therefore assume that Lemma~\ref{lem:small-doubling-structure}(b)
	holds.  	Let \(H\) be the stabilizer of \(A'\cdot B'\). By Lemma~\ref{lem:small-doubling-structure}(b.i)(b.ii),  \([G:H] <\frac{1}{2\alpha_*'- \beta_*}\), and there exists an integer \(k\ge0\) with \(
	\frac{k+1}{3k+2}\geq \alpha_*'\)
	such that
	$
	[G:H]=3k+2,
	$
	each of \(A'\) and \(B'\) meets exactly \(k+1\) cosets of \(H\), and
	\(A'\cdot B'\) is the union of exactly \(2k+1\) cosets of \(H\).
	
	We next show that in fact \(k\in\{0,1,2\}\). By
	Proposition~\ref{prop:criteria-E2-E3}(b.v) and \eqref{8.2}, the sets
	\(A'\) and \(B'\) each intersects more than
	$
	\bigl(\frac38-\eta\bigr)[G:H]
	$
	cosets of \(H\). Since each of \(A'\) and \(B'\) meets exactly \(k+1\)
	cosets of \(H\), we get
	$
	k+1
	\ge
	\bigl(\frac38-\eta\bigr)(3k+2),
	$
	which forces \(k\in\{0,1,2\}\).
	
	Moreover, the same estimate shows that the \(k+1\) cosets met by
	\(A'\) and the \(k+1\) cosets met by \(B'\) are the same. Indeed, the
	number of cosets in which \(A'\cap B'\) is non-empty is at least
	\[
	\Bigl\lceil
	\Bigl(\frac38-\eta\Bigr)(3k+2)
	\Bigr\rceil
	=
	k+1
	\qquad(k=0,1,2).
	\]	
	
	Let 
	$
	a_1H,\dots,a_{k+1}H
	$ be the \(H\)-cosets met by \(A'\) and \(B'\). Then, 	\(A',B'\subseteq S:=\bigcup_{j=1}^{k+1}a_jH  \) and  $A'\cdot B'\subset S^2$.   	Moreover, since \(A'\) and \(B'\) meet every \(H\)-coset contained in \(S\), the product 
	\(A'\cdot B'\) meets every \(H\)-coset contained in \(S^2\). Since \(H\) is the stabilizer of \(A'\cdot B'\), the set \(A'\cdot B'\) is a
	union of \(H\)-cosets. Therefore every \(H\)-coset contained in \(S^2\) is
	contained in \(A'\cdot B'\), and so
	$
	A'\cdot B'=S^2.
	$
Since \(A'\cdot B'=S^2\) is the union of exactly \(2k+1\) \(H\)-cosets, denote these cosets by
$
	b_1H,\dots,b_{2k+1}H.
$
	Then,
	$
	A'\cdot B'
	=
	\bigcup_{j=1}^{2k+1}b_jH
	=
	\Bigl(
	\bigcup_{j=1}^{k+1}a_jH
	\Bigr)^2.
	$
	Furthermore, Proposition~\ref{prop:criteria-E2-E3}(b.v) gives
	\[
		|A'\cap a_iH|\ge
	|A'\cap B'\cap a_iH|\ge \frac{\varepsilon}{2}|G|
	\qquad(1\le i\le k+1).
	\]

	We claim that
	\begin{equation}\label{8.6}
			\Bigl(
		\bigcup_{i=1}^{k+1}a_iH
		\Bigr)
		\cdot
		\Bigl(
		\bigcup_{j=1}^{2k+1}b_jH
			\Bigr)
		\subseteq
		E_3(X;\mathfrak q).
	\end{equation}
	Indeed, let \(c\in a_iH\cdot b_jH\). For every
	\(y\in A'\cap a_iH\), 
	$
	d:=y^{-1}c\in b_jH\subseteq A'\cdot B'$,  {and} $	cd^{-1}=y\in A'.
$
	Therefore
	\[
	(1_{A'\cdot B'}*1_{A'})(c)
	=
	\sum_{d\in A'\cdot B'}1_{A'}(cd^{-1})
	\ge
	|A'\cap a_iH|
	\ge \frac{\varepsilon}{2}|G|.
	\]
	By Proposition~\ref{prop:criteria-E2-E3}(b.iii), we obtain
	\(c\in E_3(X;\mathfrak q)\). This proves \eqref{8.6}.

	We now split into two cases.
	
	\medskip
	\noindent
	\textit{Case 1: \(k\in\{1,2\}\), or \(k=0\) and \(a_1\in H\).}
	
	Assume first that $
	\bigcup_{j=1}^{k+1}a_jH$
 {and} $
	\bigcup_{j=1}^{2k+1}b_jH
	$
	are not complements of each other. 
	Since there are exactly \(3k+2\)  cosets of $H$   in \(G\),  there exists a coset \(b_0H\)
	which intersects neither union.
	
	 Since $2< 3\vartheta_0<3$, the lower bound in \ref{rough_ideal_lemma} with $\gamma=1/3$ applies  and gives a positive lower bound
	 \[
	 \begin{aligned}
	 	\mathcal N_{1/3}(X;b_0,H)
	 	:=
	 	&\#\biggl\{
	 	\mathfrak n\subset \mathcal O_K :
	 	\begin{array}{l}
	 		N\mathfrak n\le X,\ [\mathfrak n]\in b_0H,\\[2pt]
	 		N\mathfrak p>X^{1/3}\text{ for every prime ideal }
	 		\mathfrak p\mid \mathfrak n
	 	\end{array}
	 	\biggr\}
	 	\\
	 	\ge\;&
	 	\bigl(1+o_K(1)\bigr)
	 	\frac{2\log(3\vartheta_0-1)}{Y\vartheta_0}
	 	\frac{X}{\log X}.
	 \end{aligned}
	 \]
	Since \(b_0H\) lies outside \(	\bigcup_{j=1}^{k+1}a_jH\), Proposition~\ref{prop:criteria-E2-E3}(b.vi)
	implies that
	\[
 	\pi (X;b_0H)
	:=
	\# \bigl\{
	\mathfrak p\subset\mathcal O_K:
	\mathfrak p\nmid\mathfrak q,\ 
	N\mathfrak p\le X,\ 
	[\mathfrak p]\in b_0H
	\bigr\}=O(\varepsilon X/\log X).
	\]
 Taking \(\varepsilon>0\) sufficiently small,
	we conclude that a positive proportion of the   ideals $\mathfrak{n}$  counted in \(	\mathcal N_{1/3}(X;b_0,H)\) are
	products of exactly two prime ideals.
	Each such ideal \(\mathfrak n\) may be written as
\(	\mathfrak n=\mathfrak p\mathfrak q\)
	with
	$
	X^{1/3}<N\mathfrak p,\ N\mathfrak q\le X^{2/3}.
	$
	Since \(b_0H\) is disjoint from
	\(\bigl(\bigcup_{j=1}^{k+1}a_jH\bigr)^2,\)
	the two prime factors cannot both have classes in
	\(\bigcup_{j=1}^{k+1}a_jH\). Hence at least one of them, say
	\(\mathfrak p\), satisfies
	$
	[\mathfrak p]\notin \bigcup_{j=1}^{k+1}a_jH.
	$
	Thus the number of such ideals is at most the number of pairs
	\((\mathfrak p,\mathfrak q)\) with
	\[
	X^{1/3}<N\mathfrak p\le X^{2/3},\qquad
	[\mathfrak p]\notin \bigcup_{j=1}^{k+1}a_jH,
	\qquad
	N\mathfrak q\le \frac{X}{N\mathfrak p}.
	\]
	For fixed \(\mathfrak p\), the number of possible \(\mathfrak q\) is
	$
	\ll_K \frac{X}{N\mathfrak p\log X},
	$
	since \(X/N\mathfrak p\in [X^{1/3},X^{2/3})\). Therefore
	\[
	\frac{X}{\log X}
	\ll_{K,\varepsilon}
	\sum_{
			X^{1/3}<N\mathfrak p\le X^{2/3}\atop
			[\mathfrak p]\notin \bigcup_{j=1}^{k+1}a_jH}
	\frac{X}{N\mathfrak p\,\log X}.
	\]
	Hence
	\[
	\sum_{
			X^{1/3}<N\mathfrak p\le X^{2/3}\atop
			[\mathfrak p]\notin  \bigcup_{j=1}^{k+1} a_jH}
	\frac1{N\mathfrak p}
	\gg_{K,\varepsilon} 1.
	\]
	Let $B:= 
	\frac{\log X}{\log Q}$. By dyadic decomposition, there exists
	$
	\beta\in
	\bigl[
	\frac{B}{3},
	\frac{2B}{3}
	\bigr]
	$
	and a coset \(a_0H\) outside \(\bigcup_{j=1}^{k+1}a_jH\) such that
	\begin{equation}\label{8.7}
		\sum_{ 
				Q^{\beta-\varepsilon}<N\mathfrak p\le Q^\beta\atop
				[\mathfrak p]\in a_0H}
		\frac1{N\mathfrak p}
		\gg _{K,\varepsilon}1.
	\end{equation}
	Since \(B\ge1+\kappa\), this \(\beta\) lies in the admissible range
	\((2\varepsilon,B]\) of Proposition~\ref{prop:criteria-E2-E3}(b.iv).
	
	Next we  consider the case that the two unions are complements.
	If \(k=0\) and \(a_1\in H\), this cannot happen. Thus we may assume
	\(k\in\{1,2\}\), corresponding  to \([G:H]=5\) and
	\([G:H]=8\). The quotient \(G/H\) is clearly cyclic when \([G:H]=5\).
	In the case \([G:H]=8\), the quotient \(G/H\) is also cyclic in the
	complementary situation.\footnote{Let \(\Gamma=G/H\) and
	write \(S=\{x,y,z\}\). Then \(|S^2|=5\) and \(S\cap S^2=\varnothing\).
		If \(\Gamma\simeq C_2^3\), then \(|S^2|\le4\), impossible. If
		\(\Gamma\simeq C_4\times C_2\), the square map has image of size \(2\), so
		two of \(x^2,y^2,z^2\) are equal, say \(x^2=y^2\).  Since \(|S^2|=5\), this
		is the only repetition among the six unordered products. Thus
		\(S^2=\{x^2,z^2,xy,xz,yz\}\). Multiplying the elements of
		\(\Gamma=S\sqcup S^2\) gives
		$
		1=(xyz)(x^2z^2xyxz yz)=xyzx^2,
		$
		where we used \(x^2=y^2\). Hence \(yz=x\), contradicting
		\(S\cap S^2=\varnothing\).} 
		A direct check in the cyclic groups of orders \(5\) and \(8\) shows that 
		$
		S=\bigcup_{j=1}^{k+1}g^{k+j}H.
		$ for some generator
		\(gH\) of \(G/H\).
		Therefore, by Corollary~\ref{cor:dyadic-prime-escape}, there exist
	$
	\beta\in(2\varepsilon,4\alpha_0]
	$
	and a coset \(a_0H\) outside
	$
	\bigcup_{j=1}^{k+1}a_jH
	$
	such that
	\begin{equation}\label{8.8}
			\sum_{ 
				Q^{\beta-\varepsilon}<N\mathfrak p\le Q^\beta\atop
				[\mathfrak p]\in a_0H}
		\frac1{N\mathfrak p}
		\gg _{K,\varepsilon }1.
	\end{equation}	
	Since \(B>4\alpha_0\), this \(\beta\) also lies in the admissible range of
	Proposition~\ref{prop:criteria-E2-E3}(b.iv).
	
	In either subcase, let
$
	c\in
	a_0H\cdot
	\bigl(
	\bigcup_{j=1}^{2k+1}b_jH
	\bigr).
$
	Then \(c\in a_0H\cdot b_jH\) for some \(j\). Hence, for every prime ideal
	\(\mathfrak p\) with \([\mathfrak p]\in a_0H\), we have
	$
	[\mathfrak p]^{-1}c\in b_jH\subseteq A'\cdot B'.
	$
	Therefore, either \eqref{8.7} or \eqref{8.8} gives
	\[
	\sum_{ 
			[\mathfrak p]d=c,\,d\in A'\cdot B'\atop
			Q^{\beta-\varepsilon}<N\mathfrak p\le Q^\beta
			}
	\frac1{N\mathfrak p}
	\gg _{K,\varepsilon}1 
	\]
	with $\beta$ in the admissible range of 
 Proposition~\ref{prop:criteria-E2-E3}(b.iv). By Proposition~\ref{prop:criteria-E2-E3}(b.iv),
	\begin{equation}	\label{8.9}
		a_0H\cdot
	\biggl(
	\bigcup_{j=1}^{2k+1}b_jH
	\biggr)
	\subseteq
	E_3(X;\mathfrak q).
	\end{equation}
 
Combining \eqref{8.6} and \eqref{8.9}, we have
\[
\biggl(
\Bigl(\bigcup_{i=1}^{k+1}a_iH\Bigr)
\cup a_0H
\biggr)
\cdot
\biggl(
\bigcup_{j=1}^{2k+1}b_jH
\biggr)
\subseteq E_3(X;\mathfrak q).
\]
Since \(a_0H\) is outside
\(\bigcup_{i=1}^{k+1}a_iH\), the two factors contain \(k+2\) and
\(2k+1\) cosets of \(H\), respectively. Hence
\[
\biggl|
\biggl(\bigcup_{i=1}^{k+1}a_iH\biggr)\cup a_0H
\biggr|
+
\biggl|
\bigcup_{j=1}^{2k+1}b_jH
\biggr|
=
(3k+3)|H|
>
|G|.
\]
Therefore their product is all of \(G\). It follows that 
$
E_3(X;\mathfrak q)=G.
$
	
\medskip
\noindent
\textit{Case 2: \(k=0\) and \(a_1\notin H\).}

In this case,  \(H\) has index \(2\), and
$
A',B'\subseteq a_1H$,
$
A'\cdot B'=H
$.
Let \(\psi\) be the quadratic character of \(G\) with kernel \(H\). Thus \(\psi(\mathfrak p)=1\) if and only if \( [\mathfrak p]\in H\). 

By \eqref{8.6}, we already know that
$
	a_1H\subseteq E_3(X;\mathfrak q).
$
It remains to prove that
$
H\subseteq E_3(X;\mathfrak q).
$
We will need Lemma \ref{lem:quadratic-prime-dichotomy} and Corollary \ref{cor:quadratic-transference-AprimeBprime} from Section \ref{sec:exceptional}. 

 Let
	\[B_0:=\max(1,3\alpha,4\alpha_0)+\frac{\kappa}{2}.\]
	Then \(Q^{B_0}<  X\), and
	$
	B_0>2\alpha+\alpha_0$,
	$
	B_0>2\alpha_0.
	$
	Thus, we may choose
	\[
	\beta_*\in
	\bigl(\max\{2\alpha,B_0/2\},\,B_0-\alpha_0\bigr)
	\]
	and then choose \(c_*>0\) sufficiently small so that
	$
	2\alpha+2c_*\beta_*<B_0.
	$
	After decreasing \(\varepsilon\) if necessary, we may assume
	$
	\beta_*>2\alpha+5\varepsilon$ and 
	$
	\beta_*+2\varepsilon<B_0-\alpha_0.
	$
	
	We apply Lemma~\ref{lem:quadratic-prime-dichotomy} with these parameters. 
Suppose first that Lemma~\ref{lem:quadratic-prime-dichotomy}(i) holds. Then
\[
\sum_{ 
		Q^{\alpha+\varepsilon}<N\mathfrak p\le Q^{\beta_*}\atop
		\psi(\mathfrak p)=1}
\frac1{N\mathfrak p}
\ge c_*.
\]
By decomposing the interval
$
(Q^{\alpha+\varepsilon},Q^{\beta_*}]
$
into dyadic intervals of the form
$
(Q^{\beta-\varepsilon},Q^\beta]
$,
we find some
$
\beta\in(2\varepsilon,\beta_*]
$
such that
\begin{equation}\label{8.13}
	\sum_{ 
			Q^{\beta-\varepsilon}<N\mathfrak p\le Q^\beta\atop
			\psi(\mathfrak p)=1 }
	\frac1{N\mathfrak p}
	\gg_\varepsilon 1.
\end{equation}
Since
$
\beta_*< B_0<B
$, this \(\beta\) lies in the admissible range
\((2\varepsilon,B]\) of Proposition~\ref{prop:criteria-E2-E3}(b.iv).
For any \(c\in H\), if \(\psi(\mathfrak p)=1\), then
\([\mathfrak p]\in H\)  and hence
$
[\mathfrak p]^{-1}c\in H=A'\cdot B'.
$
Therefore \eqref{8.13} implies
\[
\sum_{ 
	[\mathfrak p]d=c,\,d\in A'\cdot B'\atop
	Q^{\beta-\varepsilon}<N\mathfrak p\le Q^\beta
}
\frac1{N\mathfrak p}
\ge 	\sum_{ 
		Q^{\beta-\varepsilon}<N\mathfrak p\le Q^\beta\atop
		\psi(\mathfrak p)=1} 
\frac1{N\mathfrak p}
\gg_\varepsilon 1.
\]
By Proposition~\ref{prop:criteria-E2-E3}(b.iv), we have 
$
c\in E_3(X;\mathfrak q).
$
Since $c$ was arbitrary,  
$
H\subseteq E_3(X;\mathfrak q).
$

Suppose now that Lemma~\ref{lem:quadratic-prime-dichotomy}(ii) holds. Then
there exists
$
M\in[Q^{\beta_*},Q^{B_0}]
$
such that
\begin{equation}\label{8.14}
	\#\{
	\mathfrak p:
	M<N\mathfrak p\le2M,\ 
	\psi(\mathfrak p)=1
	\}
	\gg_{K ,\varepsilon}
	M\rho_K L(1,\psi)\mathcal V_\psi(Q).
\end{equation}
By the choice of $B_0$ and $\beta_*$, $M$ lies in the admissible range of  
Corollary~\ref{cor:quadratic-transference-AprimeBprime}.

Similarly, for every \(c\in H\), if \(\psi(\mathfrak p)=1\), then
$
[\mathfrak p]^{-1}c\in H=A'\cdot B',
$
and so \eqref{8.14} gives
\[
\sum_{[\mathfrak p]d=c,\ M<N\mathfrak p\le 2M 
	\atop
\psi(\mathfrak p)=1	,\  d\in A'\cdot B'}
1
\gg_{K,\varepsilon}
M\rho_K L(1,\psi)\mathcal V_\psi(Q).
\]
Corollary~\ref{cor:quadratic-transference-AprimeBprime} then gives
$
c\in E_3(X;\mathfrak q).
$
Again \(c\in H\) was arbitrary, so
$
H\subseteq E_3(X;\mathfrak q).
$
\end{proof}

	\section{The exceptional quadratic case}\label{sec:exceptional}
	Let \(\psi\) be a quadratic Hecke character modulo \(\mathfrak q\). We write
	\[
	(1*\psi)(\mathfrak a)
	:=
	\sum_{\mathfrak d\mid \mathfrak a}\psi(\mathfrak d),
	\]
	which is nonnegative on integral ideals.

	As in Section \ref{three}, we assume throughout \eqref{CS} and \eqref{Lb} for fixed parameters $0<\alpha_0\le \alpha<1$. 
	Note that \eqref{Lb} implies \eqref{CS^b}.

	We shall use the ideal-theoretic von Mangoldt function
	\begin{equation}\label{mango}
			\Lambda_K(\mathfrak a)
		:=
		\begin{cases}
			\log N\mathfrak p, & \mathfrak a=\mathfrak p^m,\\
			0, & \text{otherwise},
		\end{cases}
	\end{equation}
	so that
	\[
	\sum_{\mathfrak d\mid\mathfrak a}\Lambda_K(\mathfrak d)
	=
	\log N\mathfrak a.
	\]
	
	Finally, define
	\[
	\mathcal V_\psi(Q)
	:=
	\prod_{\mathfrak p\mid\mathfrak q}
	\Bigl(1-\frac1{N\mathfrak p}\Bigr)
	\prod_{{
			2<N\mathfrak p\le Q\atop
			\psi(\mathfrak p)=1}}
	\Bigl(1-\frac2{N\mathfrak p}\Bigr).
	\]
 
	We shall use the standard beta-sieve fundamental lemma in sieve dimension
	\(\varkappa\); see \cite[Lem.~6.8]{opera}. Although it is often stated for
	rational integers, its proof is purely combinatorial and depends only on
	unique factorization into primes. Since the nonzero integral ideals of \(K\)
	form a free commutative monoid generated by the prime ideals, the same
	argument applies verbatim to ideals, with the norm \(N\mathfrak d\) replacing
	the integer \(d\).
	
	\begin{lemma}[Fundamental lemma of the sieve over ideals]
		\label{lem:fundamental-sieve-ideals}
		Let \(\varkappa\ge1\) be fixed. Let \(z\ge2\), put
		$
		\mathcal P(z):=\prod_{N\mathfrak p<z}\mathfrak p,
		$
		and let \(D=z^s\), where \(s\ge 9\varkappa+1\). There exist coefficients
		\(\lambda_{\mathfrak d}^{\pm}\), supported on squarefree ideals
		\(\mathfrak d\mid\mathcal P(z)\) with \(N\mathfrak d\le D\), such that
		$
		|\lambda_{\mathfrak d}^{\pm}|\le1
		$
		and the following hold.
		
		\begin{enumerate}
			\item[(i)]
			For every integral ideal \(\mathfrak n\),
			\[
		\sum_{\mathfrak d\mid\mathfrak n}
		\lambda_{\mathfrak d}^{-}
		\le
		1_{(\mathfrak n,\mathcal P(z))=1}
		\le
		\sum_{\mathfrak d\mid\mathfrak n}
		\lambda_{\mathfrak d}^{+}.
		\]
			
			\item[(ii)]
			Let \(h\) be a multiplicative function on squarefree ideals, with
			\(0\le h(\mathfrak p)<1\). Assume that for some \(K_0\ge1\),
			\[
			\prod_{w\le N\mathfrak p<y}
			\left(1-h(\mathfrak p)\right)^{-1}
			\le
			K_0
			\Bigl(\frac{\log y}{\log w}\Bigr)^{\varkappa}
			\]
			for all \(2\le w\le y \). Then
			\[
			\sum_{\mathfrak d\mid\mathcal P(z)}
			\lambda_{\mathfrak d}^{+}h(\mathfrak d)
			\le
			\bigl(1+O_{\varkappa,K_0}(e^{-s})\bigr)
			\prod_{N\mathfrak p<z}(1-h(\mathfrak p)),
			\]
			and
			\[
			\sum_{\mathfrak d\mid\mathcal P(z)}
			\lambda_{\mathfrak d}^{-}h(\mathfrak d)
			\ge
			\bigl(1-O_{\varkappa,K_0}(e^{-s})\bigr)
			\prod_{N\mathfrak p<z}(1-h(\mathfrak p)).
			\]
		\end{enumerate}
	\end{lemma}

	\begin{lemma}
		\label{lem:one-star-psi}
		Let \(\psi\) be a non-principal quadratic Hecke character modulo \(\mathfrak q\).
		For every \(\varepsilon>0\), there exists
		\(\eta=\eta(\varepsilon)>0\) such that, whenever
		$x\ge Q^{\alpha_0+\varepsilon},$
		we have
		\[
		\sum_{N\mathfrak a\le x}(1*\psi)(\mathfrak a)
		=
		\rho_K L(1,\psi)x
		+
		O_{K,\varepsilon}(x^{1-\eta}).
		\]
	\end{lemma}
	
	\begin{proof}
		Write
		\[
		S(x)
		:=
		\sum_{N\mathfrak a\le x}(1*\psi)(\mathfrak a)
		=
		\sum_{N\mathfrak b\,N\mathfrak c\le x}\psi(\mathfrak c).
		\]
		Let
	$
	\theta:=\frac{\alpha_0+\varepsilon/2}{\alpha_0+\varepsilon}$, 
	 $
	T_0:=x^\theta .
	$
	We split the sum according to whether \(N\mathfrak c\le T_0\) or
	\(N\mathfrak c>T_0\). This gives
		\begin{equation}\label{eq_S(x)}
				S(x)
			=
			\sum_{N\mathfrak c\le T_0}
			\psi(\mathfrak c)
			\#\{\mathfrak b:N\mathfrak b\le x/N\mathfrak c\}
			+
			\sum_{N\mathfrak b\le x/T_0}
			\sum_{T_0<N\mathfrak c\le x/N\mathfrak b}
			\psi(\mathfrak c).
		\end{equation}
		Recall the
			ideal-counting estimate \cite[Thm.~VI.3.3]{lang}
	$	\#\{\mathfrak a:N\mathfrak a\le T\}
	=
	\rho_K T+O_K\bigl(T^{1-1/n_K}\bigr).$
		Let \(\delta_K:=1/n_K\). Then
		\[
			\sum_{N\mathfrak c\le T_0}
			\psi(\mathfrak c)
			\#\bigl\{\mathfrak b:N\mathfrak b\le x/N\mathfrak c\bigr\}
		 =
			\rho_K x
			\sum_{N\mathfrak c\le T_0}
			\frac{\psi(\mathfrak c)}{N\mathfrak c}
		 +
			O_K \Bigl(
			x^{1-\delta_K}
			\sum_{N\mathfrak c\le T_0}
			(N\mathfrak c)^{-1+\delta_K}
			\Bigr),
		\]
		where the error term is
		$
		\ll_K
		x^{1-\delta_K}T_0^{\delta_K}
	=
	x^{1-\delta_K (1-\theta)}
		=
		x^{1-\eta_1} 
		$
		with \(\eta_1=\delta_K (1-\theta)>0 \).
		
		Since \(x\ge Q^{\alpha_0+\varepsilon}\), we have
		\(T_0\ge Q^{\alpha_0+\varepsilon/2}\).
		Hence, by \eqref{CS^b} applied with \(\varepsilon/2\), there exists
		\(\eta_2=\eta_2(\varepsilon)>0\) such that
		\begin{equation}\label{eq_A(T)}
				A(T):=\sum_{N\mathfrak a\le T}\psi(\mathfrak a)
			\ll_{K,\varepsilon} T^{1-\eta_2}
			\qquad (T\ge T_0).
		\end{equation}
		Partial summation gives, for \(Y>T_0\),
		\[
		\sum_{T_0<N\mathfrak c\le Y}\frac{\psi(\mathfrak c)}{N\mathfrak c}
		=
		\frac{A(Y)}{Y}-\frac{A(T_0)}{T_0}
		+\int_{T_0}^{Y}\frac{A(t)}{t^2}\,dt
		\ll_{K,\varepsilon} T_0^{-\eta_2}.
		\]
		Letting \(Y\to\infty\), we obtain
		$
		\sum_{N\mathfrak c>T_0}\frac{\psi(\mathfrak c)}{N\mathfrak c}
		\ll_{K,\varepsilon} T_0^{-\eta_2}.
		$
		Hence
		\[
		\sum_{N\mathfrak c\le T_0}
		\frac{\psi(\mathfrak c)}{N\mathfrak c}
		=
		L(1,\psi)+O_{K,\varepsilon}(x^{-\theta\eta_2}).
		\]
		  Therefore the first sum  in \eqref{eq_S(x)} is
		\begin{equation}\label{first_sum}
			\sum_{N\mathfrak c\le T_0}
			\psi(\mathfrak c)
			\#\bigl\{\mathfrak b:N\mathfrak b\le x/N\mathfrak c\bigr\}=
			\rho_K L(1,\psi)x + O_{K,\varepsilon}\bigl(x^{1-\min(\eta_1,\theta\eta_2)}\bigr).
		\end{equation}
		
		For the second part,   since
		\(x/N\mathfrak b\ge T_0\), by \eqref{eq_A(T)}, we have 
		\[
		\sum_{T_0<N\mathfrak c\le x/N\mathfrak b}
		\psi(\mathfrak c)
		=A(x/N\mathfrak b)-A(T_0)
		\ll_{K,\varepsilon}
		\Bigl(\frac{x}{N\mathfrak b}\Bigr)^{1-\eta_2}
		+
		T_0^{1-\eta_2}.
		\]
		
		Therefore, the second sum in \eqref{eq_S(x)} is
		\begin{align}
			\sum_{N\mathfrak b\le x/T_0}
			\sum_{T_0<N\mathfrak c\le x/N\mathfrak b}\psi(\mathfrak c)
			&\ll_{K,\varepsilon}
			x^{1-\eta_2}
			\sum_{N\mathfrak b\le x/T_0}(N\mathfrak b)^{-1+\eta_2}
			+
			T_0^{1-\eta_2}\#\bigl\{N\mathfrak b\le x/T_0\bigr\}\nonumber  \\
			&\ll_{K,\varepsilon}
			x^{1-\eta_2}\Bigl(\frac{x}{T_0}\Bigr)^{\eta_2}
			+
			T_0^{1-\eta_2}\frac{x}{T_0}  \,\,\ll
			x ^{1-\theta\eta_2}. \label{second_sum}
		\end{align}
		Let \(\eta=\min\{\eta_1,\theta\eta_2\}  >0\).  
		Combining \eqref{first_sum} and \eqref{second_sum}, we have  \[
		S(x) = 	\rho_K L(1,\psi)x + O_{K,\varepsilon}(x^{1- \eta }).
		\]  
	\end{proof}
	
	\begin{lemma}
		\label{lem:divisible-one-star-psi}
		Let \(\psi\) be a non-principal quadratic Hecke character modulo
		\(\mathfrak q\). For every \(\varepsilon>0\), there exists
		\(\eta=\eta(\varepsilon)>0\) such that, uniformly for every nonzero
		integral ideal \(\mathfrak d\), whenever
		$
		\frac{x}{N\mathfrak d}\ge Q^{\alpha_0+\varepsilon},
		$
		we have
		\[
		\sum_{{N\mathfrak a\le x\atop
				\mathfrak d\mid\mathfrak a}}
		(1*\psi)(\mathfrak a)
		=
		\rho_K L(1,\psi)\,h(\mathfrak d)\,x
		+
		O_{K,\varepsilon}\biggl(
		\tau_3(\mathfrak d)
		\Bigl(\frac{x}{N\mathfrak d}\Bigr)^{1-\eta}
		\biggr),
		\]
		where   \(	h \) is defined on nonzero integral ideals by
		\[
		h(\mathfrak d)
		:=
		\sum_{\mathfrak d=\mathfrak d_1\mathfrak f\mathfrak g}
		\frac{\mu(\mathfrak f)\psi(\mathfrak f\mathfrak g)}
		{N\mathfrak d_1 (N\mathfrak f)^2 N\mathfrak g}.
		\]
	Moreover, $h$ is multiplicative and for every prime ideal \(\mathfrak p\),
\begin{equation}\label{h(p)}
		h(\mathfrak p)
	=
	\frac{1+\psi(\mathfrak p)}{N\mathfrak p}
	-
	\frac{\psi(\mathfrak p)}{(N\mathfrak p)^2},
\end{equation}
	\begin{equation}\label{h(p^k)}
			|h(\mathfrak p^k)|
		\le
		\frac{ k+1}{(N\mathfrak p)^k}\qquad  (k\ge1 ).
	\end{equation}
	\end{lemma}
	
	\begin{proof}
		Write
		\[
		S_{\mathfrak d}(x)
		:=
		\sum_{{N\mathfrak a\le x\atop
				\mathfrak d\mid\mathfrak a}}
		(1*\psi)(\mathfrak a)=\sum_{{
				N\mathfrak m_1N\mathfrak m_2\le x\atop
				\mathfrak d\mid\mathfrak m_1\mathfrak m_2}}
		\psi(\mathfrak m_2).
		\]
		For a pair
		\((\mathfrak m_1,\mathfrak m_2)\) with
		\(\mathfrak d\mid \mathfrak m_1\mathfrak m_2\), put  
		$
		\mathfrak d_1=(\mathfrak d,\mathfrak m_1)$,
		$
		\mathfrak d_2=\mathfrak d/\mathfrak d_1.
		$
		Then \(\mathfrak d_2\mid\mathfrak m_2\), and writing
		\(\mathfrak m_1=\mathfrak d_1\mathfrak n_1\),
		\(\mathfrak m_2=\mathfrak d_2\mathfrak n_2\), we have
		\((\mathfrak n_1,\mathfrak d_2)=1\). Conversely, this condition gives a
		unique such pair  \((\mathfrak m_1,\mathfrak m_2)\). Hence
		\[
		S_{\mathfrak d}(x)
		=
		\sum_{\mathfrak d=\mathfrak d_1\mathfrak d_2}
		\psi(\mathfrak d_2)
		\sum_{{
				N\mathfrak n_1N\mathfrak n_2\le x/N\mathfrak d\atop
				(\mathfrak n_1,\mathfrak d_2)=1}}
		\psi(\mathfrak n_2).
		\]
		Writing \(\mathfrak d_2=\mathfrak f\mathfrak g\) and
		\(\mathfrak n_1=\mathfrak f\mathfrak r\), by Möbius inversion 
		$
		1_{(\mathfrak n_1,\mathfrak d_2)=1}
		=
		\sum_{\mathfrak f\mid(\mathfrak n_1,\mathfrak d_2)}
		\mu(\mathfrak f),
		$
		we have 
		\begin{equation}\label{9.1}
				S_{\mathfrak d}(x)
			=
			\sum_{\mathfrak d=\mathfrak d_1\mathfrak f\mathfrak g}
			\mu(\mathfrak f)\psi(\mathfrak f\mathfrak g)
			\sum_{N\mathfrak m\le x/(N\mathfrak d\,N\mathfrak f)}
			(1*\psi)(\mathfrak m).
		\end{equation}
		
	Let
	\[
	X:=\frac{x}{N\mathfrak d},\qquad \gamma:=\frac{\varepsilon}{4(\alpha_0+\varepsilon)}.
	\]
	For the terms with \(N\mathfrak f\ge X^\gamma\),  we use the elementary bound
	$
	\sum_{N\mathfrak m\le y}(1*\psi)(\mathfrak m)
	\ll_K y\log(2y).
	$
	To see this, using
	$
	|(1*\psi)(\mathfrak m)|\le \tau_K(\mathfrak m),
	$ then
	\[
	\sum_{N\mathfrak m\le y}|(1*\psi)(\mathfrak m)|
	\le
	\sum_{N\mathfrak m\le y}\tau_K(\mathfrak m)
	=\sum_{N\mathfrak r\le y}\#\bigl\{\mathfrak s: N \mathfrak s\leq y/N \mathfrak r \bigr\}
	\ll_K
	\sum_{N\mathfrak r\le y}\frac y{N\mathfrak r}
	\ll_K y\log(2y).
	\]
	Therefore, the  total contribution of the terms with \(N\mathfrak f\ge X^\gamma\) to \eqref{9.1} is
	\[
	\sum_{{\mathfrak d=\mathfrak d_1\mathfrak f\mathfrak g\atop
			N\mathfrak f\ge X^\gamma}}
	\biggl|
	\sum_{N\mathfrak m\le X/N\mathfrak f}
	(1*\psi)(\mathfrak m)
	\biggr|  \ll_K
	\sum_{{\mathfrak d=\mathfrak d_1\mathfrak f\mathfrak g\atop
			N\mathfrak f\ge X^\gamma}}
	\frac{X}{N\mathfrak f}\log(2X)
	\ll_K
	\tau_3(\mathfrak d)X^{1-\gamma}\log(2X)	\ll_{K,\varepsilon}
	\tau_3(\mathfrak d)X^{1-\gamma/2}.
	\]

	Now suppose \(N\mathfrak f<X^\gamma\). Since \(X\ge Q^{\alpha_0+\varepsilon}\),
	we have
	$
	\frac{X}{N\mathfrak f}
	>
	X^{1-\gamma}
	\ge
	Q^{\alpha_0+3\varepsilon/4}.
	$
	Applying Lemma~\ref{lem:one-star-psi} with \(3\varepsilon/4\), we have, for some $\eta_0=\eta(\varepsilon)>0$, 
	\[
	\sum_{N\mathfrak m\le X/N\mathfrak f}
	(1*\psi)(\mathfrak m)
	=
	\rho_K L(1,\psi)\frac{X}{N\mathfrak f}
	+
	O_{K,\varepsilon}\biggl(
	\Bigl(\frac{X}{N\mathfrak f}\Bigr)^{1-\eta_0}
	\biggr).
	\]
	Therefore the contribution of the  terms with \(N\mathfrak f<X^\gamma\) is
	\[
	\rho_K L(1,\psi)x
	\sum_{{\mathfrak d=\mathfrak d_1\mathfrak f\mathfrak g\atop
			N\mathfrak f<X^\gamma}}
	\frac{\mu(\mathfrak f)\psi(\mathfrak f\mathfrak g)}
	{N\mathfrak d_1 (N\mathfrak f)^2 N\mathfrak g}
	+
	O_{K,\varepsilon}\bigl(
	\tau_3(\mathfrak d)X^{1-\eta_0}
	\bigr).
	\]
		
	We may extend the main term to all factorizations
	\(\mathfrak d=\mathfrak d_1\mathfrak f\mathfrak g\).  Indeed, the omitted
	part with \(N\mathfrak f\ge X^\gamma\) is
	\begin{equation}\label{omitted}
			\ll_K
		L(1,\psi)x
		\sum_{{\mathfrak d=\mathfrak d_1\mathfrak f\mathfrak g\atop
				N\mathfrak f\ge X^\gamma}}
		\frac{1}{N\mathfrak d\,N\mathfrak f}
		\ll_K
		L(1,\psi)\tau_3(\mathfrak d)X^{1-\gamma}.
	\end{equation}
	We shall use the standard bound 
	$
	L(1,\psi)\ll_{K,\varepsilon'} Q^{\varepsilon'}
	$  for any \(\varepsilon'>0  \). 
	Indeed, for the primitive character \(\psi^*\) inducing \(\psi\) with
	conductor \(\mathfrak f_\psi\),  \cite[Lem.~2.1]{zaman}  applied at \(s=1\), gives
	$
	L(1,\psi^*)\ll_{K,\varepsilon'}(N\mathfrak f_\psi)^{\varepsilon'}
	\ll_{K,\varepsilon'}Q^{\varepsilon'}.
	$
	The missing Euler factors between \(L(s,\psi)\) and \(L(s,\psi^*)\) contribute
	only \(Q^{o_K(1)}\), by \cite[Lem.~2.4]{zaman}. After relabeling
	\(\varepsilon'\), this gives the stated bound.

	Using this bound
	with \(\varepsilon'>0\) sufficiently small, since $X\geq Q^{\alpha_0+\varepsilon}$, \eqref{omitted} is 
	$
	\ll_{K,\varepsilon}
	\tau_3(\mathfrak d)X^{1-\gamma/2}.
    $
		
	Combining the above estimates, we have
	 \[
	 S_{\mathfrak d}(x)
	 =
	 \rho_KL(1,\psi)h(\mathfrak d)x
	 +
	 O_{K,\varepsilon}
	 \bigl(
	 \tau_3(\mathfrak d)X^{1-\eta}
	 \bigr),
	 \]
	 for some \(\eta>0\), where
	 \[
	 h(\mathfrak d)
	 :=
	 \sum_{\mathfrak d=\mathfrak d_1\mathfrak f\mathfrak g}
	 \frac{\mu(\mathfrak f)\psi(\mathfrak f\mathfrak g)}
	 {N\mathfrak d_1(N\mathfrak f)^2N\mathfrak g}.
	 \]
	 
			The function \(h\) is multiplicative, since it is defined by a finite
		Dirichlet convolution of multiplicative functions on ideals.
		
		For \(\mathfrak d=\mathfrak p^k\), since \(\mu(\mathfrak p^e)=0\) for
		\(e\ge2\), we  have \[
		h(\mathfrak p^k)
		=
		\sum_{j=0}^{k}
		\frac{\psi(\mathfrak p^j)}{(N\mathfrak p)^k}
		-
		\sum_{j=0}^{k-1}
		\frac{\psi(\mathfrak p^{j+1})}{(N\mathfrak p)^{k+1}}.
		\]
			In particular,
		\[
		h(\mathfrak p)
		=
		\frac{1+\psi(\mathfrak p)}{N\mathfrak p}
		-
		\frac{\psi(\mathfrak p)}{(N\mathfrak p)^2},
		\]
		We have  \(\psi(\mathfrak p)\in\{0,\pm1\}\). 
		If \(\psi(\mathfrak p)=1\), then $ h(\mathfrak p^k) = \frac{k+1}{(N\mathfrak p)^k}-\frac{k}{(N\mathfrak p)^{k+1}}, $ while if \(\psi(\mathfrak p)=0\), then \(h(\mathfrak p^k)=\frac1{(N\mathfrak p)^k}\). Finally, if \(\psi(\mathfrak p)=-1\), then   \(\bigl|
		\sum_{j=0}^{k}\psi(\mathfrak p^j)
		\bigr|,
		\,
		\bigl|
		\sum_{j=1}^{k}\psi(\mathfrak p^j)
	\bigr|
		\le 1\), and hence $ |h(\mathfrak p^k)| \le \frac1{(N\mathfrak p)^k}+\frac1{(N\mathfrak p)^{k+1}}. $ Therefore, in any case,  
	    \[|h(\mathfrak p^k)|
	    \le
	    \frac{ k+1}{(N\mathfrak p)^k}.\]
	\end{proof}

	\begin{lemma}
		\label{lem:h-euler-product}
		Let \(h\) be as in Lemma~\ref{lem:divisible-one-star-psi}. Then
		\(0\le h(\mathfrak p)<1\) for every prime ideal \(\mathfrak p\), and
		\(h\) satisfies the sieve dimension condition  with \(\varkappa=2\); that is,
		there exists \(K_0 >1\) such that for all \(2\le w\le z \),
		\[
		\prod_{w\le N\mathfrak p<z}
		(1-h(\mathfrak p))^{-1}
		\le
		K_0\Bigl(\frac{\log z}{\log w}\Bigr)^2. 
		\]
		  And for any fixed \(\varepsilon>0\),
		\[
		\prod_{N\mathfrak p<Q^\varepsilon}
		(1-h(\mathfrak p))
		\asymp_{K,\varepsilon}
		\mathcal V_\psi(Q) :=
		\prod_{\mathfrak p\mid\mathfrak q}
		\Bigl(1-\frac1{N\mathfrak p}\Bigr)
		\prod_{{
				2<N\mathfrak p\le Q\atop
				\psi(\mathfrak p)=1}}
		\Bigl(1-\frac2{N\mathfrak p}\Bigr).
		\]
	\end{lemma}
	
	\begin{proof}
		If \(\mathfrak p\mid\mathfrak q\), then \(\psi(\mathfrak p)=0\), and  
		$
		h(\mathfrak p)=\frac1{N\mathfrak p}.
		$
		If \(\psi(\mathfrak p)=-1\), then
		$
		h(\mathfrak p)=\frac1{(N\mathfrak p)^2}.
		$
		Finally, if \(\psi(\mathfrak p)=1\), then
		$
		h(\mathfrak p)
		=
		\frac2{N\mathfrak p}-\frac1{N\mathfrak p^2}$, and 
		$
		1-h(\mathfrak p)
		=
		\bigl(1-\frac1{N{\mathfrak p}}\bigr)^2$.
		In any case,  \(0\le h(\mathfrak p)<1\) and
		$
		(1-h(\mathfrak p))^{-1}
		\le
		\bigl(1-\frac1{N\mathfrak p}\bigr)^{-2}.
		$
		Therefore, by Mertens' theorem   \cite[Thm.~2]{merten} (see \eqref{merten}),
		\[
		\prod_{w\le N\mathfrak p<z}
		(1-h(\mathfrak p))^{-1}
		\le
		\prod_{w\le N\mathfrak p<z}
		\Bigl(1-\frac1{N\mathfrak p}\Bigr)^{-2}
		\ll_K
		\left(\frac{\log z}{\log w}\right)^2.
		\]
		This proves the sieve dimension assertion with \(\varkappa=2\).
		
		We now  compare the Euler product with \(\mathcal V_\psi(Q)\). The factors
		corresponding to prime ideals with \(\psi(\mathfrak p)=-1\) satisfy 
		\[
		1\ge \prod_{{N\mathfrak p<Q^\varepsilon\atop
				\psi(\mathfrak p)=-1}}
		\Bigl(1-\frac1{(N\mathfrak p)^2}\Bigr)\gg _{K} 1,
		\] since \(\sum_{\mathfrak p}(N\mathfrak p)^{-2}\ll_K 1  \) and hence the product is bounded below by a positive constant depending
		only on \(K\).
		For prime ideals with \(\psi(\mathfrak p)=1\) and \(N\mathfrak p>2\), we have
		\[
		1-h(\mathfrak p)
		=
		\Bigl(1-\frac1{N\mathfrak p}\Bigr)^2=\Bigl(1-\frac2{N\mathfrak p}\Bigr)\biggl(1+O\Bigl(\frac1{(N\mathfrak p)^2}\Bigr)\biggr)
		\]
	Therefore, by the absolute convergence of
	\(\sum_{\mathfrak p}(N\mathfrak p)^{-2}\),
		\[
		\prod_{{N\mathfrak p<Q^\varepsilon\atop
				\psi(\mathfrak p)=1}}
		(1-h(\mathfrak p))
		\asymp_K
		\prod_{{2<N\mathfrak p<Q^\varepsilon\atop
				\psi(\mathfrak p)=1}}
		\Bigl(1-\frac2{N\mathfrak p}\Bigr).
		\]
	The finitely many prime ideals with \(N\mathfrak p=2\) are absorbed into
	the implied constants.

		 Finally, we compare the cutoffs \(Q^\varepsilon\) and \(Q\).
		By Mertens theorem in the form \cite[Lem.~2.4]{merten}, there is a constant $B_K>0$ such that
		\begin{equation}\label{merten_reciprocal}
			\sum_{N\mathfrak p\le z}\frac1{N\mathfrak p}
			=
			\log\log z+B_K+O_K\Bigl(\frac{1}{\log z}\Bigr).
		\end{equation}
		Thus
		$
		\sum_{Q^\varepsilon\le N\mathfrak p\le Q}\frac1{N\mathfrak p}
		\ll_{K,\varepsilon}1,
		$
		and therefore
		\[
		\prod_{{Q^\varepsilon\le N\mathfrak p\le Q\atop
				\psi(\mathfrak p)=1}}
		\left(1-\frac2{N\mathfrak p}\right)
		\asymp_{K,\varepsilon}1.
		\]	
	Next, consider the factors corresponding to prime ideals dividing
	\(\mathfrak q\). Since \(N\mathfrak p\le Q\), we have 
		\[
			\prod_{{\mathfrak p\mid\mathfrak q\atop  N\mathfrak p\ge Q^\varepsilon}}
		\Bigl(1-\frac1{N\mathfrak p}\Bigr) 
		\ge
		\prod_{Q^\varepsilon\le N\mathfrak p\le Q}
		\Bigl(1-\frac1{N\mathfrak p}\Bigr)
		\gg_{K,\varepsilon}1.
		\]

		Combining these estimates gives
		\[
		\prod_{N\mathfrak p<Q^\varepsilon}
		(1-h(\mathfrak p))
		\asymp_{K,\varepsilon}
		\prod_{\mathfrak p\mid\mathfrak q}
		\Bigl(1-\frac1{N\mathfrak p}\Bigr)
		\prod_{{2<N\mathfrak p\le Q\atop
				\psi(\mathfrak p)=1}}
		\Bigl(1-\frac2{N\mathfrak p}\Bigr)
		=
		\mathcal V_\psi(Q).
		\]
	\end{proof}

\begin{lemma}
	\label{lem:upper-bound-psi-one-primes}
	Let \(\psi\) be a non-principal quadratic Hecke character
	modulo \(\mathfrak q\). For every sufficiently small constant \(\varepsilon>0\), there exists an ineffective
	constant \(C_{K,\varepsilon}>0\) such that, whenever
	$
	y\ge Q^{\alpha_0+\varepsilon},
	$
	we have
	\[
	\#\bigl\{\mathfrak p:N\mathfrak p\le y,\ \psi(\mathfrak p)=1\bigr\}
	\le
	C_{K,\varepsilon}
	\,y\,\rho_KL(1,\psi)\mathcal V_\psi(Q).
	\]
\end{lemma}

\begin{proof}
	Let \(s_0\ge19\) be fixed and put
	\[
	z=Q^{\varepsilon/(2s_0)},
	\qquad
	D=z^{s_0}=Q^{\varepsilon/2}.
	\]
	Let \(\lambda_{\mathfrak d}^{+}\) be the upper-bound beta-sieve weights
	from Lemma~\ref{lem:fundamental-sieve-ideals}.
	
	Since \(\psi\) is quadratic, \((1*\psi)(\mathfrak a)\ge0\) for every
 \(\mathfrak a\). Moreover, if 
	 \(\psi(\mathfrak p)=1\), then
	$
	(1*\psi)(\mathfrak p)=1+\psi(\mathfrak p)=2.
	$
	The prime ideals with \(N\mathfrak p<z\) contribute \(O_K(z)\) and the
	remaining primes satisfies \((\mathfrak p,\mathcal P(z))=1\),   hence
	\[
	\begin{aligned}
		\#\{\mathfrak p:N\mathfrak p\le y,\ \psi(\mathfrak p)=1\}
		&\ll_K
		z+
		\sum_{N\mathfrak a\le y}
		(1*\psi)(\mathfrak a)
		1_{(\mathfrak a,\mathcal P(z))=1}  \\
		&\le
		O_K(z)+
		\sum_{\mathfrak d\mid\mathcal P(z)}
		\lambda_{\mathfrak d}^{+}
		\sum_{{N\mathfrak a\le y\atop
				\mathfrak d\mid\mathfrak a}}
		(1*\psi)(\mathfrak a).
	\end{aligned}
	\]
	Since  \(\lambda_{\mathfrak d}^{+}\) is supported on
	\(N\mathfrak d\le D=Q^{\varepsilon/2}\), we have
	$
	\frac{y}{N\mathfrak d}
	\ge
	Q^{\alpha_0+\varepsilon/2}.
	$
	Applying Lemma~\ref{lem:divisible-one-star-psi}   with \(\varepsilon/2\)  gives
	\[
	\sum_{{N\mathfrak a\le y\atop
			\mathfrak d\mid\mathfrak a}}
	(1*\psi)(\mathfrak a)
	=
	\rho_K L(1,\psi)  h(\mathfrak d)y
	+
	O_{K,\varepsilon}\biggl(
	\tau_3(\mathfrak d) 
	\Bigl(\frac{y}{N\mathfrak d}\Bigr)^{1-\eta}
	\biggr)
	\]
	for some \(\eta>0\). Therefore
	\[ 
	\sum_{\mathfrak d\mid\mathcal P(z)}
	\lambda_{\mathfrak d}^{+}
	\sum_{{N\mathfrak a\le y\atop
			\mathfrak d\mid\mathfrak a}}
	(1*\psi)(\mathfrak a)
	 =
	\rho_K L(1,\psi)y
	\sum_{\mathfrak d\mid\mathcal P(z)}
	\lambda_{\mathfrak d}^{+}h(\mathfrak d)
	+
	O_{K,\varepsilon }(y^{1-\eta/2}),
	\]
	where  we use that 
	$
	\sum_{N\mathfrak d\le D}
	\frac{\tau_3(\mathfrak d)}{(N\mathfrak d)^{1-\eta}}
	\ll_{K,\varepsilon}
	D^\eta(\log D)^{O_K(1)} \ll_{K,\varepsilon} y^{\eta/2} 
	$ for the error term.
	
We now apply the fundamental lemma  Lemma~\ref{lem:fundamental-sieve-ideals} with sieve dimension
\(\varkappa=2\).
Since \(s_0\ge19\), by Lemma \ref{lem:h-euler-product} and Lemma~\ref{lem:fundamental-sieve-ideals}, we have 
\[
\sum_{\mathfrak d\mid\mathcal P(z)}
\lambda_{\mathfrak d}^{+}h(\mathfrak d)
\ll_K
\prod_{N\mathfrak p<z}(1-h(\mathfrak p))\asymp_{K,\varepsilon}
\mathcal V_\psi(Q).
\]
Hence
\[
\#\bigl\{
\mathfrak p:
N\mathfrak p\le y,\
\psi(\mathfrak p)=1
\bigr\}
\ll_{K,\varepsilon}
y\rho_KL(1,\psi)\mathcal V_\psi(Q)
+
y^{1-\eta/2}
+
z.
\]

By the (ineffective)  Siegel-type lower bound for real nontrivial Hecke characters
\cite[Lem.~10]{Mitsui}, for every \(\delta>0\) one has
\begin{equation}\label{siegel}
	L(1,\psi)\gg_{K,\delta}Q^{-\delta}.
\end{equation}
 Also, by  Mertens' theorem in the form \cite[Lem.~2.4]{merten} (see \eqref{merten_reciprocal}),
\begin{equation}\label{bound_V_psi}
\mathcal V_\psi(Q)\gg_K(\log Q)^{-3}\gg_{K,\delta}Q^{-\delta}.
\end{equation}
Choosing \(\delta>0\) sufficiently small so that $
2\delta 
<
\min\bigl\{
\frac{(\alpha_0+\varepsilon)\eta}{2},
\alpha_0+\varepsilon-\frac{\varepsilon}{2s_0}
\bigr\} 
$, and since
\(y\ge Q^{\alpha_0+\varepsilon}\),
\(z=Q^{\varepsilon/(2s_0)}\),  the terms \(y^{1-\eta/2}\) and \(z\) are
absorbed into the main term. This proves the lemma.
\end{proof}
	
\begin{lemma}
	\label{lem:quadratic-prime-dichotomy}
Let \(B_0>0\) be fixed and suppose
	$
	B_0>\max\{2\alpha+\alpha_0, 2\alpha_0\}.
	$
	Choose
	\[
	\beta_*\in
	\bigl(\max\{2\alpha,B_0/2\},\,B_0-\alpha_0\bigr),
	\qquad
	c_*>0,
	\]
	such that
	\[
	2\alpha+2c_*\beta_*<B_0.
	\]
	Let \(\varepsilon>0\) be sufficiently small so that
	\[
	\beta_*>2\alpha+5\varepsilon,
	\qquad
	\beta_*+2\varepsilon<B_0-\alpha_0.
	\]
	Let \(\psi\) be a non-principal quadratic Hecke character modulo
	\(\mathfrak q\). Then at least one of the following holds.
	
	\begin{enumerate}
		\item[(i)]	
		\[
		\sum_{Q^{\alpha+\varepsilon}<N\mathfrak p\le Q^{\beta_*} \atop \psi(\mathfrak p)=1
			}
		\frac1{N\mathfrak p}
		\ge c_*.
		\]
		\item[(ii)]
		There exists
	$
		M\in[Q^{\beta_*},Q^{B_0}]
	$
		such that
		\[
		\#\bigl\{
		\mathfrak p:
		M<N\mathfrak p\le2M,\ 
		\psi(\mathfrak p)=1
		\bigr\}
		\gg_{K, \varepsilon}
		M\,\rho_K L(1,\psi)\,\mathcal V_\psi(Q).
		\]
	\end{enumerate}
\end{lemma}

\begin{proof}
	Put
	$
	Y:=Q^{B_0}.
$
	Let \(s_0\ge 19\) be a constant  to be chosen sufficiently large later, and put
	\[
	w:=Q^{\varepsilon/s_0},
	\qquad
	D:=w^{s_0}=Q^\varepsilon.
	\]
	Let \(\lambda_{\mathfrak e}^{\pm}\) be the beta-sieve weights from
	Lemma~\ref{lem:fundamental-sieve-ideals}, with sifting range \(z=w\) and
	level \(D\).

	Define \[a_{\mathfrak n}:=(1*\psi)(\mathfrak n)1_{(\mathfrak n,\mathcal P(w))=1}.\] Let \(\Lambda_K\) denote the von Mangoldt function on ideals as defined in \eqref{mango}. 
	Using
	$
	\sum_{\mathfrak d\mid\mathfrak n}\Lambda_K(\mathfrak d)
	=
	\log N\mathfrak n,
	$
	we have 
	\begin{equation}\label{9.2}
	 S(Y):=  \sum_{N\mathfrak d\le Y}
		\Lambda_K(\mathfrak d)
		\sum_{N\mathfrak n\le Y\atop
				\mathfrak d\mid\mathfrak n}
	 a_{\mathfrak n}
		=
		\sum_{N\mathfrak n\le Y }a_{\mathfrak n}\log N\mathfrak n.
	\end{equation}
	We first lower-bound the right-hand side of \eqref{9.2}. Since \(\psi\) is quadratic, \(1*\psi\ge0\), and hence \(a_{\mathfrak n}\ge0\). Using \(1*\psi\ge0\) and
	$	1_{(\mathfrak n,\mathcal P(w))=1}
	\ge
	\sum_{\mathfrak e\mid\mathfrak n}
	\lambda_{\mathfrak e}^{-},$
	we have
	\[
	\sum_{N\mathfrak n\le Y }a_{\mathfrak n}\log N\mathfrak n
	\ge
	\sum_{\mathfrak e\mid\mathcal P(w)}
	\lambda_{\mathfrak e}^{-}
	\sum_{{N\mathfrak n\le Y\atop
			\mathfrak e\mid\mathfrak n}}
	(1*\psi)(\mathfrak n)\log N\mathfrak n.
	\]
	For \(N\mathfrak e\le D=Q^\varepsilon\), we have
	$
	\frac{Y}{N\mathfrak e}
	\ge
	Q^{B_0-\varepsilon}
	\ge
	Q^{\alpha_0+\varepsilon}
	$
	after decreasing \(\varepsilon\), since \(B_0>2\alpha_0\ge\alpha_0\). Applying
	Lemma~\ref{lem:divisible-one-star-psi} and partial summation, there is   \(\eta'>0\) such that uniformly for all such $\mathfrak{e}$, 
	\[
	\sum_{{N\mathfrak n\le Y\atop
			\mathfrak e\mid\mathfrak n}}
	(1*\psi)(\mathfrak n)\log N\mathfrak n
	=
	\rho_KL(1,\psi)h(\mathfrak e)(Y\log Y-Y)
	+
	O_{K,\varepsilon }\bigl(
	\tau_3(\mathfrak e) Y^{1-\eta'}
	\bigr).
	\]
By Lemma~\ref{lem:h-euler-product}, the multiplicative function \(h\) satisfies the sieve-dimension hypothesis of Lemma~\ref{lem:fundamental-sieve-ideals} with dimension \(2\). Since \(s_0\ge19=9\cdot2+1\), the fundamental lemma gives
\[
\sum_{\mathfrak e\mid\mathcal P(w)}
\lambda_{\mathfrak e}^{-}h(\mathfrak e)
\ge
\bigl(1-O_K(e^{-s_0})\bigr)V(w),\qquad 	\sum_{\mathfrak e\mid\mathcal P(w)}
\lambda_{\mathfrak e}^{+}h(\mathfrak e)
\le
\bigl(1+O_K(e^{-s_0})\bigr)V(w),
\]
where
$
V(w):=\prod_{N\mathfrak p<w}(1-h(\mathfrak p)) 
$.

Summing over
\(\mathfrak e\mid\mathcal P(w)\), the total contribution of the error terms is
\[\ll_{K,\varepsilon}
Y^{1-\eta'}
\sum_{N\mathfrak e\le D}\tau_3(\mathfrak e)
\ll_{K,\varepsilon}
Y^{1-\eta'}D^{1+o(1)}.\]
Since \(D=Q^\varepsilon\), \(Y=Q^{B_0}\), after decreasing $\varepsilon$ once more, and then choosing the exponent $\delta$ in the Siegel lower bound \eqref{siegel} sufficiently small, this error is absorbed into the main term. 
Moreover, since
$
Y\log Y-Y=(1-o(1))B_0\,Y\log Q,
$ we obtain
\begin{equation}\label{9.3}
 	\sum_{N\mathfrak n\le Y }a_{\mathfrak n}\log N\mathfrak n
	\ge \bigl(1-O_K(e^{-s_0}) -o_{K,\varepsilon,s_0}(1)\bigr) B_0
	Y\log Q\,\rho_KL(1,\psi)   V(w).
\end{equation}

	We now prove an upper bound for $S(Y)$. First we estimate the contribution of the terms with \(N\mathfrak d\le Q^{\beta_*}\) in \eqref{9.2}. 
	
	Using the upper-bound sieve 
	$
	1_{(\mathfrak n,\mathcal P(w))=1}
	\le
	\sum_{\mathfrak e\mid\mathfrak n}
	\lambda_{\mathfrak e}^{+},
	$
	and that \((\mathfrak d,\mathfrak e)=1\) when \(\mathfrak e\mid\mathcal P(w)\), we have
	\[
\sum_{{N\mathfrak d\le Q^{\beta_*}\atop
				(\mathfrak d,\mathcal P(w))=1}}
		\Lambda_K(\mathfrak d)
		\sum_{N\mathfrak n\le Y\atop
				\mathfrak d\mid\mathfrak n   }
		a_{\mathfrak n}
		\le
		\sum_{{N\mathfrak d\le Q^{\beta_*}\atop
				(\mathfrak d,\mathcal P(w))=1}}
		\Lambda_K(\mathfrak d)
		\sum_{\mathfrak e\mid\mathcal P(w)}
		\lambda_{\mathfrak e}^{+}
		\sum_{{N\mathfrak n\le Y\atop
				\mathfrak d\mathfrak e\mid\mathfrak n}}
		(1*\psi)(\mathfrak n).
	\]
	For \(N\mathfrak d\le Q^{\beta_*}\) and \(N\mathfrak e\le D=Q^\varepsilon\),
$
	\frac{Y}{N\mathfrak d\,N\mathfrak e}
	\ge
	Q^{B_0-\beta_*-\varepsilon}
	>
	Q^{\alpha_0+\varepsilon},
	$
	by our choice of \(\varepsilon\). Applying  Lemma~\ref{lem:divisible-one-star-psi}, the total error is again absorbed into
	the main term as before. Therefore, the contribution of the terms with \(N\mathfrak d\le Q^{\beta_*}\) in $S(Y)$ is 
	\begin{equation}\label{9.4}
		\le
		(1+o_{K,\varepsilon,s_0}(1))Y\rho_KL(1,\psi)
		\Bigl(
		\sum_{\mathfrak e\mid\mathcal P(w)}
		\lambda_{\mathfrak e}^{+}h(\mathfrak e)
		\Bigr)
		\biggl(
		\sum_{ {N\mathfrak d\le Q^{\beta_*}\atop
				(\mathfrak d,\mathcal P(w))=1}}
		\Lambda_K(\mathfrak d)h(\mathfrak d)
		\biggr).
	\end{equation}

	We next bound the inner sum over \(\mathfrak d\) in \eqref{9.4}. The contribution of prime powers \(\mathfrak d=\mathfrak p^k\), \(k\ge2\), is
	\(O_K(1)\), since
	$
	|h(\mathfrak p^k)|\le \frac{k+1}{(N\mathfrak p)^k} 
	$
	by \eqref{h(p^k)}.
	For primes,  if \(\psi(\mathfrak p)=1\), then  $
	h(\mathfrak p)
	=
	\frac2{N\mathfrak p}-\frac1{N\mathfrak p^2} \le  \frac2{N\mathfrak p} $.    If \(\psi(\mathfrak p)=-1\), then
	$
	h(\mathfrak p)=\frac1{(N\mathfrak p)^2} .
	$ And if \(\mathfrak p\mid\mathfrak q\), then
	 $ 
	h(\mathfrak p)=\frac1{N\mathfrak p} 
	$. Moreover, the primes $\mathfrak{p}$ with \((\mathfrak p,\mathcal P(w))=1\) have \(N\mathfrak p\ge w\).
	Therefore, by \eqref{merten_reciprocal},
	the primes with \(\psi(\mathfrak p)=-1\) contribute \(O_{K,\varepsilon,s_0}(1)\). For primes $\mathfrak{p}$ with \(\mathfrak p\mid\mathfrak q\), the contribution is  at most 
	\[	\sum_{ {\mathfrak p\mid\mathfrak q\atop
			N\mathfrak p\ge w}}
	\frac{\log N\mathfrak p}{N\mathfrak p}
	\le
	\frac1w\sum_{\mathfrak p\mid\mathfrak q}\log N\mathfrak p
	\le
	\frac{\log Q}{w}
	=o_{\varepsilon,s_0}(\log Q).\]
	We now split the sum at 
 \(Q^{\alpha+\varepsilon}\). In the range \(N\mathfrak d\le Q^{\alpha+\varepsilon}\), by the above estimates and \eqref{merten_reciprocal}, we have 
	\[
	\begin{aligned}
		\sum_{ {N\mathfrak d\le Q^{\alpha+\varepsilon}\atop
				(\mathfrak d,\mathcal P(w))=1}}
		\Lambda_K(\mathfrak d)h(\mathfrak d)
		&\le
		 \sum_{ {N\mathfrak p\le Q^{\alpha+\varepsilon}\atop
				\psi(\mathfrak p)=1}}
		\Lambda_K(\mathfrak p)h(\mathfrak p)
		+o_{K,\varepsilon,s_0}(\log Q)  \\
		&\le
		2\sum_{N\mathfrak p\le Q^{\alpha+\varepsilon}}
		\frac{\log N\mathfrak p}{N\mathfrak p}
		+o_{K,\varepsilon,s_0}(\log Q)
		=
		(2\alpha+2\varepsilon+o_{K,\varepsilon,s_0}(1))\log Q .
	\end{aligned}
	\]

	If (i)  fails, then
	\[\sum_{ {
			Q^{\alpha+\varepsilon}<N\mathfrak p\le Q^{\beta_*}\atop
			\psi(\mathfrak p)=1}}
	\frac1{N\mathfrak p}<c_*.\]
	Hence
	\[
	\sum_{ {
			Q^{\alpha+\varepsilon}<N\mathfrak d\le Q^{\beta_*}\atop
			(\mathfrak d,\mathcal P(w))=1}}
	\Lambda_K(\mathfrak d)h(\mathfrak d)
	\le
	2\sum_{ {
			Q^{\alpha+\varepsilon}<N\mathfrak p\le Q^{\beta_*}\atop
			\psi(\mathfrak p)=1}}
	\frac{\log N\mathfrak p}{N\mathfrak p}
+o_{K,\varepsilon,s_0}(\log Q)
	\le
	(2c_*\beta_*+o_{K,\varepsilon,s_0}(1))\log Q.
	\]
	Summing over these two ranges together with \eqref{9.4}, we have
	\[
	\sum_{ {N\mathfrak d\le Q^{\beta_*}\atop
			(\mathfrak d,\mathcal P(w))=1}}
	\Lambda_K(\mathfrak d)
	\sum_{ {N\mathfrak n\le Y\atop
			\mathfrak d\mid\mathfrak n}} a_{\mathfrak n}
	 \le \bigl(1+O_K(e^{-s_0})  +o_{K,\varepsilon,s_0}(1) \bigr)(2\alpha+2\varepsilon +2c_*\beta_*)
	Y\log Q\,\rho_KL(1,\psi) V (w).
	\]

	Set
	$
	\Delta:=B_0-(2\alpha +2c_*\beta_*)>0.
    $
	Taking \(\varepsilon\) sufficiently small, we may assume
	$
	B_0-(2\alpha+2\varepsilon+2c_*\beta_*)>\Delta/2.
	$
	We then choose \(s_0\) sufficiently large so that the \(O_K(e^{-s_0})\)
	terms are small compared with \(\Delta\), and then take \(Q\) sufficiently
	large so that the \(o_{K,\varepsilon,s_0}(1)\) terms are also small so that from \eqref{9.3}, we obtain
	\[
	\sum_{Q^{\beta_*}<N\mathfrak d\le Y}
	\Lambda_K(\mathfrak d)
	\sum_{ {N\mathfrak n\le Y\atop
			\mathfrak d\mid\mathfrak n}}
	a_{\mathfrak n}
	\gg_{K,\varepsilon,s_0}
	Y\log Q\,\rho_KL(1,\psi)V(w).
	\]
Since \(w=Q^{\varepsilon/s_0}\) and \(s_0\) is now fixed,
Lemma~\ref{lem:h-euler-product} gives
$
V(w)\asymp_{K,\varepsilon }\mathcal V_\psi(Q),
$
hence we have 
\begin{equation}\label{9.5}
	\sum_{Q^{\beta_*}<N\mathfrak d\le Y}
	\Lambda_K(\mathfrak d)
	\sum_{ {N\mathfrak n\le Y\atop
			\mathfrak d\mid\mathfrak n}}
	a_{\mathfrak n}
	\gg_{K,\varepsilon}
	Y\log Q\,\rho_KL(1,\psi)\mathcal V_\psi(Q).
\end{equation}
Note that  since
\(0\le a_{\mathfrak n}\le \tau(\mathfrak n)\), for any ideal
\(\mathfrak b\), we have 
\begin{equation}\label{9.6a}
	\sum_{ {N\mathfrak n\le Y\atop \mathfrak b\mid\mathfrak n}}
	a_{\mathfrak n}
	\le
	\sum_{N\mathfrak m\le Y/N\mathfrak b}
	\tau(\mathfrak b\mathfrak m)
	\ll_K
	\tau(\mathfrak b)\frac{Y}{N\mathfrak b}\log Y .
\end{equation}
Therefore, the contribution to the sum in
\eqref{9.5} from prime powers \(\mathfrak d=\mathfrak p^m\) with \(m\ge2\), is
\[
\ll_K
Y \log Y 
\sum_{ {m\ge2\atop (N\mathfrak p)^m>Q^{\beta_*}}}
\frac{(m+1)\log N\mathfrak p}{(N\mathfrak p)^m}
\ll_K
Y(\log Q)^{2}Q^{-\beta_*/2},
\]
which is negligible compared with
\(Y\log Q\,\rho_KL(1,\psi)\mathcal V_\psi(Q)\), using \eqref{siegel} and \eqref{bound_V_psi}.

Now we consider \(\mathfrak d=\mathfrak p\). Since
\(\beta_*>B_0/2\), any prime ideal with \(N\mathfrak p>Q^{\beta_*}\) satisfies
\((N\mathfrak p)^2>Y\).
Thus, if \(N\mathfrak n\le Y\) and \(\mathfrak p\mid\mathfrak n\), then \(\mathfrak p\) divides \(\mathfrak n\) exactly once. Hence, writing \(\mathfrak n=\mathfrak p\mathfrak m\), then \((\mathfrak p,\mathfrak m)=1\) and  $ (1*\psi)(\mathfrak n) = (1+\psi(\mathfrak p))(1*\psi)(\mathfrak m). $ Therefore the inner sum vanishes whenever \(\psi(\mathfrak p)=-1\).
 
The remaining primes with \(\psi(\mathfrak p)=0\) divide \(\mathfrak q\). By \eqref{9.6a}, their total contribution is
\[ \ll_K Y\log Y \sum_{ {\mathfrak p\mid\mathfrak q\atop N\mathfrak p>Q^{\beta_*}}} \frac{\log N\mathfrak p}{N\mathfrak p} \le Y\log Y  Q^{-\beta_*}\log Q, \]
which is negligible compared with \(Y\log Q\,\rho_KL(1,\psi)\mathcal V_\psi(Q)\), again using  \eqref{siegel} and \eqref{bound_V_psi}. 
 
Therefore, in \eqref{9.5} we may restrict the outer sum to primes
\(\mathfrak d=\mathfrak p\) with \(\psi(\mathfrak p)=1\). Thus
\begin{equation}\label{9.6}
	Y\log Q\,\rho_KL(1,\psi)\mathcal V_\psi(Q)
	\ll_{K,\varepsilon}
	\sum_{ {
			Q^{\beta_*}<N\mathfrak p\le Y\atop
			\psi(\mathfrak p)=1}}
	\log N\mathfrak p
	\sum_{ {
			N\mathfrak n\le Y\atop
			\mathfrak p\mid\mathfrak n}}
	a_{\mathfrak n}.
\end{equation}

Moreover, as noted above, if \(N\mathfrak p>Q^{\beta_*}\), then
\(\mathfrak p\) divides \(\mathfrak n\) exactly once. Writing
\(\mathfrak n=\mathfrak p\mathfrak m\), we therefore have
\((\mathfrak p,\mathfrak m)=1\)  and $
a_{\mathfrak n}\le 2 \tau(\mathfrak m)
$. Since \(N\mathfrak p>w\), the condition
\((\mathfrak n,\mathcal P(w))=1\) is equivalent to
\((\mathfrak m,\mathcal P(w))=1\).

Hence we have 
\begin{align}
	\sum_{ {N\mathfrak n\le Y\atop \mathfrak p\mid\mathfrak n}}
	a_{\mathfrak n}\ll\sum_{{N\mathfrak m\le Y/N\mathfrak p\atop
			(\mathfrak m,\mathcal P(w))=1}}
	\tau(\mathfrak m)
	&\le
	\sum_{{N\mathfrak s\le Y/N\mathfrak p\atop
			(\mathfrak s,\mathcal P(w))=1}}
	\#\Bigl\{
	\mathfrak r:
	N\mathfrak r\le \frac{Y }{N\mathfrak pN\mathfrak s},\;
	(\mathfrak r,\mathcal P(w))=1
	\Bigr\} \nonumber \\
	&\ll_K
	\sum_{{N\mathfrak s\le Y/N\mathfrak p\atop
			(\mathfrak s,\mathcal P(w))=1}}
    \Bigl(
	\frac{Y}{N\mathfrak p N\mathfrak s\log w}+1
	\Bigr)\nonumber \\
	& \ll_{K, \varepsilon }
	\frac{Y}{N\mathfrak p\,\log w}+1. \label{9.7}
\end{align}
Here we use the standard upper-bound linear sieve estimate
\[
\#\{\mathfrak a:N\mathfrak a\le X,\;(\mathfrak a,\mathcal P(w))=1\}
\ll_K
\frac{X}{\log w}+1,
\]
which can be
obtained by applying Lemma~\ref{sieve_lemma} to the sequence of all
integral ideals \(N\mathfrak a\le X\) with density
\(g(\mathfrak p)=1/N\mathfrak p\), together with Mertens' theorem for
prime ideals. By partial summation this also gives
\[
\sum_{{N\mathfrak a\le X\atop
		(\mathfrak a,\mathcal P(w))=1}}
\frac1{N\mathfrak a}
\ll_K
1+\frac{\log X}{\log w}.
\]
	Applying \eqref{9.7} to \eqref{9.6}, we obtain
	\begin{equation}\label{9.8}
			Y\log Q\,\rho_KL(1,\psi)\mathcal V_\psi(Q)
	 \ll_{K,\varepsilon}
		\frac{Y}{\log w}\!\!\!
		\sum_{{
				Q^{\beta_*}<N\mathfrak p\le Y\atop
				\psi(\mathfrak p)=1}}\!\!\!
		\frac{\log N\mathfrak p}{N\mathfrak p} 
		+
		(\log Y)
		\#\bigl\{
		\mathfrak p:
		Q^{\beta_*}<N\mathfrak p\le Y,\ 
		\psi(\mathfrak p)=1
		\bigr\}.
	\end{equation}
 If the
	second term in  \eqref{9.8} dominates, then
	\[
	\#\bigl\{
	\mathfrak p:
	Q^{\beta_*}<N\mathfrak p\le Y,\ 
	\psi(\mathfrak p)=1
\bigr\}
	\gg_{K,\varepsilon}
	Y\,\rho_KL(1,\psi)\mathcal V_\psi(Q).
	\]
	A dyadic decomposition then gives some
	\(M\in[Q^{\beta_*},Y]\) such that
	\[
	\#\bigl\{
	\mathfrak p:
	M<N\mathfrak p\le2M,\ 
	\psi(\mathfrak p)=1
	\bigr\}
		\gg_{K,\varepsilon}
	M\,\rho_KL(1,\psi)\mathcal V_\psi(Q).
	\]	
	Otherwise the first term in \eqref{9.8} dominates. Since
	\(\log w\asymp_\varepsilon\log Q\), we get
\[
\sum_{{
		Q^{\beta_*}<N\mathfrak p\le Y\atop
		\psi(\mathfrak p)=1}}
\frac{\log N\mathfrak p}{N\mathfrak p}
\gg_{K,\varepsilon}
(\log Q)^2
\rho_KL(1,\psi)\mathcal V_\psi(Q).
\]
	Again by dyadic decomposition, for some \(M\in[Q^{\beta_*},Y]\),
	\[
	\sum_{{
			M<N\mathfrak p\le2M\atop
			\psi(\mathfrak p)=1}}
	\frac{\log N\mathfrak p}{N\mathfrak p}
	\gg_{K,\varepsilon}
	\log Q\,
	\rho_KL(1,\psi)\mathcal V_\psi(Q).
	\]
	Since \(\log N\mathfrak p\ll_{B_0}\log Q\) on this range, this implies
	\[
	\#\bigl\{
	\mathfrak p:
	M<N\mathfrak p\le2M,\ 
	\psi(\mathfrak p)=1
	\bigr\}
\gg_{K,\varepsilon}
	M\,\rho_KL(1,\psi)\mathcal V_\psi(Q).
	\]
	This proves  (ii).
\end{proof}
\begin{lemma}
	\label{lem:twisted-one-star-psi}
	Let \(\psi\) and \(\xi\) be non-principal Hecke characters modulo
	\(\mathfrak q\), with \(\psi\) quadratic. Suppose $\psi\neq \xi$. For every \(\varepsilon>0\), there exists
	\(\eta=\eta(\varepsilon)>0\) such that  whenever
$
		T\ge Q^{2\alpha+4\varepsilon},$
	uniformly for nonzero integral ideals $\mathfrak{d}$ with 
	$N\mathfrak d\le Q^\varepsilon,
	$
	we have
	\[
	\sum_{ {
			N\mathfrak a\le T\atop
			\mathfrak d\mid\mathfrak a}}
	\xi(\mathfrak a)(1*\psi)(\mathfrak a)
	\ll_{K,\varepsilon}
	\frac{T}{N\mathfrak d}\,Q^{-\eta}.
	\]
\end{lemma}
\begin{proof}
	We use
	\[
	\xi(\mathfrak a)(1*\psi)(\mathfrak a)
	=
	(\xi*\xi\psi)(\mathfrak a).
	\]
	Since \(\xi\notin\{\chi_0,\psi\}\), both \(\xi\) and \(\xi\psi\) are
	non-principal.
	
	First we record a hyperbola estimate. For \(U\ge Q^{2\alpha+2\varepsilon}\),
	put
	\[
	H(U):=
	\sum_{N\mathfrak b\,N\mathfrak c\le U}
	\xi(\mathfrak b)(\xi\psi)(\mathfrak c).
	\]
	By Dirichlet's hyperbola decomposition,
	\[
		H(U)=
		\sum_{N\mathfrak b\le U^{1/2}}
		\xi(\mathfrak b)
		\sum_{N\mathfrak c\le U/N\mathfrak b}
		(\xi\psi)(\mathfrak c)+
		\sum_{N\mathfrak c\le U^{1/2}}
		(\xi\psi)(\mathfrak c)
		\sum_{N\mathfrak b\le U/N\mathfrak c}
		\xi(\mathfrak b)-
		\Bigl(\sum_{N\mathfrak b\le U^{1/2}}\xi(\mathfrak b)\Bigr)
			\Bigl(\sum_{N\mathfrak c\le U^{1/2}}(\xi\psi)(\mathfrak c)\Bigr).
	\]
	By \(\eqref{CS}\), for some \(\theta=\theta(\varepsilon)>0\),    whenever \(X\ge Q^{\alpha+\varepsilon}\), we have
	\[
	\sum_{N\mathfrak a\le X}\xi(\mathfrak a)\ll_{K,\varepsilon}X^{1-\theta},
	\qquad
	\sum_{N\mathfrak a\le X}(\xi\psi)(\mathfrak a)
	\ll_{K,\varepsilon}X^{1-\theta}.
	\]
	 	Since every inner sum in $H(U)$ has length at least
	 $
	 U^{1/2}\ge Q^{\alpha+\varepsilon},
	 $
	  we have
	\[
	H(U)
\ll_{K,\varepsilon}
U^{1-\theta}
\sum_{N\mathfrak b\le U^{1/2}}\frac{1}{(N\mathfrak b)^{1-\theta}}
+
U^{1-\theta}
\sum_{N\mathfrak c\le U^{1/2}}\frac{1}{(N\mathfrak c)^{1-\theta}}
+
U^{1-\theta}   \ll_{K,\varepsilon}
U^{1-\theta/2}.
	\]
Since \(U\ge Q^{2\alpha+2\varepsilon}\), this gives
	\begin{equation}\label{hyper}
		H(U)\ll_{K,\varepsilon}UQ^{-\eta_0}
	\end{equation}
	for some \(\eta_0=\eta_0(\varepsilon)>0\).
	
	We now impose the divisibility condition. Expanding the convolution,
	\[
	\sum_{ {N\mathfrak a\le T\atop \mathfrak d\mid\mathfrak a}}
	\xi(\mathfrak a)(1*\psi)(\mathfrak a)
	=
	\sum_{ {N\mathfrak b\,N\mathfrak c\le T\atop
			\mathfrak d\mid\mathfrak b\mathfrak c}}
	\xi(\mathfrak b)(\xi\psi)(\mathfrak c).
	\]	
	For a given pair \((\mathfrak b,\mathfrak c)\), put
	$
	\mathfrak d_1=(\mathfrak d,\mathfrak b)$,
	 $
	\mathfrak d_2=\mathfrak d/\mathfrak d_1.
	$
	Then \(\mathfrak d_2\mid\mathfrak c\), and writing
	$
	\mathfrak b=\mathfrak d_1\mathfrak b_0$,
	$
	\mathfrak c=\mathfrak d_2\mathfrak c_0,
	$
	we have \((\mathfrak b_0,\mathfrak d_2)=1\). Conversely, these conditions
	recover the condition \((\mathfrak d,\mathfrak b)=\mathfrak d_1\) and
	\(\mathfrak d\mid\mathfrak b\mathfrak c\).
	Thus the sum can be written as 
	\[
	\sum_{\mathfrak d_1\mid\mathfrak d}
	\xi(\mathfrak d_1)(\xi\psi)(\mathfrak d_2)
	\sum_{ {
			N\mathfrak b_0N\mathfrak c_0\le T/N\mathfrak d\atop
			(\mathfrak b_0,\mathfrak d_2)=1}}
	\xi(\mathfrak b_0)(\xi\psi)(\mathfrak c_0).
	\]
	By Möbius inversion,
$
	1_{(\mathfrak b_0,\mathfrak d_2)=1}
	=
	\sum_{ \mathfrak r\mid(\mathfrak b_0,\mathfrak d_2)}
	\mu(\mathfrak r).
$
	Writing \(\mathfrak b_0=\mathfrak r\mathfrak b\), this becomes
	\[
	\sum_{\mathfrak d_1\mid\mathfrak d}
	\sum_{\mathfrak r\mid\mathfrak d_2}
	\mu(\mathfrak r)
	\xi(\mathfrak d_1\mathfrak r)(\xi\psi)(\mathfrak d_2)
	\sum_{N\mathfrak bN\mathfrak c\le T/(N\mathfrak d\,N\mathfrak r)}
	\xi(\mathfrak b)(\xi\psi)(\mathfrak c).
	\]
  For the inner sum, 
  since
	\(N\mathfrak d\le Q^\varepsilon\) and \(N\mathfrak r\le Q^\varepsilon\),
	we have
	$
	T/(N\mathfrak d\,N\mathfrak r)\ge Q^{2\alpha+4\varepsilon-\varepsilon-\varepsilon}
	=
	Q^{2\alpha+2\varepsilon}.
	$ Thus, the hyperbola estimate \eqref{hyper} applies to the inner sum and we have
	\[ 
	\sum_{ {N\mathfrak a\le T\atop \mathfrak d\mid\mathfrak a}}
	\xi(\mathfrak a)(1*\psi)(\mathfrak a)
	 \ll_{K,\varepsilon}
	Q^{-\eta_0}
	\sum_{\mathfrak d_1\mid\mathfrak d}
	\sum_{\mathfrak r\mid \mathfrak d/\mathfrak d_1}
	\frac{T}{N\mathfrak d\,N\mathfrak r}  \ll_{K,\varepsilon}
	\frac{T}{N\mathfrak d}
	Q^{-\eta_0}
	\tau(\mathfrak d)^2.
	\]
	Finally, since \(N\mathfrak d\le Q^\varepsilon\), the divisor bound for
	ideals gives \(\tau(\mathfrak d)^2\ll_{K,\varepsilon,\eta_0}Q^{\eta_0/2}\).
	Therefore,   we have 
	\[
	\sum_{ {
			N\mathfrak a\le T\atop
			\mathfrak d\mid\mathfrak a}}
	\xi(\mathfrak a)(1*\psi)(\mathfrak a)
	\ll_{K,\varepsilon}
	\frac{T}{N\mathfrak d}Q^{-\frac{\eta_0}{2}}.
	\]	
\end{proof}
\begin{lemma}[Sharp Halász--Montgomery with weight \(1*\psi\)]
	\label{lem:sharp-HM-psi}
  Let \(\varepsilon>0\) and
	\(C\ge1\) be fixed. Suppose
	\[
	Q^{2\alpha+4\varepsilon}\le X\le Q^C.
	\]
	Let \(\chi_1,\dots,\chi_R\) be distinct Hecke characters modulo
	\(\mathfrak q\), and let \(\psi\) be a non-principal quadratic
	Hecke character modulo \(\mathfrak q\). Then there exists
	\(\eta=\eta(\varepsilon)>0\) such that, for any complex coefficients
	\(a_{\mathfrak n}\),
	\[\sum_{j=1}^{R}\,
	\biggl|\! 
	\sum_{ {
			N\mathfrak n\le X\atop
			(\mathfrak n,\mathcal P(Q^\varepsilon))=1}}\!\!\!\!\!
	(1*\psi)(\mathfrak n)
	a_{\mathfrak n}
	\chi_j(\mathfrak n)
	\biggr|^2  
	 \ll_{K,\varepsilon,C}
	\bigl(
	X\rho_K L(1,\psi)\mathcal V_\psi(Q)
	+
	XQ^{-\eta}R
	\bigr)\!\!\!
	\sum_{ {
			N\mathfrak n\le X\atop
			(\mathfrak n,\mathcal P(Q^\varepsilon))=1}}\!\!\!\!\!
	(1*\psi)(\mathfrak n)|a_{\mathfrak n}|^2 .\]
\end{lemma}
\begin{proof}
	Define
	\[W(\mathfrak n):=(1*\psi)(\mathfrak n).\]
	Since \(\psi\) is quadratic, \(W(\mathfrak n)\ge0\).
	
		By the duality principle \cite[Sec.~7.1, p.~170]{book_analytic},  it suffices to prove that  for arbitrary complex numbers
	\(c_1,\dots,c_R\),
	\begin{equation}\label{9.9}
		\sum_{{
				N\mathfrak n\le X\atop
				(\mathfrak n,\mathcal P(Q^\varepsilon))=1}}
		W(\mathfrak n)
		\biggl|
		\sum_{j=1}^{R}c_j\chi_j(\mathfrak n)
		\biggr|^2   \ll_{K,\varepsilon,C}
		\bigl(
		X\rho_K L(1,\psi)\mathcal V_\psi(Q)
		+
		XQ^{-\eta}R
		\bigr)
		\sum_{j=1}^{R}|c_j|^2 .
	\end{equation}

	Fix \(s_0\ge 19\), and let 
	\[
	z:=Q^{\varepsilon/s_0},
	\qquad
	D:=z^{s_0}=Q^\varepsilon.
	\]
	Let \(\lambda_{\mathfrak d}^{+}\) be the upper-bound sieve weights with
	sifting range \(z\) and level \(D\). Applying the sieve inequality
	$
	1_{(\mathfrak n,\mathcal P(Q^\varepsilon))=1}
	\le
	1_{(\mathfrak n,\mathcal P(z))=1}
	\le
	\sum_{\mathfrak d\mid\mathfrak n}
	\lambda_{\mathfrak d}^{+},
	$
	the left-hand side of \eqref{9.9} is bounded by 
	\begin{equation}\label{contribution}
				\sum_{ 
					N\mathfrak n\le X } \Bigl(\sum_{\mathfrak d\mid\mathfrak n}
				\lambda_{\mathfrak d}^{+}\Bigr)
			W(\mathfrak n)
			\biggl|
			\sum_{j=1}^{R}c_j\chi_j(\mathfrak n)
			\biggr|^2  =
		\sum_{j,k=1}^{R}
		c_j\overline{c_k}
		\sum_{\mathfrak d\mid\mathcal P(z)}
		\lambda_{\mathfrak d}^{+}
		\sum_{ {
				N\mathfrak n\le X\atop
				\mathfrak d\mid\mathfrak n}}
		\chi_j(\mathfrak n)\overline{\chi_k(\mathfrak n)}
		W(\mathfrak n).
	\end{equation}
	
	Write
	$
	\xi_{j,k}:=\chi_j\overline{\chi_k}.
	$
	First suppose
	$\xi_{j,k}\notin\{\chi_0,\psi\}.$
	For  \(N\mathfrak d\le D=Q^\varepsilon\) and \(X\ge Q^{2\alpha+4\varepsilon}\), Lemma~\ref{lem:twisted-one-star-psi}
	gives
	\[
	\sum_{{
			N\mathfrak n\le X\atop
			\mathfrak d\mid\mathfrak n}}
	\xi_{j,k}(\mathfrak n)W(\mathfrak n)
	\ll_{K,\varepsilon}
	\frac{X}{N\mathfrak d}Q^{-\eta_0}
	\]
	for some \(\eta_0=\eta_0(\varepsilon)>0\).
	Using \(|\lambda_{\mathfrak d}^{+}|\le1\) and that
	\[
	\sum_{{\mathfrak d\mid\mathcal P(z)\atop N\mathfrak d\le D}}
	\frac1{N\mathfrak d}
	\le
	\prod_{N\mathfrak p<z}\Bigl(1+\frac1{N\mathfrak p}\Bigr)
	\ll_K \log z
	\ll_{K,\varepsilon }\log Q,
	\]
	we have, 	for some \(\eta_1=\eta_1(\varepsilon)>0\),
	\[
	\sum_{\mathfrak d\mid\mathcal P(z)}
	\lambda_{\mathfrak d}^{+}
	\sum_{ {
			N\mathfrak n\le X\atop
			\mathfrak d\mid\mathfrak n}}
	\xi_{j,k}(\mathfrak n)W(\mathfrak n)
	\ll_{K,\varepsilon}
	XQ^{-\eta_1}.
	\]
	Therefore the total contribution of these   pairs with $\xi_{j,k}\notin\{\chi_0,\psi\} $ to \eqref{contribution} is
	\begin{equation}\label{9.11}
			\ll_{K,\varepsilon}
		XQ^{-\eta_1}
		\sum_{j,k=1}^{R}|c_j||c_k|
		\ll
		XQ^{-\eta_1}R
		\sum_{j=1}^{R}|c_j|^2.
	\end{equation}

	It remains to treat the pairs for which
	$
	\xi_{j,k}\in\{\chi_0,\psi\}.
	$
	If \(\xi_{j,k}=\chi_0\), then \(j=k\). If \(\xi_{j,k}=\psi\), then
	\(\chi_j=\psi\chi_k\),  for each \(j\) there is at most one such
	\(k\), and these corresponding \(k\)'s are distinct as \(j\) varies.   Hence
	\[
	\sum_{ {j,k\atop \xi_{j,k}\in\{\chi_0,\psi\}}}
	|c_j||c_k|
	\le 2
	\sum_{j=1}^R |c_j|^2 .
	\]
 
Note that 
$
	\sum_{\mathfrak d\mid\mathfrak n}\lambda_{\mathfrak d}^{+} \ge0
	$ for any \(\mathfrak n\),  the exceptional contribution in \eqref{contribution} is bounded in absolute value by	
	\[
	\sum_{ {j,k\atop \xi_{j,k}\in\{\chi_0,\psi\}}}\!\!\!\!
	|c_j||c_k|\!\!
	\sum_{N\mathfrak n\le X}
	\biggl(\sum_{\mathfrak d\mid\mathfrak n}\lambda_{\mathfrak d}^{+}\biggr)
	W(\mathfrak n) \le
2	\sum_{j=1}^R |c_j|^2\!
	\sum_{N\mathfrak n\le X}
			\biggl(\sum_{\mathfrak d\mid\mathfrak n}\lambda_{\mathfrak d}^{+}	\biggr)W(\mathfrak n)
=2
	\sum_{j=1}^{R}|c_j|^2\!
	\sum_{\mathfrak d\mid\mathcal P(z)}
	\lambda_{\mathfrak d}^{+}
	\sum_{
			N\mathfrak n\le X\atop
			\mathfrak d\mid\mathfrak n}
	W(\mathfrak n).
	\]
	For \(N\mathfrak d\le D=Q^\varepsilon\), by Lemma~\ref{lem:divisible-one-star-psi}, there is some \(\eta_2=\eta_2(\varepsilon)>0\),
such that
	\[
	\sum_{{
			N\mathfrak n\le X\atop
			\mathfrak d\mid\mathfrak n}}
	W(\mathfrak n)
	=
X\rho_K L(1,\psi)h(\mathfrak d)	
	+
	O_{K,\varepsilon}\biggl(
	\tau_3(\mathfrak d) 
	\Bigl(\frac{X}{N\mathfrak d}\Bigr)^{1-\eta_2}
	\biggr).
	\]
  Summing over \(\mathfrak d\), the   error term is bounded by 
  \[
  \sum_{N\mathfrak d\le D}
  \tau_3(\mathfrak d)
  \Bigl(\frac{X}{N\mathfrak d}\Bigr)^{1-\eta_2}
  \le
  X^{1-\eta_2}
  \sum_{N\mathfrak d\le D}
  \frac{\tau_3(\mathfrak d)}{(N\mathfrak d)^{1-\eta_2}}  \ll_{K,\varepsilon}
  X^{1-\eta_2}D^{\eta_2}(\log D)^{O_K(1)}.
  \]
  Hence, for some  \(\eta_3=\eta_3(\eta)>0\), 
	\[
	\sum_{\mathfrak d\mid\mathcal P(z)}
	\lambda_{\mathfrak d}^{+}
	\sum_{{
			N\mathfrak n\le X\atop
			\mathfrak d\mid\mathfrak n}}
	W(\mathfrak n)
	\ll_{K,\varepsilon }
	X\rho_KL(1,\psi)
	\sum_{\mathfrak d\mid\mathcal P(z)}
	\lambda_{\mathfrak d}^{+}h(\mathfrak d)
	+
	XQ^{-\eta_3}.
	\]
	By   Lemma \ref{lem:fundamental-sieve-ideals} and   	Lemma~\ref{lem:h-euler-product}, 
	\[
	\sum_{\mathfrak d\mid\mathcal P(z)}
	\lambda_{\mathfrak d}^{+}h(\mathfrak d)
	\ll_K
	\prod_{N\mathfrak p<z}(1-h(\mathfrak p))\asymp_{K,\varepsilon}
	\mathcal V_\psi(Q).
	\]
	Thus the exceptional contribution from pairs with 	$
	\xi_{j,k}\in\{\chi_0,\psi\}.
	$ in \eqref{contribution}  is 
\begin{equation}\label{9.13}
		\ll_{K,\varepsilon }
	\bigl(
	X\rho_KL(1,\psi)\mathcal V_\psi(Q)
	+
	XQ^{-\eta_3}
	\bigr)
	\sum_{j=1}^{R}|c_j|^2.
\end{equation}
Combining \eqref{9.11}, \eqref{9.13} and taking
$
\eta=\min\{\eta_1,\eta_3\} 
$
gives
\eqref{9.9}.
\end{proof}

\begin{proposition}[Exceptional quadratic transference]
	\label{prop:quadratic-transference-final}
	Assume the setup of Proposition~\ref{prop:transference-conclusion}.
	Let \(A\subseteq G\) be the dense-model set constructed there.
	
	Let \(\psi\) be a non-principal quadratic Hecke character modulo
	\(\mathfrak q\). 
	Suppose that, for some \(c\in G\) and some   \(M\) with
	\[
	Q^{2\alpha+5\varepsilon}\le M \le X/2,
	\]
	\begin{equation}\label{9.14}
		\sum_{[\mathfrak p]a_1a_2=c,\  a_1,a_2\in A
			\atop
			M<N\mathfrak p\le2M,\  \psi(\mathfrak p)=1}
		1
		\gg
		\frac{|G|\,M\rho_KL(1,\psi)\mathcal V_\psi(Q)}
		{(\log Q)^{1/2-\varepsilon}}.
	\end{equation}
	Then
$
	c\in E_3(X;\mathfrak q).
$
\end{proposition}

\begin{proof}
	Let \(f\) and \(g\) be the functions appearing in the proof of
	Proposition~\ref{prop:transference-conclusion}. Thus \(f :\mathcal A(X;\mathfrak q)\to\mathbb R_{\ge0}\)	is given by  \[	f(\mathfrak a)
	:=
	\frac{\vartheta}{2}
	V_{\mathfrak q}\log X\,
	1_{\mathfrak a=\mathfrak p}
	1_{N\mathfrak p\ge z},\] and
	\(g:G\to[0,1+o(1)]\) is its dense model.
	
	Define \(f_0:\mathcal A(X;\mathfrak q)\to\mathbb R_{\ge0}\) by
	\[
	f_0(\mathfrak a)
	:=
	1_{\mathfrak a=\mathfrak p}
	1_{M<N\mathfrak p\le2M}
	1_{\psi(\mathfrak p)=1}.
	\]
	Note that  the function \(f_0\) is supported on prime ideals of
	norm at most \(X\) since \(2M\le X\).
	
	For \(c\in G\), define
	\[
	T_{f,f_0}(c)
	:=
	\frac1{|\mathcal A(X;\mathfrak q)|^2}\sum_{ {
			\mathfrak a_1,\mathfrak a_2,\mathfrak a_3
			\in\mathcal A(X;\mathfrak q)\atop
			[\mathfrak a_1\mathfrak a_2\mathfrak a_3]=c}}
	f(\mathfrak a_1)f(\mathfrak a_2)f_0(\mathfrak a_3),
	\]
	\[
	T_{g,f_0}(c)
	:=
  \frac1{|G|^2}	\sum_{ {
			a_1,a_2\in G,\ \mathfrak a_3\in\mathcal A(X;\mathfrak q)\atop
			a_1a_2[\mathfrak a_3]=c}}
	g(a_1)g(a_2)f_0(\mathfrak a_3).
	\]
	If \(	T_{f,f_0}(c)>0\), then by definition,  \(c\in E_3(X;\mathfrak q)\).

	For a character \(\chi\) of \(G\), define
	\[
	\widehat f(\chi)
	:=
	\frac1{|\mathcal A(X;\mathfrak q)|}
	\sum_{\mathfrak a\in\mathcal A(X;\mathfrak q)}
	f(\mathfrak a)\chi(\mathfrak a),
	\qquad
	\widehat g(\chi)
	:=
	\frac1{|G|}
	\sum_{a\in G}g(a)\chi(a),
	\]
	and
	\[
	\widehat f_0(\chi)
	:=
	\sum_{\mathfrak a\in\mathcal A(X;\mathfrak q)}
	f_0(\mathfrak a)\chi(\mathfrak a)
	=
	\sum_{ {
			M<N\mathfrak p\le2M\atop
			\psi(\mathfrak p)=1}}
	\chi(\mathfrak p).
	\]
	Here we do not include the normalization factor
	\(
	|\mathcal A(X;\mathfrak q)|^{-1}
	\)
	in the definition of \(\widehat f_0(\chi)\). This avoids introducing a
	different normalization scale and leads to cleaner definition for
	\(T_{g,f_0}\).

	By orthogonality of characters,
	\[
	 T_{f,f_0}(c)
	-
 T_{g,f_0}(c)
	=
	\frac1{|G|}
	\sum_{\chi\in\widehat G}
	\left(\widehat f(\chi)^2-\widehat g(\chi)^2\right)
	\widehat f_0(\chi)\overline{\chi(c)}.
	\]
	Therefore
	\begin{equation}\label{9.15}
			\left|
		T_{f,f_0}(c)
		-
		T_{g,f_0}(c)
		\right|
		\le
		\frac1{|G|}
		\sum_{\chi\in\widehat G}
		\bigl(|\widehat f(\chi)|+|\widehat g(\chi)|\bigr)
		|\widehat f_0(\chi)|
		|\widehat f(\chi)-\widehat g(\chi)|.
	\end{equation}
	Define
	\[
	A_\psi(M):=M\rho_KL(1,\psi)\mathcal V_\psi(Q),
	\qquad
	\mathcal L:=n_K+10.
	\]
		We split the characters according to the size of \(\widehat f_0(\chi)\).
	Define
	\[
	\mathcal X_1
	:=
	\Bigl\{
	\chi\in\widehat G:
	|\widehat f_0(\chi)|
	<
	(\log Q)^{-\mathcal L}A_\psi(M)
	\Bigr\},
	\qquad
	\mathcal X_2:=\widehat G\setminus\mathcal X_1 .
	\]
	
	For \(\chi\in\mathcal X_1\), using
	$
	|\widehat g(\chi)|\le |\widehat f(\chi)|$, 
 $
	|\widehat f(\chi)-\widehat g(\chi)|\le |\widehat f(\chi)|,
	$
	and the mean-square estimate
$
	\sum_{\chi\in\widehat G}|\widehat f(\chi)|^2
	\ll_{K }(\log Q)^{n_K+1} 
$   from \eqref{f_l2_mean_bound},
	we have
	\begin{equation}\label{9.16}
			\frac1{|G|}
			\sum_{\chi\in\mathcal X_1}
			\bigl(|\widehat f(\chi)|+|\widehat g(\chi)|\bigr)
			|\widehat f_0(\chi)|
			|\widehat f(\chi)-\widehat g(\chi)| 
			\ll_{K}
			\frac{A_\psi(M)}{|G|(\log Q)^{\mathcal L}}
			\sum_{\chi\in\widehat G}|\widehat f(\chi)|^2
			\ll_{K}
			\frac{A_\psi(M)}{|G|(\log Q)^9}.
	\end{equation}

	For \(\chi\in\mathcal X_2\), we will apply
Lemma~\ref{lem:sharp-HM-psi} at length \(2M\) with coefficients
\[
a_{\mathfrak n}
=
\frac12\,
1_{\mathfrak n=\mathfrak p}
1_{M<N\mathfrak p\le2M}
1_{\psi(\mathfrak p)=1}.
\]
Since \((1*\psi)(\mathfrak p)=2\) when $\psi(\mathfrak p)=1$ and $M>Q^{\varepsilon}$, we have
\[
\sum_{ {N\mathfrak n\le2M\atop
		(\mathfrak n,\mathcal P(Q^\varepsilon))=1}}
(1*\psi)(\mathfrak n)a_{\mathfrak n}\chi(\mathfrak n)
=
\widehat f_0(\chi).
\]
Moreover, by Lemma~\ref{lem:upper-bound-psi-one-primes},
\[
\sum_{ {N\mathfrak n\le2M\atop
		(\mathfrak n,\mathcal P(Q^\varepsilon))=1}}
(1*\psi)(\mathfrak n)|a_{\mathfrak n}|^2
=
\frac12
\#\{\mathfrak p:M<N\mathfrak p\le2M,\ \psi(\mathfrak p)=1\}\ll_{K,\varepsilon}
A_\psi(M).
\]
Therefore Lemma~\ref{lem:sharp-HM-psi} gives, for some \(\eta_0=\eta_0(\varepsilon)>0\),
\begin{equation}\label{9.17a}
	\sum_{\chi\in\mathcal X_2}|\widehat f_0(\chi)|^2
	\ll_{K,\varepsilon }
	\left(A_\psi(M)+MQ^{-\eta_0}|\mathcal X_2|\right)A_\psi(M).
\end{equation}

We now bound \(|\mathcal X_2|\). By the definition of \(\mathcal X_2\), 
$
|\widehat f_0(\chi)|
\ge
(\log Q)^{-\mathcal L}A_\psi(M)$ for any $\chi\in\mathcal X_2$. 
  Using this lower bound in \eqref{9.17a}, we have
\[
|\mathcal X_2|(\log Q)^{-2\mathcal L}A_\psi(M)^2
\ll_{K,\varepsilon}
A_\psi(M)^2
+
MQ^{-\eta_0}|\mathcal X_2|A_\psi(M).
\]
The second term on the right is absorbed into the left. Indeed, by the
Siegel lower bound   \eqref{siegel} and Lemma~\ref{lem:h-euler-product}, for any \(\delta>0\), 
\begin{equation}\label{bound}
	\frac{A_\psi(M)}{M}
	=
	\rho_KL(1,\psi)\mathcal V_\psi(Q)
	\gg_{K,\delta}Q^{-2\delta}.
\end{equation}
 Choosing \(\delta>0\) sufficiently small, we have
$
MQ^{-\eta_0}A_\psi(M)
\le
\frac12
(\log Q)^{-2\mathcal L}A_\psi(M)^2
$
for \(Q\) sufficiently large. Therefore
\begin{equation}\label{9.17c}
	|\mathcal X_2|\ll_{K,\varepsilon }(\log Q)^{2\mathcal L}.
\end{equation}
Applying \eqref{bound} and \eqref{9.17c} to \eqref{9.17a},   we have
\begin{equation}\label{9.17}
	\sum_{\chi\in\mathcal X_2}|\widehat f_0(\chi)|^2
	\ll_{K,\varepsilon }
	A_\psi(M)^2.
\end{equation}

We now prove the mean-square bound for \(f\) restricted to \(\mathcal X_2\). This follows similarly as in the proof of
Proposition~\ref{prop:transference-conclusion}.
By Lemma~\ref{lem:ray-HM}(i),   for some \(\eta_1=\eta_1(\varepsilon)>0\),
\[
\sum_{\chi\in\mathcal X_2}|\widehat f(\chi)|^2
\ll_{K,\varepsilon}
\Bigl(
\frac{X}{\log Q}
+
|\mathcal X_2|X^{1-\eta_1}Q^{\varepsilon\eta_1}
\Bigr)
\frac1{|\mathcal A(X;\mathfrak q)|^2}
\sum_{\mathfrak a\in\mathcal A(X;\mathfrak q)}f(\mathfrak a)^2.
\]
 Using \eqref{5.14}, this gives
\[
\sum_{\chi\in\mathcal X_2}|\widehat f(\chi)|^2
\ll_{K,\varepsilon}
\frac{\log X}{\log Q}
+
|\mathcal X_2|X^{-\eta_1}Q^{\varepsilon\eta_1}\log X.
\]
Since \(Q^{2\alpha+5\varepsilon}\le X\le Q^C\),  
\(|\mathcal X_2|\ll_{K,\varepsilon}(\log Q)^{2\mathcal L}\), taking $\varepsilon$ sufficiently small and fixed, the second term is \(o_{K,\varepsilon}(1)\),
and \(\log X/\log Q\le C\). Hence
\[
\sum_{\chi\in\mathcal X_2}|\widehat f(\chi)|^2
\ll_{K,\varepsilon}1.
\]
Moreover,  \(|\widehat g(\chi)|\le |\widehat f(\chi)|\), this implies
\begin{equation}\label{9.18}
	\sum_{\chi\in\mathcal X_2}
	\bigl(|\widehat f(\chi)|+|\widehat g(\chi)|\bigr)^2
	\ll_{K,\varepsilon}1.
\end{equation}

Finally, Proposition~\ref{prop:transference-conclusion} gives the
dense-model estimate
\[
|\widehat f(\chi)-\widehat g(\chi)|
\le
\delta,
\qquad
\delta=(\log Q)^{-1/2+\varepsilon/2}.
\]
Combining this with \eqref{9.17}, \eqref{9.18}, and the
Cauchy--Schwarz inequality, we obtain
\begin{equation}\label{9.19}
	\begin{aligned}
		&\frac1{|G|}
		\sum_{\chi\in\mathcal X_2}
		\bigl(|\widehat f(\chi)|+|\widehat g(\chi)|\bigr)
		|\widehat f_0(\chi)|
		|\widehat f(\chi)-\widehat g(\chi)| \\
		&\qquad\le
		\frac{\delta}{|G|}
		\Bigl(
		\sum_{\chi\in\mathcal X_2}
		\bigl(|\widehat f(\chi)|+|\widehat g(\chi)|\bigr)^2
		\Bigr)^{1/2}
		\Bigl(
		\sum_{\chi\in\mathcal X_2}
		|\widehat f_0(\chi)|^2
		\Bigr)^{1/2}  \ll_{K,\varepsilon }
		\frac{\delta}{|G|} A_\psi(M)  .
	\end{aligned}
\end{equation}

	Combining \eqref{9.15}, \eqref{9.16}, and \eqref{9.19}, we obtain
	\begin{equation}\label{9.20}
		T_{f,f_0}(c)
	=
	T_{g,f_0}(c)
	+
	O _{K,\varepsilon }\Bigl(
	\frac{\delta}{|G|}
	A_\psi(M)
	\Bigr).
	\end{equation}
	By the definition of $A$ in \eqref{def_A}, \(g(a)\ge\varepsilon/10\) for \(a\in A\). Thus, by hypothesis \eqref{9.14},  
	\[
	T_{g,f_0}(c)
	\gg_{\varepsilon} 	\frac1{|G|^2} 
	\sum_{[\mathfrak p]a_1a_2=c,\ a_1,a_2\in A
		\atop
		M<N\mathfrak p\le2M,\  \psi(\mathfrak p)=1}
	1 
	\gg
	\frac{A_\psi(M)}{|G|(\log Q)^{1/2-\varepsilon}}
	 .
	\]
	This dominates the error term in \eqref{9.20} as 
	$
	\delta=(\log Q)^{-1/2+\varepsilon/2}.
	$
	Hence \(T_{f,f_0}(c)>0\), and
	$
	c\in E_3(X;\mathfrak q).
	$
\end{proof}
\begin{corollary}
	\label{cor:quadratic-transference-AprimeBprime}
	Assume the setup of Proposition~\ref{prop:criteria-E2-E3}(b), and let
	\(A',B'\subseteq A\) be the sets there.
		Let \(\psi\) be a non-principal quadratic Hecke character modulo
	\(\mathfrak q\). 
	Suppose that, for some \(c\in G\) and some   \(M\) with
	\[
	Q^{2\alpha+5\varepsilon}\le M \le X/2,
	\]
	\begin{equation}\label{9.21}
			\sum_{[\mathfrak p]d=c,\ d\in A'\cdot B'
				\atop
				M<N\mathfrak p\le2M,\ \psi(\mathfrak p)=1}
		1
		\gg 
		M\rho_K L(1,\psi)\mathcal V_\psi(Q).
	\end{equation}
	Then
	$
	c\in E_3(X;\mathfrak q).
	$
\end{corollary}

\begin{proof}
	By \eqref{6.4} in the proof of
	Proposition~\ref{prop:criteria-E2-E3}(b.ii),  we have that for every
	\(d\in A'\cdot B'\),
	\[
	(1_A*1_A)(d) = \#\{(a_1,a_2)\in A^2:a_1a_2=d\}\gg |G|.
	\] 
Consequently,
	\[
		\sum_{[\mathfrak p]a_1a_2=c,\  a_1,a_2\in A
		\atop
		M<N\mathfrak p\le2M,\  \psi(\mathfrak p)=1}
		1  \ge
		\sum_{[\mathfrak p]d=c,\ d\in A'\cdot B'
			\atop
			M<N\mathfrak p\le2M,\ \psi(\mathfrak p)=1}
		(1_A*1_A)(d) \gg
		|G|
		\sum_{[\mathfrak p]d=c,\ d\in A'\cdot B'
			\atop
			M<N\mathfrak p\le2M,\ \psi(\mathfrak p)=1}
		1.
	\]
	Hence \eqref{9.21} implies condition \eqref{9.14} of
	Proposition~\ref{prop:quadratic-transference-final}. Applying Proposition~\ref{prop:quadratic-transference-final}, we conclude that
	$
	c\in E_3(X;\mathfrak q).
	$
\end{proof}

\section{Products of two prime ideals}
In this section, we will prove Theorem \ref{intro:two-prime-ray}. Before the proof we need some lemmas. As before, we write
$
G=\operatorname{Cl}_{\mathfrak q}^{(\infty)}$,
$
Q=N\mathfrak q, 
$
and assume that \eqref{CS} and \eqref{Lb} hold for fixed parameters
\(0<\alpha_0\le \alpha<1\). Note that \eqref{Lb} implies \eqref{CS^b}.

\begin{lemma} 
	\label{lem:ray-prime-energy}
Let \(\kappa>0\) and \(B\ge \max(1,3\alpha)+\kappa\) be fixed, and let
\(\varepsilon>0\) be sufficiently small. Set \(X:=Q^B\),  \(	z:=X^{2/3}\), and let \(\mathcal P_1,\mathcal P_2\) be finite sets
	of prime ideals such that for some \(Q^{2\varepsilon}\leq Y_1\le X/2\),
	\[
	\mathcal P_1
	\subseteq
	\{\mathfrak p\nmid\mathfrak q:Y_1<N\mathfrak p\le 2Y_1\}, 
	\qquad
	\mathcal P_2
	\subseteq
	\{\mathfrak p\nmid\mathfrak q:	z<N\mathfrak p\le X\}.
	\]
	Suppose  that, for some fixed constant $ C_0\ge 1$ and for any $\nu>0$, 
	\begin{equation}\label{eq:P1-lower-for-energy}
|\mathcal P_1|\gg_\nu Y_1Q^{-\nu},\qquad \frac{2X}{\log 	z}
		\le
		(C_0+o(1))|\mathcal P_2|.
	\end{equation}
	Define 	the multiplicative energy 
	\[
	\mathcal E_G(\mathcal P_1,\mathcal P_2)
	:=
	\#\bigl\{(\mathfrak p_1,\mathfrak p_1',\mathfrak p_2,\mathfrak p_2') \in \mathcal{P}_1^2\times \mathcal{P}_2^2:
	[\mathfrak p_1\mathfrak p_2]=
	[\mathfrak p_1'\mathfrak p_2']\bigr\}.
	\]
	Then
	\[
	\mathcal E_G(\mathcal P_1,\mathcal P_2)
	\le
	\frac{C_0+o_{K,\varepsilon}(1)}{|G|}
	|\mathcal P_1|^2|\mathcal P_2|^2.
	\]
\end{lemma}
\begin{proof}
	
	Let \(\lambda_{\mathfrak d}^{+}\) be the upper-bound linear sieve weights as in Lemma \ref{sieve_lemma}
	with sifting range \(z\) and level \(D=z\)  so that \(s=1\). Since all
	prime ideals in \(\mathcal P_2\) have norm \(>z\), using the upper bound sieve inequality
	$
	1_{\mathfrak n=\mathfrak p\in\mathcal P_2}
	\le
	1_{(\mathfrak n,\mathcal P(z))=1}
	\le
	\sum_{\mathfrak d\mid\mathfrak n}\lambda_{\mathfrak d}^{+},
	$
	we have
	\[
	\mathcal E_G(\mathcal P_1,\mathcal P_2)
 \le
\sum_{\mathfrak p_1,\mathfrak p_1'\in\mathcal P_1}
\sum_{\mathfrak p_2\in\mathcal P_2}
\sum_{ {N\mathfrak n\le X\atop (\mathfrak n,\mathfrak q)=1}}
\Bigl(\sum_{\mathfrak d\mid\mathfrak n}\lambda_{\mathfrak d}^{+}\Bigr)
1_{[\mathfrak p_1\mathfrak p_2]=[\mathfrak p_1'\mathfrak n]}.
	\]
	By orthogonality of characters of \(G\), this is
	\begin{equation}\label{E1}
			\mathcal E_G(\mathcal P_1,\mathcal P_2)
		 \le
		\frac1{|G|}
		\sum_{\chi\in\widehat G}
		\biggl|
		\sum_{\mathfrak p\in\mathcal P_1}\chi(\mathfrak p)
		\biggr|^2
		\Bigl(
		\sum_{\mathfrak p\in\mathcal P_2}\chi(\mathfrak p)
		\Bigr)
		\biggl(
		\sum_{ {N\mathfrak n\le X\atop(\mathfrak n,\mathfrak q)=1}}
		\overline{\chi}(\mathfrak n)
		\sum_{\mathfrak d\mid\mathfrak n}\lambda_{\mathfrak d}^{+}
		\biggr).
	\end{equation}
	The contribution of the principal character \(\chi=\chi_0\) in \eqref{E1} is
	\[
	\frac1{|G|}
	|\mathcal P_1|^2|\mathcal P_2|
	\sum_{{N\mathfrak n\le X\atop(\mathfrak n,\mathfrak q)=1}}
	\sum_{\mathfrak d\mid\mathfrak n}\lambda_{\mathfrak d}^{+}.
	\]
Interchanging the order of summation and applying the ideal-counting
estimate in Lemma \ref{lem:ideal-counting-prime-to-q}, we get
	\[
	\begin{aligned}
		\sum_{ {N\mathfrak n\le X\atop(\mathfrak n,\mathfrak q)=1}}
	\sum_{\mathfrak d\mid\mathfrak n}\lambda_{\mathfrak d}^{+}
	&=
	\sum_{ (\mathfrak d,\mathfrak q)=1}
	\lambda_{\mathfrak d}^{+}
	\#\Bigl\{
	\mathfrak m:
	N\mathfrak m\le \frac{X}{N\mathfrak d},\
	(\mathfrak m,\mathfrak q)=1
	\Bigr\}  \\
		&=
		V_{\mathfrak q}X
		\sum_{(\mathfrak d,\mathfrak q)=1}
		\frac{\lambda_{\mathfrak d}^{+}}{N\mathfrak d}
		+	
		O_{K,\delta}\biggl(\sum _{N\mathfrak{d}\le z}\Bigl(\frac{X}{N \mathfrak{d}}\Bigr)^{1-1/n_K} Q^{\delta}\biggr)
	\end{aligned}
	\]
	where
$
	V_{\mathfrak q}:=
	\rho_K\prod_{\mathfrak p\mid\mathfrak q}
	\bigl(1-\frac1{N\mathfrak p}\bigr),
	$
	Since
	$
	\sum_{N\mathfrak d\le z}
	(N\mathfrak d)^{-1+1/n_K}
	\ll_K z^{1/n_K},
	$
	the error term is  \( \ll_{K,\delta} Q^\delta X^{1-1/n_K} z^{1/n_K}=o_{K}(X/\log z) \) provided \(\delta>0\) is chosen sufficiently small.
	
By the upper-bound linear sieve, Lemma~\ref{sieve_lemma} with \(s=1\),
together with the prime-ideal Mertens theorem \eqref{merten}, we have
	\[
	\begin{aligned}
		V_{\mathfrak q}
		\sum_{(\mathfrak d,\mathfrak q)=1}
		\frac{\lambda_{\mathfrak d}^{+}}{N\mathfrak d}
		&\le
		(2e^\gamma+o_K(1))
		\rho_K
		\prod_{N\mathfrak p<z}
		\Bigl(1-\frac1{N\mathfrak p}\Bigr) = 	\frac{2+o_K(1)}{\log z}.
	\end{aligned}
	\]	
	Hence, using the lower bound for \(|\mathcal P_2|\) in
	\eqref{eq:P1-lower-for-energy}, the  contribution of the principal character is
	\begin{equation}\label{E2}
			\le 
		\frac{(2+o_K(1))}{|G|}
		|\mathcal P_1|^2|\mathcal P_2|
		\frac{X}{\log z}	\le
		\frac{C_0+o_{K }(1)}{|G|}
		|\mathcal P_1|^2|\mathcal P_2|^2.
	\end{equation}

	Write
	\[
	S_1(\chi):=\sum_{\mathfrak p\in\mathcal P_1}\chi(\mathfrak p),
	\qquad
	S_2(\chi):=\sum_{\mathfrak p\in\mathcal P_2}\chi(\mathfrak p),
	\qquad 	R(\chi):=
	\sum_{ {N\mathfrak n\le X\atop(\mathfrak n,\mathfrak q)=1}}
	\chi(\mathfrak n)
	\sum_{\mathfrak d\mid\mathfrak n}\lambda_{\mathfrak d}^{+}.
	\]
 We claim that the non-principal-characters contribution in \eqref{E1} is
	\begin{equation}\label{eq:energy-nonprincipal-ray}
			\frac1{|G|}
		\sum_{\chi\ne\chi_0}
		|S_1(\chi)|^2|S_2(\chi)||R(\overline\chi)|
		\ll_{K,\varepsilon}
		\frac{|\mathcal P_1|^2|\mathcal P_2|^2}
		{|G| \log  Q  }.
	\end{equation}
	Assuming the claim, combining it with \eqref{E2} proves the
	lemma.
	
	We now prove  \eqref{eq:energy-nonprincipal-ray}.
	
	Split the non-principal characters into
	\[
	\mathcal X_1
	:=
	\left\{
	\chi\ne\chi_0:
	|S_1(\chi)|
	\le
	\frac{|\mathcal P_1|}{(\log  Q)^{A_1}}
	\right\},
	\]
	and
	\[
	\mathcal X_2
	:=
	\left\{
	\chi\ne\chi_0:
	|S_1(\chi)|
	>
	\frac{|\mathcal P_1|}{(\log  Q)^{A_1 }}
	\right\},
	\]
	where $A_1>0$ is a constant to be chosen later. 
	
	For \(\mathcal X_1\), Cauchy--Schwarz gives
	\[
	\frac1{|G|}
	\sum_{\chi\in\mathcal X_1}
	|S_1(\chi)|^2|S_2(\chi)||R(\overline\chi)| 
	\le
	\frac{|\mathcal P_1|^2}{|G|(\log  Q)^{2A_1}}
	\biggl(\sum_{\chi\in\widehat G}|S_2(\chi)|^2\biggr)^{1/2}
		\biggl(\sum_{\chi\in\widehat G}|R(\chi)|^2\biggr)^{1/2}.
	\]	
	Applying Lemma~\ref{lem:ray-class-mean-value-sharp}, with   the coefficients
	supported on \(\mathcal P_2\) and using that that  $X\geq Q^{1+\kappa}$, we have
\begin{equation}\label{E3a}
			\sum_{\chi\in\widehat G}|S_2(\chi)|^2
		\ll_K
		(X+Q)(\log(3Q))^{n_K}|\mathcal P_2|
		\ll_K
		X(\log Q)^{n_K}|\mathcal P_2|.
\end{equation}	
	Next write
	\[
	R(\chi)
	=
	\sum_{ {N\mathfrak n\le X\atop(\mathfrak n,\mathfrak q)=1}}
	b_{\mathfrak n}\chi(\mathfrak n),
	\qquad
	b_{\mathfrak n}:=
	\sum_{\mathfrak d\mid\mathfrak n}\lambda_{\mathfrak d}^{+}.
	\]
	Since \(|\lambda_{\mathfrak d}^{+}|\le1\), we have
	$
	|b_{\mathfrak n}|\le \tau(\mathfrak n).
	$
	Moreover, using \(\tau(\mathfrak n)^2\le \tau_4(\mathfrak n)\), we have
	\[
	\sum_{ {N\mathfrak n\le X\atop(\mathfrak n,\mathfrak q)=1}}
	|b_{\mathfrak n}|^2
	\le
	\sum_{N\mathfrak n\le X}\tau(\mathfrak n)^2
	\le
	\sum_{N\mathfrak n\le X}\tau_4(\mathfrak n)
	\ll_K
	X(\log 2X)^3.
	\]
	Applying Lemma~\ref{lem:ray-class-mean-value-sharp} again gives
\begin{equation}\label{E3b}
	\sum_{\chi\in\widehat G}|R(\chi)|^2
	\ll_K
	(X+Q)(\log(3Q))^{n_K}X(\log 2X)^3
	\ll_K
	X^2(\log Q)^{n_K+3}.
\end{equation}
 From \eqref{eq:P1-lower-for-energy}, we have
	$
 \frac{X}{\log z}\ll 	|\mathcal P_2|.
$
Combining \eqref{E3a} and \eqref{E3b}, and using  that  \(\log z\asymp\log Q\), we obtain 
	\begin{equation}\label{E3}
			\frac1{|G|}
		\sum_{\chi\in\mathcal X_1}
		|S_1(\chi)|^2|S_2(\chi)||R(\overline\chi)|
		\ll_K
		\frac{|\mathcal P_1|^2|\mathcal P_2|^2}
		{|G|(\log  Q)^{2A_1-n_K-3}}.
	\end{equation}
	Choose \(A_1\) so that \(2A_1-n_K-3\ge 1 \); for example,
	\(A_1=n_K+2\) suffices. Then \eqref{E3} is
	$
	O_{K }\Bigl(
		\frac{|\mathcal P_1|^2|\mathcal P_2|^2}
	{|G|(\log  Q)^{ n_K+1}}
	\Bigr).
	$

	We next consider \(\mathcal X_2\). First   we record a pointwise bound for \(R(\chi)\) for $\chi$ non-principal. Write
	\[
	R(\chi)
	=
	\sum_{{\mathfrak d\mid\mathcal P(z),\ 
			N\mathfrak d\le D\atop
			(\mathfrak d,\mathfrak q)=1}}
	\lambda_{\mathfrak d}^{+}\chi(\mathfrak d)
	\sum_{ {N\mathfrak m\le X/N\mathfrak d\atop
			(\mathfrak m,\mathfrak q)=1}}
	\chi(\mathfrak m).
	\]
For \(N\mathfrak d\le D=z=X^{2/3}\), taking \(\varepsilon<\kappa/3\), we have
	$
	\frac{X}{N\mathfrak d}\ge X^{1/3}\ge Q^{ \alpha+\kappa /3}\ge Q^{\alpha+\varepsilon}.
	$
	Applying \eqref{CS} to the inner sum, we have for some \(\eta_0=\eta_0(\varepsilon)>0\),
	\[
	\sum_{ {N\mathfrak m\le X/N\mathfrak d\atop
			(\mathfrak m,\mathfrak q)=1}}
	\chi(\mathfrak m)
	\ll_{K ,\varepsilon}
	\Bigl(\frac{X}{N\mathfrak d}\Bigr)^{1-\eta_0}.
	\]
	Therefore, 
	\begin{equation}\label{E4R}
			R(\chi)
		\ll_{K ,\varepsilon}
		X^{1-\eta_0}
		\sum_{N\mathfrak d\le D}
		\frac1{(N\mathfrak d)^{1-\eta_0}} 
		\ll_{K}
		X^{1-\eta_0}D^{\eta_0}=XQ^{-\eta_1},\qquad\eta_1: =\frac{B}{3}\eta_0.
	\end{equation}

	We now bound \(|\mathcal X_2|\) using higher-moment estimates, similar to that in the proof of Proposition~\ref{prop:transference-conclusion}(iv). Define
	\[
	\beta:=\frac{\log Y_1}{\log Q},
	\qquad
	L:=\Bigl\lfloor \frac{B}{\beta}\Bigr\rfloor .
	\]
	Then $\beta\ge 2\varepsilon$, by the same elementary  proof for \eqref{claim}, we have \(L\beta>B/2
\ge \alpha+\kappa/2\).	Taking $\varepsilon$ sufficiently small, then $Y_1^L= Q^{L\beta}\ge Q^{\alpha+2\varepsilon}.$ 
	Expanding the \(L\)-th power, we have
	\[
	(S_1(\chi))^L
	=
	\biggl(\sum_{\mathfrak p\in\mathcal P_1}\chi(\mathfrak p)\biggr)^L
	=
	\sum_{\mathfrak n} c_{\mathfrak n}\chi(\mathfrak n).
	\]
	The coefficients \(c_{\mathfrak n}\) are supported on products of \(L\)
	prime ideals from \(\mathcal P_1\). Hence
	$
	Y_1^L<N\mathfrak n\le (2Y_1)^L 
	$
 and 
	$
	\sum_{\mathfrak n}|c_{\mathfrak n}|^2
	\ll_L
	|\mathcal P_1|^L.
	$
	Hence,
$
	\sum_{\mathfrak n} N\mathfrak n\,|c_{\mathfrak n}|^2
	\ll_L
	Y_1^L|\mathcal P_1|^L.
	$
	
	Moreover, since every \(\mathfrak p\in\mathcal P_1\) has
	\(N\mathfrak p>Y_1\ge Q^{2\varepsilon}\), all ideals \(\mathfrak n\) in the
	support of \(c_{\mathfrak n}\) are coprime to \(\mathcal P(Q^\varepsilon)\).
 Applying Lemma~\ref{lem:ray-HM}(ii) to the coefficients
	$
	a_{\mathfrak n}:=N\mathfrak n\,c_{\mathfrak n}
	$
	in the interval \(Q^{\alpha+2\varepsilon}\le Y_1^L<N\mathfrak n\le(2Y_1)^L\), there is some   $\eta_2=\eta_2(\varepsilon)>0$  such that
	\begin{align}\label{E5a}
		\sum_{\chi\in\mathcal X_2}|S_1(\chi)|^{2L}
		&\ll_{K,\varepsilon}
		\biggl(
		1+
		|\mathcal X_2|
		\Bigl(\frac{Q^\varepsilon}{Y_1^L}\Bigr)^{\eta_2}
		\biggr)
		\sum_{\mathfrak n}N\mathfrak n\,|c_{\mathfrak n}|^2 \nonumber  \\
		&\ll_{K,\varepsilon}
		\bigl(
		1+
		|\mathcal X_2|Y_1^{-L\eta_2}Q^{\varepsilon\eta_2}
		\bigr)
		Y_1^L|\mathcal P_1|^L.
	\end{align}
	On the other hand, by the definition of \(\mathcal X_2\),
	\[
	|\mathcal X_2|
	\biggl(\frac{|\mathcal P_1|}{(\log  Q)^{A_1}}\biggr)^{2L}
	<
	\sum_{\chi\in\mathcal X_2}|S_1(\chi)|^{2L}.
	\]
	Applying this in \eqref{E5a}, we obtain
	\[
	|\mathcal X_2|
	\frac{|\mathcal P_1|^{2L}}{(\log Q)^{2A_1L}}
	\ll_{K,\varepsilon}
	Y_1^L|\mathcal P_1|^L
	+
	|\mathcal X_2|Y_1^{L(1-\eta_2)}Q^{\varepsilon\eta_2}
	|\mathcal P_1|^L.
	\]
		Since \(|\mathcal P_1|\gg_\nu Y_1   Q ^{-\nu}\) by \eqref{eq:P1-lower-for-energy} and  $Y_1\geq Q^{2\varepsilon}$. Choosing \(\nu>0\) sufficiently small, say   $\nu<\frac{\varepsilon\eta_2 (2L-1)}{L }$, then the second term on the
	right-hand side of \eqref{E5a} is absorbed into the left-hand side. Thus 
	\begin{equation}\label{E5}
			|\mathcal X_2|
		\ll_{K,\varepsilon}
		\biggl(\frac{Y_1}{|\mathcal P_1|}\biggr)^L
		(\log  Q)^{2A_1L}
		\ll_{K,\varepsilon,\nu}
		Q^{ \nu L }(\log  Q)^{ 2A_1 L}.
	\end{equation}
	
Finally, using the trivial bounds
$
|S_1(\chi)|\le |\mathcal P_1|, 
|S_2(\chi)|\le |\mathcal P_2|,
$ together with \eqref{E4R} and \eqref{E5},
	we have
	\[
 \frac1{|G|}
\sum_{\chi\in\mathcal X_2}
|S_1(\chi)|^2|S_2(\chi)||R(\overline\chi)|  
\ll_{K,\varepsilon}
\frac{XQ^{-\eta_1}}{|G|}
|\mathcal P_1|^2|\mathcal P_2|
|\mathcal X_2|   \ll_{K,\varepsilon,\nu}
\frac{X}{|G|}
|\mathcal P_1|^2|\mathcal P_2| Q^{\nu L-\eta_1 } (\log  Q)^{ 2A_1 L}
.
	\]
	Since
	$
	|\mathcal P_2|\gg \frac{X}{\log z}\gg \frac{X}{\log Q},
	$
	  taking $\nu>0$ sufficiently small so that $\nu L<\min\{ \varepsilon\eta_2(2L-1),\eta_1\}$,
 the above is 
	\begin{equation}\label{E7}
		\frac1{|G|}
		\sum_{\chi\in\mathcal X_2}
		|S_1(\chi)|^2|S_2(\chi)||R(\overline\chi)| 
		\ll_{K,\varepsilon,\nu} 
		\frac{|\mathcal P_1|^2|\mathcal P_2|^2  }{|G|}Q^{\nu L-\eta_1 } (\log  Q)^{ 2A_1 L+1} 
		\ll_{K,\varepsilon} \frac{|\mathcal P_1|^2|\mathcal P_2|^2  }{|G|(\log Q )}.
	\end{equation}

	Combining \eqref{E3} and \eqref{E7} gives the claim
	\eqref{eq:energy-nonprincipal-ray}, completing the proof.
\end{proof}
\begin{proposition}
	\label{prop:E2-energy-increment}
	Assume the setup of Proposition~\ref{prop:criteria-E2-E3}(b) and    let
	\(A',B'\subseteq G\) be the sets   there. Assume additionally that $X\ge Q^{\max(1,3\alpha,4\alpha_0)+\kappa}$. Suppose  
	Lemma~\ref{lem:small-doubling-structure}(b) holds, and let \(H\) be the
	stabilizer of \(A'\cdot B'\). Assume that
	\[
	[G:H]=3k+2,\qquad k\in\{0,1,2\},
	\]
	and that
	\[
	A', B' \subset \bigcup_{i=1}^{k+1}a_iH, \qquad
	A'\cdot B'
	=
	\bigcup_{j=1}^{2k+1}b_jH
	=
	\biggl(\bigcup_{i=1}^{k+1}a_iH\biggr)^2.
	\]
	Then
	\[
	|E_2(X;\mathfrak q)\setminus (A'\cdot B')|
	\ge
	\bigl(\gamma_k-O(\varepsilon^{1/2})-o_K(1)\bigr)|G|,
	\]
	where
	\[
	\gamma_0=\frac13,\quad
	\gamma_1=\frac7{45},\quad
	\gamma_2=\frac19.
	\]
\end{proposition}

\begin{proof}
	Let
	\[
	S:=\bigcup_{i=1}^{k+1}a_iH,
	\qquad
	Y:=[G:H]=3k+2,\qquad B: =\frac{\log X}{\log Q} .
	\]
	We follow the discussion in the proof of
	Theorem~\ref{intro:three-prime-ray} in
	Section~\ref{three}. There are two
	cases.
	
	\medskip
	\noindent
	\emph{The non-complementary case.}
	Suppose that \(S\) and \(S^2=A'\cdot B'\) are not complementary subsets of
	\(G/H\). 
	By  \eqref{8.7},
	   there exist a coset
	  \(a_0H\not\subset S\) and
	  \(\beta\in[B/3,2B/3]\) such that
  \begin{equation}\label{10.4}
  	\sum_{\substack{
  			Q^{\beta-\varepsilon}<N\mathfrak p\le Q^\beta\\
  			[\mathfrak p]\in a_0H}}
  	\frac1{N\mathfrak p}
  	\gg_{K,\varepsilon}1.
  \end{equation}

	\medskip
	\noindent
	\emph{The complementary case.}
Suppose that \(S\) and \(S^2\) are complementary subsets of \(G/H\).

	If \(k=1\) or \(k=2\), then \(G/H\) has order \(5\) or \(8\). In these
	cases, by Corollary~\ref{cor:dyadic-prime-escape} or
	\eqref{8.8}, there exist a coset
	\(a_0H\not\subset S\) and
	$
	\beta\in(2\varepsilon,4\alpha_0]
	$
	such that
	\begin{equation}\label{10.5}
			\sum_{\substack{
				Q^{\beta-\varepsilon}<N\mathfrak p\le Q^\beta\\
				[\mathfrak p]\in a_0H}}
		\frac1{N\mathfrak p}
		\gg_{K,\varepsilon}1.
	\end{equation}
In either case \eqref{10.4} or \eqref{10.5}, \(Q^\beta\le X/2\). A dyadic decomposition of the interval
\([Q^{\beta-\varepsilon},Q^\beta]\) therefore gives  
\(Y_1\) with \(2Y_1\le X\) such that
	\[
	 \#
	\{\mathfrak p:Y_1<N\mathfrak p\le2Y_1,\ [\mathfrak p]\in a_0H\} \gg_{K,\varepsilon}\frac{Y_1}{\log Q}.
	\]

 We now consider the   complementary case  with \(k=0\), following   Case~2 in the proof of
 Theorem~\ref{intro:three-prime-ray}. In this situation,
 \([G:H]=2\) and \(S\ne H\). Let \(\psi\) be the non-principal quadratic
 character of \(G\) with kernel \(H\). Then, 
	$
	\psi(\mathfrak p)=1
 \iff
	[\mathfrak p]\in H.
	$
	
	Choose a fixed number \(B_0<B\) with \(B_0>\max\{2\alpha+\alpha_0,2\alpha_0\}  \), for example \(B_0=B-\kappa/2\).   Applying
	Lemma~\ref{lem:quadratic-prime-dichotomy} to \(\psi\), one of the following
	two alternatives holds for $\varepsilon>0$  sufficiently small.

	\noindent
	\emph{Case (i).}
	Suppose that 
	Lemma~\ref{lem:quadratic-prime-dichotomy}(i) holds. Then there exists $	\beta_*\in
	 (\max\{2\alpha,B_0/2\}, B_0-\alpha_0 )$ such that 
	\[
	\sum_{\substack{
			Q^{\alpha+\varepsilon}<N\mathfrak p\le Q^{\beta_*}\\
			[\mathfrak p]\in H}}
	\frac1{N\mathfrak p}
	\ge c_*,  
	\]
	A dyadic decomposition of the interval $[Q^{\alpha+\varepsilon} , Q^{\beta_*} ]$ then gives \(Y_1\) with \(2Y_1\le X\) such that
	\[
	\#\{\mathfrak p:Y_1<N\mathfrak p\le2Y_1,\ [\mathfrak p]\in H\}
	\gg_{K,\varepsilon}
	\frac{Y_1}{\log Q}.
	\]

	\noindent
	\emph{Case (ii).}
	Suppose that    
	Lemma~\ref{lem:quadratic-prime-dichotomy}(ii) holds. Then there exists
	$
	M\in[Q^{\beta_*},Q^{B_0}]
	$
	such that
	\[
	\#\{\mathfrak p:M<N\mathfrak p\le2M,\ \psi(\mathfrak p)=1\}
	\gg_{K,\varepsilon}
	M\rho_KL(1,\psi)\mathcal V_\psi(Q).
	\]
Since \(\psi(\mathfrak p)=1\) if and only if
\([\mathfrak p]\in H\), and
$
	L(1,\psi)\gg_{K,\delta}Q^{-\delta}$, $
	\mathcal V_\psi(Q)\gg_{K,\delta}Q^{-\delta}$,
	it follows that, for every fixed \(\delta>0\),
	\[
	\#\{\mathfrak p:M<N\mathfrak p\le2M,\ [\mathfrak p]\in H\}
	\gg_{K,\varepsilon,\delta}
	M Q^{-\delta}.
	\]
	Since \(M\le Q^{B_0}\), we also have \(2M\le X\).
		
	Therefore, in both the complementary and non-complementary cases, we obtain a coset
	\(a_0H\not\subset S\) and a dyadic interval
	$
	Y_1<N\mathfrak p\le2Y_1$ with $2Y_1\le X$
	such that writing
	\[
	\mathcal P_1
	:=
	\{\mathfrak p:Y_1<N\mathfrak p\le2Y_1,\ [\mathfrak p]\in a_0H\},
	\]
	we have, for every fixed \(\delta>0\),
\begin{equation}\label{P1}
	|\mathcal P_1|\gg_{K,\varepsilon,\delta}Y_1Q^{-\delta}.
\end{equation}

  Since \(a_0H\not\subset S\), the product
$
	(S\cup a_0H)S
$
	contains strictly more \(H\)-cosets than \(S^2\). Hence there is
	\(j_0\in\{1,\dots,k+1\}\) such that
	$
		a_0a_{j_0}H\not\subset A'\cdot B'.
	$
	
	We now lower-bound the number of primes in the coset \(a_{j_0}H\). For a coset
	\(bH\), write
	\[
	\pi(T;bH)
	:=
	\#\bigl\{\mathfrak p\subset \mathcal{O}_K:\mathfrak p\nmid\mathfrak q,\ N\mathfrak p\le T,\ 
	[\mathfrak p]\in bH\bigr\}.
	\]

	If \(bH\not\subset S\), then
	\(A'\cap B'\cap bH=\varnothing\). Hence, by Proposition~\ref{prop:criteria-E2-E3}(b.vi),
	\[
	\pi(X;bH)\ll  \frac{\varepsilon X}{Y\log X}.
	\]
	Summing over the \(2k+1\) cosets outside \(S\), we have
	\begin{equation}\label{outside-prime-small}
		\sum_{bH\not\subset S}\pi(X;bH)
		\ll
	 \frac{	\varepsilon X}{\log X}.
	\end{equation}
	
	On the other hand, for any coset \(bH\), we  apply Lemma \ref{rough_ideal_lemma} to  \(bH\). The contribution to  \(\pi(X;bH)\) from prime ideals with
	\(N\mathfrak p\le X^{1/3}\) is
	\(o_K(X/\log X)\) by the prime ideal theorem.  And the number of remaining prime ideals  
	is  bounded by \(\mathcal{N}_{1/3}(X;b,H)\), hence Lemma \ref{rough_ideal_lemma} gives
	\begin{equation}\label{coset-prime-upper}
		\pi(X;bH)
		\le
		\Bigl(\frac{2}{Y\vartheta_0}+o_K(1)\Bigr)\frac{X}{\log X}.
	\end{equation}

	Moreover, the prime ideal theorem gives
	\[
	\sum_{bH\in G/H}\pi(X;bH)
	=
	(1+o_K(1))\frac{X}{\log X}.
	\]
	Subtracting the contributions of the \(k\) cosets
	\(a_iH\ne a_{j_0}H\) contained in \(S\) using
	\eqref{coset-prime-upper}, together with the contribution from the
	cosets outside \(S\) using \eqref{outside-prime-small}, we obtain
	\[
	\pi(X;a_{j_0}H)
	\ge
	\Bigl(
	1-\frac{2k}{Y\vartheta_0}
	-O(\varepsilon)-o_K(1)
	\Bigr)\frac{X}{\log X}.
	\]
	  Since   $X\ge Q^{\max(1,3\alpha,4\alpha_0)+\kappa}$, by \eqref{8.2}, 
	$
	\vartheta_0\ge \frac34-O(\varepsilon^{1/2}).
	$
	Therefore, 
	\[
	1-\frac{2k}{Y\vartheta_0}
	\ge
	1-\frac{8k}{3(3k+2)}-O(\varepsilon^{1/2}).
	\]
	Taking
	$
	c_k
	=
	1-\frac{8k}{3(3k+2)},
	$
	then we have 
	\begin{equation}\label{10.9}
		\pi(X;a_{j_0}H)
		\ge
		\bigl(c_k-O(\varepsilon^{1/2})-o_K(1)\bigr)\frac{X}{\log X},
	\end{equation}
	where
	\[
	c_0=1,\quad c_1=\frac7{15},\quad c_2=\frac13.
	\]

	Let 
	\[
	z:=X^{2/3},
	\qquad
	\mathcal P_2
	:=
	\{\mathfrak p:z<N\mathfrak p\le X,\ [\mathfrak p]\in a_{j_0}H\}.
	\]
	Since the number of prime ideals of norm at most \(z\) is
	\(o_K(X/\log X)\), it follows from \eqref{10.9} that
	\begin{equation}\label{10.10}
		|\mathcal P_2|
		\ge
		\Bigl(c_k-O(\varepsilon^{1/2})-o_K(1)\Bigr)\frac{X}{\log X}.
	\end{equation}
	By \eqref{P1}, \eqref{10.10}, and \(\log z=\frac23\log X\),
	Lemma~\ref{lem:ray-prime-energy} applies to $\mathcal{P}_1$ and $\mathcal{P}_2$  with
	$
	C_k=\frac3{c_k}+O(\varepsilon^{1/2})+o_K(1).
	$ 
	
	For \(u\in G\), define
	\[
	r(u)
	:=
	\#\bigl\{(\mathfrak p_1,\mathfrak p_2)\in\mathcal P_1\times\mathcal P_2:
	[\mathfrak p_1\mathfrak p_2]=u\bigr\}.
	\]
	Since \(	a_0a_{j_0}H\not\subset A'\cdot B'\), if \(r(u)>0\), then
	\[
	u\in E_2(X;\mathfrak q)\cap a_0a_{j_0}H
	\subseteq
	E_2(X;\mathfrak q)\setminus(A'\cdot B').
	\]
	Therefore, by Cauchy--Schwarz,
	\begin{equation}\label{10.11}
			|\mathcal P_1|^2|\mathcal P_2|^2
		=
		\Bigl(\sum_{u\in G}r(u)\bigr)^2
		\le
		|E_2(X;\mathfrak q)\cap a_0a_{j_0}H|
		\sum_{u\in G}r(u)^2.
	\end{equation}
	The second factor \(	\sum_{u\in G}r(u)^2\) is the multiplicative energy
	\[
	\mathcal E(\mathcal P_1,\mathcal P_2)
	:=
	\#\{(\mathfrak p_1,\mathfrak p_1',\mathfrak p_2,\mathfrak p_2') \in \mathcal{P}_1^2\times \mathcal{P}_2^2:
	[\mathfrak p_1\mathfrak p_2]=
	[\mathfrak p_1'\mathfrak p_2'] \}.
	\]
	And by Lemma~\ref{lem:ray-prime-energy},
	\begin{equation}\label{10.12}
			\mathcal E(\mathcal P_1,\mathcal P_2)
		\le
		\frac{C_k+O(\varepsilon^{1/2})+o_K(1)}{|G|}
		|\mathcal P_1|^2|\mathcal P_2|^2.
	\end{equation}
	Combining \eqref{10.11} and \eqref{10.12}, we have 
	\[
	|E_2(X;\mathfrak q)\setminus(A'\cdot B')|\ge|E_2(X;\mathfrak q)\cap a_0a_{j_0}H|
	\ge
	\Bigl(\frac1{C_k}-O(\varepsilon^{1/2})-o_K(1)\Bigr)|G|.
	\]
	Taking \(\gamma_k:=\frac{c_k}{3},\) we have 
	\[
	|E_2(X;\mathfrak q)\setminus(A'\cdot B')|
	\ge
	\left(\gamma_k-O(\varepsilon^{1/2})-o_K(1)\right)|G|, 
	\]
	where
	\[
	\gamma_0=\frac13,\quad
	\gamma_1=\frac7{45},\quad
	\gamma_2=\frac19.
	\]
\end{proof}

\begin{proof}[Proof of Theorem~\ref{intro:two-prime-ray}]
	Let \(\kappa>0\) be fixed, and let
	\(\varepsilon>0\) be chosen sufficiently small.  For $Q$ sufficiently large in terms of $\varepsilon$, we want to show that 
	\begin{enumerate}
		\item[(i)] If
		$
		X\ge Q^{\max(1,3\alpha,4\alpha_0)+\kappa},
		$
		then
		$
		|E_2(X;\mathfrak q)|
		\ge
		\bigl(\frac23-\varepsilon\bigr)|G|.
		$
		
		\item[(ii)] If
		$
		X\ge Q^{\max(1,4\alpha,4\alpha_0)+\kappa},
		$
		then
		$
		|E_2(X;\mathfrak q)|
		\ge
		\bigl(\frac{11}{16}-\varepsilon\bigr)|G|.
		$
	\end{enumerate}
It suffices to take $
X= Q^{\max(1,3\alpha,4\alpha_0)+\kappa} 
$   in (i) and 	$
X= Q^{\max(1,4\alpha,4\alpha_0)+\kappa} 
$ in (ii). 
	Let
	$
	B:=\frac{\log X}{\log Q}.
	$
	We apply Proposition~\ref{prop:criteria-E2-E3}.
	
	 If
	Proposition~\ref{prop:criteria-E2-E3}(a) holds, then
	\[
	|E_2(X;\mathfrak q)|
	\ge
	(\vartheta-3\varepsilon)|G|.
	\]
	In case {\rm (i)}, the assumption \(B\ge3\alpha+\kappa\) implies that
	\[
	\vartheta
	\ge
	\frac23+\frac{\kappa}{9\alpha+3\kappa} -\varepsilon.
	\]	
 Hence, if \(\varepsilon>0\) is sufficiently small that
$
 3\varepsilon
 <
 \frac{\kappa}{9\alpha+3\kappa},
 $
 then
	\[
	|E_2(X;\mathfrak q)|
	\ge
	\Bigl(\frac23-\varepsilon\Bigr)|G|.
	\]
	Similarly, in case {\rm (ii)}, the assumption
	\(B\ge4\alpha+\kappa\) gives that for sufficiently small \(\varepsilon\),
	\[
	|E_2(X;\mathfrak q)|
	\ge
	\Bigl(\frac34-\varepsilon\Bigr)|G|.
	\] 
This is stronger than the desired
	\((11/16-\varepsilon)|G|\) bound.
	
	We may therefore assume that Proposition~\ref{prop:criteria-E2-E3}(b)
	holds. Let \(A',B'\subseteq G\) be the sets  in Proposition~\ref{prop:criteria-E2-E3}(b).
	By the same argument as in the proof of Theorem~\ref{intro:three-prime-ray},
	using Proposition~\ref{prop:criteria-E2-E3}(b.v), the \(k+1\) cosets of
	\(H\) met by \(A'\) and \(B'\) are the same. Thus, for
	\(S=\bigcup_{i=1}^{k+1}a_iH\), we have
	$
	A',B'\subseteq S$, $A'B'=S^2.
	$
	And  Proposition~\ref{prop:E2-energy-increment}
	applies. 
	
	First we prove (i). We apply Lemma~\ref{lem:small-doubling-structure} with
	\[
	\alpha_*=\frac13- \varepsilon^{1/2},
	\qquad
	\alpha_*'=\frac38-\varepsilon^{1/2},
	\qquad
	\beta_*=\frac23-5\varepsilon^{1/2}.
	\]
	These parameters satisfy the conditions of Lemma \ref{lem:small-doubling-structure} for \(\varepsilon>0\)  sufficiently
	small. The lower bounds for \(|A'|\), \(|B'|\), and the upper bound for indices of cosets follow from Proposition~\ref{prop:criteria-E2-E3}(b.i),(b.v). 
	
	If Lemma~\ref{lem:small-doubling-structure}(a) holds, then
	\[
	|A'\cdot B'|\ge\Bigl(\frac23-5\varepsilon^{1/2}\Bigr)|G|.
	\]
	By Proposition~\ref{prop:criteria-E2-E3}(b.ii),
	\[
	|E_2(X;\mathfrak q)|
	\ge
	|(A'\cdot B')\cap E_2(X;\mathfrak q)|
	\ge
	|A'\cdot B'|-\varepsilon |G|	\ge
	\Bigl(\frac23-6\varepsilon^{1/2} \Bigr)|G|.
	\]

	We may therefore assume that Lemma~\ref{lem:small-doubling-structure}(b)
	holds. Then, for some \(k\in\{0,1,2\}\), the set \(A'\cdot B'\) is the
	union of \(2k+1\) cosets of a subgroup \(H\) of index \(3k+2\). By
	Proposition~\ref{prop:criteria-E2-E3}(b.ii) again,
	\[
	|(A'\cdot B')\cap E_2(X;\mathfrak q)|
	\ge
	|A'\cdot B'|-\varepsilon |G|
	\ge
	\Bigl(\frac{2k+1}{3k+2}-\varepsilon\Bigr)|G|.
	\]
	By Proposition~\ref{prop:E2-energy-increment},
	\[
	|E_2(X;\mathfrak q)\setminus(A'\cdot B')|
	\ge
	\bigl(\gamma_k-O(\varepsilon^{1/2})- o_{K }(1)\bigr)|G|.
	\]
	Hence, 
	\[
	|E_2(X;\mathfrak q)|
	\ge
	\Bigl(
	\frac{2k+1}{3k+2}+\gamma_k-O(\varepsilon^{1/2})- o_{K }(1)
	\Bigr)|G|.
	\]
Since 
$
\frac{2k+1}{3k+2}+\gamma_k>\frac23$,  for $ k=0,1,2 $,
renaming \(\varepsilon\) if necessary, the desired bound follows for
\(\varepsilon>0\) sufficiently small and \(Q\) sufficiently large.

	For (ii), we repeat the same argument by applying Lemma \ref{lem:small-doubling-structure} with
	\[
	\alpha_*=\alpha_*'=\frac38-  \varepsilon^{1/2} ,
	\qquad
	\beta_*=\frac{11}{16}-5\varepsilon^{1/2}.
	\]
	The assumptions \(B\ge4\alpha+\kappa\) and \(B\ge4\alpha_0+\kappa\) ensure
	that Proposition~\ref{prop:criteria-E2-E3}(b.i),(b.v) give the required condition for the bounds of 
	size and coset indices in Lemma \ref{lem:small-doubling-structure}. If Lemma~\ref{lem:small-doubling-structure}(a) holds,  then
	\[
	|E_2(X;\mathfrak q)|
	\ge
	\Bigl(\frac{11}{16}-6 \varepsilon^{1/2} \Bigr)|G|.
	\]
	Otherwise, Lemma~\ref{lem:small-doubling-structure}(b) holds, and the same  argument above gives
	\[
	|E_2(X;\mathfrak q)|
	\ge
	\Bigl(
	\frac{2k+1}{3k+2}+\gamma_k-O(\varepsilon^{1/2})- o_{K }(1)
	\Bigr)|G|.
	\]
Since
$
\frac{2k+1}{3k+2}+\gamma_k>\frac{11}{16}$,  for $ k=0,1,2 $, 
renaming \(\varepsilon\) if necessary completes the proof. 
\end{proof}

\bibliographystyle{alpha}

\begin{thebibliography}{99}
	
	\bibitem{harmonic}
	S. Axler, P. Bourdon, and R. Wade,
	\emph{Harmonic Function Theory},
	Springer, New York, 2013.
	
	\bibitem{cubefree_totally_real}
	O. Balkanova, D. Frolenkov, and H. Wu,
	On Weyl's subconvex bound for cube-free Hecke characters: totally real case,
	arXiv:2108.12283.
	
	\bibitem{ramanujan}
	V. Blomer and F. Brumley,
	On the Ramanujan conjecture over number fields,
	\emph{Ann. of Math. (2)} \textbf{174} (2011), no.~1, 581--605.
	
	\bibitem{sieve}
	M. D. Coleman,
	The Rosser--Iwaniec sieve in number fields, with an application,
	\emph{Acta Arith.} \textbf{65} (1993), no.~1, 53--83.
	
	\bibitem{ideal_product}
	J.-M. Deshouillers, S. Gun, O. Ramar\'e, and J. Sivaraman,
	Representing ideal classes of ray class groups by products of prime ideals of small size,
	\emph{Math. Z.} \textbf{310} (2025), no.~4, Paper No.~68.
	\bibitem{erdos}
	P. Erd\H{o}s, A. M. Odlyzko, and A. S\'ark\"ozy,
	On the residues of products of prime numbers,
	\emph{Period. Math. Hungar.} \textbf{18} (1987), no.~3, 229--239.
	
	\bibitem{opera}
	J. Friedlander and H. Iwaniec,
	\emph{Opera de Cribro},
	American Mathematical Society Colloquium Publications, vol.~57,
	American Mathematical Society, Providence, RI, 2010.
	\bibitem{Green}
	B. Green,
	Roth's theorem in the primes,
	\emph{Ann. of Math. (2)} \textbf{161} (2005), no.~3, 1609--1636.
	
	\bibitem{GreenTao}
	B. Green and T. Tao,
	Restriction theory of the Selberg sieve, with applications,
	\emph{J. Théor. Nombres Bordeaux} \textbf{18} (2006), no.~1, 147--182.
	
	\bibitem{counting_ideal}
	S. Gun, O. Ramar\'e, and J. Sivaraman,
	Counting ideals in ray classes,
	\emph{J. Number Theory} \textbf{243} (2023), 13--37.
	
	\bibitem{heathbrown}
	D. R. Heath-Brown,
	Zero-free regions for Dirichlet \(L\)-functions, and the least prime in an arithmetic progression,
	\emph{Proc. London Math. Soc. (3)} \textbf{64} (1992), no.~2, 265--338.
	
	\bibitem{hecke}
	E.~Hecke,
	{Eine neue Art von Zetafunktionen und ihre Beziehungen zur Verteilung der Primzahlen},
	\textit{Math. Z.} \textbf{1} (1918), no.~4, 357--376.
 
	
	\bibitem{book_analytic}
	H. Iwaniec and E. Kowalski,
	\emph{Analytic Number Theory},
	American Mathematical Society Colloquium Publications, vol.~53,
	American Mathematical Society, Providence, RI, 2004.
	
	\bibitem{lang}
	S. Lang,
	\emph{Algebraic Number Theory},
	Graduate Texts in Mathematics, vol.~110,
	Springer-Verlag, New York, 1994.
	
	\bibitem{counting}
	D. Masser and J. Vaaler,
	Counting algebraic numbers with large height II,
	\emph{Trans. Amer. Math. Soc.} \textbf{359} (2007), no.~1, 427--445.
	
	\bibitem{matomaki}
	K. Matom\"aki and J. Ter\"av\"ainen,
	Products of primes in arithmetic progressions,
	\emph{J. Reine Angew. Math.} \textbf{2024} (2024), no.~808, 109--149.
	
	
	\bibitem{Mitsui}
	T. Mitsui,
	Generalized prime number theorem,
	\emph{Japanese J. Math.} \textbf{26} (1956), 1--42.
	
	\bibitem{norm}
	M. R. Murty and J. Van Order,
	Counting integral ideals in a number field,
	\emph{Expo. Math.} \textbf{25} (2007), no.~1, 53--66.
	
	
	
	\bibitem{neukirch}
	J. Neukirch,
	\emph{Algebraic Number Theory},
	Grundlehren der Mathematischen Wissenschaften, vol.~322,
	Springer-Verlag, Berlin, 1999.
	

	
	\bibitem{rademacherPL}
	H. Rademacher,
	On the Phragm\'en--Lindel\"of theorem and some applications,
	\emph{Math. Z.} \textbf{72} (1959), 192--204.
	
		\bibitem{notes}
	A. Raghuram,
	Notes on the arithmetic of Hecke \(L\)-functions,
	\emph{Proc. Indian Acad. Sci. Math. Sci.} \textbf{132} (2022), no.~2, Paper No.~71.
	
	\bibitem{merten}
	M. Rosen,
	A generalization of Mertens' theorem,
	\emph{J. Ramanujan Math. Soc.} \textbf{14} (1999), no.~1, 1--19.
	
	\bibitem{sawin}
	W. Sawin,
	Square-root cancellation for sums of factorization functions over squarefree progressions in \(\mathbb{F}_q[t]\),
	\emph{Acta Math.} \textbf{233} (2024), no.~2, 285--418.
	
	\bibitem{taovu}
	T. Tao and V. H. Vu,
	\emph{Additive Combinatorics},
	Cambridge Studies in Advanced Mathematics, vol.~105,
	Cambridge University Press, Cambridge, 2006.
	
	\bibitem{counting_Martin}
	M. Widmer,
	Counting primitive points of bounded height,
	\emph{Trans. Amer. Math. Soc.} \textbf{362} (2010), no.~9, 4793--4829.
	
	\bibitem{wu}
	H. Wu,
	Burgess-like subconvexity for \(\mathrm{GL}_1\),
	\emph{Compos. Math.} \textbf{155} (2019), no.~8, 1457--1499.
	
	\bibitem{xie}
	L. Xie,
	On a conjecture of Erd\H{o}s over function fields,
	arXiv:2510.17612;
	to appear in \emph{Math. Z.}
	
	\bibitem{zaman}
	A. Zaman,
	Explicit estimates for the zeros of Hecke \(L\)-functions,
	\emph{J. Number Theory} \textbf{162} (2016), 312--375.
	
\end{thebibliography}

\end{document}